\documentclass[10pt]{amsart} 

\usepackage{amsmath, amssymb, mathrsfs}

\usepackage[mathscr]{euscript} 

\newlength{\mylength}
\usepackage{enumitem}
\setenumerate{listparindent=\parindent, itemsep=0cm, parsep=\mylength, topsep=0cm}
\setitemize{listparindent=\parindent, itemsep=0cm, parsep=\mylength, topsep=0cm}

\usepackage[breaklinks=true]{hyperref} 
\usepackage{comment} 

\usepackage{url}

\usepackage{tikz-cd}

\usepackage{amsthm}

\makeatletter
\renewenvironment{proof}{\par
	\pushQED{\qed}%
	\normalfont \topsep6\p@\@plus6\p@\relax
	\noindent\emph{Proof.} 
	\ignorespaces
}{%
\popQED\endtrivlist\@endpefalse
}
\makeatother

\makeatletter
\newenvironment{proofofclaim}{\par
	\pushQED{\qed}%
	\normalfont \topsep6\p@\@plus6\p@\relax
	\noindent\emph{Proof of claim.} 
	\ignorespaces
}{%
\popQED\endtrivlist\@endpefalse
}
\makeatother

\makeatletter
\newenvironment{myproof}[1]{\par
	\pushQED{\qed}%
	\normalfont \topsep6\p@\@plus6\p@\relax
	\noindent\emph{#1.} 
	\ignorespaces
}{%
\popQED\endtrivlist\@endpefalse
}
\makeatother

\newtheoremstyle{mythm}% name of the style to be used
{\mylength}% measure of space to leave above the theorem. E.g.: 3pt
{0pt}% measure of space to leave below the theorem. E.g.: 3pt
{\itshape}% name of font to use in the body of the theorem
{0pt}% measure of space to indent
{\bfseries}% name of head font
{.\ }% punctuation between head and body
{ }% space after theorem head; " " = normal interword space
{\thmname{#1}\thmnumber{ #2}\thmnote{ (#3)}}

\theoremstyle{mythm} 
\newtheorem*{claim}{Claim}
\newtheorem*{thm}{Theorem} 
\newtheorem*{prop}{Proposition} 
\newtheorem*{lem}{Lemma}
\newtheorem*{cor}{Corollary}

\theoremstyle{definition}
\newtheorem*{defn}{Definition} 

\theoremstyle{remark} 
\newtheorem*{rmk}{Remark}

\newcommand{\nc}{\newcommand} 
\nc{\on}{\operatorname}
\nc{\rnc}{\renewcommand} 

\nc{\Vect}{\on{Vect}}
\nc{\QCoh}{\on{QCoh}}

\nc{\wt}{\widetilde}
\nc{\wh}{\widehat} 
\nc{\ol}{\overline} 

\nc{\disk}{\mathbb{D}}

\nc{\BN}{\mathbb{N}}
\nc{\BZ}{\mathbb{Z}}
\nc{\BQ}{\mathbb{Q}}
\nc{\BR}{\mathbb{R}}
\nc{\BC}{\mathbb{C}}
\nc{\BA}{\mathbb{A}}
\nc{\BP}{\mathbb{P}}
\nc{\BE}{\mathbb{E}}
\nc{\BG}{\mathbb{G}}
\nc{\qbar}{\ol{\mathbb{Q}}_\ell}
\nc{\ul}{\underline}

\nc{\fset}{\on{fSet}}
\nc{\fsetsurj}{\on{fSet}^{\on{surj}}}
\nc{\fsetsurjne}{\on{fSet}^{\on{surj}}_{\on{n.e.}}}
\nc{\fsetne}{\on{fSet}_{\on{n.e.}}}
\nc{\fsetsurjneop}{\on{fSet}^{\on{surj}, \op}_{\on{n.e.}}}
\nc{\fsetsurjop}{\on{fSet}^{\on{surj}, \op}}

\nc{\raninf}{\on{Ran}_{\on{inf}}}
\nc{\ranf}[1]{\on{Ran}_{\la #1 \ra}}

\nc{\holim}{\displaystyle\sideset{}{^\cdot}\lim}
\nc{\nlim}[1]{\displaystyle\sideset{}{^{#1}}\lim}

\nc{\la}{\langle}
\nc{\ra}{\rangle} 

\nc{\mb}{\mathbf}
\nc{\mf}{\mathfrak}
\nc{\mc}{\mathscr}

\nc{\ira}{\hookrightarrow}
\nc{\hra}{\hookrightarrow}
\nc{\sra}{\twoheadrightarrow} 

\renewcommand{\setminus}{\smallsetminus}

\nc{\Ext}{\on{Ext}}
\nc{\Spec}{\on{Spec}}
\rnc{\Im}{\on{Im}}
\rnc{\Re}{\on{Re}}
\nc{\Id}{\on{Id}}
\nc{\id}{\on{id}}
\nc{\Hom}{\on{Hom}}
\nc{\op}{{\on{op}}}

\nc{\Supp}{\on{Supp}}

\nc{\Ann}{\on{Ann}}
\nc{\fun}{\on{Fun}}

\nc{\oh}{\mc{O}}
\nc{\D}{\mc{D}}
\nc{\Gr}{\on{Gr}}
\nc{\gm}{\mathbb{G}_m}
\nc{\ga}{\mathbb{G}_a}
\nc{\Pic}{\mathbf{Pic}}
\nc{\pr}{\on{pr}}
\nc{\Ran}{\on{Ran}}
\nc{\ran}{\on{Ran}}
\nc{\dr}{{\on{dR}}}

\DeclareMathOperator*\colim{colim}

\newcommand{\bb}[1]{[\![ #1 ]\!]}

\nc{\End}{\on{End}}

\newenvironment{cd}{\begin{equation*}\begin{tikzcd}}{\end{tikzcd}\end{equation*}\ignorespacesafterend}
\nc{\e}[1]{\begin{align*} #1 \end{align*}}

\usepackage[margin=1.5in]{geometry}

\title[$n$-excisive functors, canonical connections, and line bundles on $\Ran(X)$]{$n$-excisive functors, canonical connections, and line bundles on the Ran space} 
\author{James Tao}
\date{June 18, 2019} 

\begin{document}
	
\maketitle

\vspace{-\mylength} 

\begin{abstract}
	Let $X$ be a smooth algebraic variety over $k$. We prove that any flat quasicoherent sheaf on $\Ran(X)$ canonically acquires a $\D$-module structure. In addition, we prove that, if the geometric fiber $X_{\ol{k}}$ is connected and admits a smooth compactification, then any line bundle on $S \times \Ran(X)$ is pulled back from $S$, for any locally Noetherian $k$-scheme $S$. Both theorems rely on a family of results which state that the (partial) limit of an $n$-excisive functor defined on the category of pointed finite sets is trivial. 
\end{abstract}

\vspace{0.6in}

\begin{center}
	\includegraphics[scale=0.22]{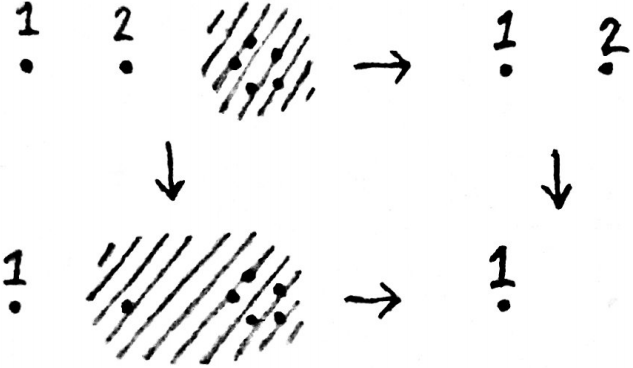}
\end{center}

\newpage 

\setlength{\parskip}{0.22cm}
\tableofcontents

\newpage

\setlength{\parskip}{0.25cm}

\section{Introduction} 

\subsection{The Ran space} \label{ran-top}

Given a topological space $X$, its Ran space $\Ran(X)$ is the set of nonempty finite subsets of $X$, endowed with a suitable topology. This space plays a role in the study of configuration spaces and factorization structures associated to $X$. It has a semigroup operation given by `union of finite subsets,' which is idempotent because $S \cup S = S$ for any subset $S \subset X$. This semigroup operation can be used to show that $\Ran(X)$ is weakly contractible, see~\cite[3.4.1]{bd}. 

Beilinson and Drinfeld introduced an analogous notion in algebraic geometry: for any $k$-scheme $X$, the Ran prestack of $X$, denoted $\Ran(X)$, is the Set-valued presheaf on affine schemes over $k$ which sends a test scheme $T$ to the set of nonempty finite subsets of $\on{Maps}(T, X)$. Hence, for every nonempty finite set $I$, there is a natural map $X^I \to \Ran(X)$ sending an $I$-tuple in $\on{Maps}(T, X)$ to the corresponding subset of $\on{Maps}(T, X)$.  (See Section~\ref{ran-def} for more details.) 

The prestack $\Ran(X)$ is the natural base over which various algebro-geometric objects with factorization structure are defined. As such, it plays a role in relating geometric structures (affine Grassmannians and loop groups) to algebraic structures (chiral algebras and factorization algebras) living on $X$, thereby enabling problems concerning the former to be translated into problems concerning the latter. See~\cite[0.6]{g} for more comments on why the Ran prestack is useful, and see~\cite[Sect.\ 3]{z} for a clear explanation of how to define affine Grassmannians over $\Ran(X)$ when $X$ is a curve. 

In this paper we assume that $X$ is a smooth algebraic variety over $k$, and we study quasicoherent sheaves and line bundles on $\Ran(X)$. 

\begin{rmk}
	Although Beilinson and Drinfeld, in their development of chiral algebras, mainly restricted to the case in which $X$ is a curve, it was later demonstrated by Francis and Gaitsgory that the notion of chiral (and factorization) algebras is intelligible and interesting when $X$ is any separated finite type $k$-scheme, see \cite{fg}. Likewise, we have taken pains to ensure that our results apply when $\dim X > 1$. However, we need the additional hypothesis of smoothness in order to control the behavior of quasicoherent sheaves on infinitesimal neighborhoods; this issue does not arise for $\D$-modules because they are trivial on infinitesimal neighborhoods, by definition. 
\end{rmk}

\subsection{Main results} 

We shall prove two statements about sheaves on $\Ran(X)$: 
\begin{enumerate}[label=\textbf{(\arabic*)}]
	\item[\textbf{(1)}] If $X$ is a smooth $k$-scheme, then any flat quasicoherent sheaf on $\Ran(X)$ has a unique $\D$-module structure. (Theorem~\ref{thm-dr})
	\item[\textbf{(2)}] Let $X$ be a smooth $k$-scheme satisfying property (C): the base change $X_{\ol{k}}$ admits an open embedding into a smooth proper connected $\ol{k}$-variety.\footnote{In particular, $X$ is smooth and geometrically integral over $k$.} Then, for any locally Noetherian $k$-scheme $S$, every line bundle on $S \times \Ran(X)$ is pulled back from $S$. (Theorem~\ref{main2}) \label{r2}
\end{enumerate} 
The proofs of both statements rely on some results about the vanishing of limits of $n$-excisive functors. Here is one such result: 
\begin{enumerate}
	\item[\textbf{(3)}] Let $G : \fset_* \to \mc{A}$ be a functor from pointed finite sets to some abelian category $\mc{A}$. Assume that $G$ is $n$-excisive in the sense that it sends certain strongly cocartesian $(n+1)$-hypercubes in $\fset_*$ to weakly cartesian diagrams in $\mc{A}$. Then $\lim_{I \in \fsetsurjne} G(\{*\} \sqcup I) \simeq G(\{*\})$. (Proposition~\ref{prop1}) \label{r3}
\end{enumerate}
There is an analogous result when $\mc{A}$ is replaced by a stable $(\infty, 1)$-category. (Theorem~\ref{thm2})

\subsubsection{Discussion of \textbf{(1)}} \label{discuss-1} 

Beilinson and Drinfeld defined a factorization algebra to be a quasicoherent sheaf $\mc{F}$ on $\Ran(X)$ equipped with a unital factorization structure\footnote{Tangential point: Beilinson and Drinfeld also impose a mild flatness hypothesis, which could be summarized as `flatness along the diagonals,' see \cite[Lem.\ 3.4.3(i)]{bd} and also Remark~\ref{delta-flat}.} -- this consists of the datum of `partial multiplicativity' of $\mc{F}$ with respect to the semigroup structure of $\Ran(X)$, and the datum of a `unit' $1 \in \mc{F}$. Using both of these data, Beilinson and Drinfeld proved that any factorization algebra canonically acquires a $\D$-module structure~\cite[Prop.\ 3.4.7]{bd}. What our point \textbf{(1)} shows is that, if $\mc{F}$ is flat, then neither datum is necessary. On the other hand, if $\mc{F}$ does happen to have a unital factorization structure, then comparing the proofs of \cite[Prop.\ 3.4.7]{bd} and Theorem~\ref{thm-dr} shows that our $\D$-module structure on $\mc{F}$ equals the one produced by Beilinson and Drinfeld. 

A non-flat quasicoherent sheaf on $\Ran(X)$ need not admit any $\D$-module structure; for example, take a skyscraper sheaf at a closed point of $X$, and push it forward along the map $X \to \Ran(X)$. 

An immediate corollary of point \textbf{(1)} is that every flat schematic map $Y \to \Ran(X)$ automatically descends to $\Ran(X)_{\dr}$, i.e.\ $Y$ acquires the structure of $\D$-scheme over $\Ran(X)$ (and hence its pullbacks to each $X^I$ are $\D_{X^I}$-schemes), in the sense of~\cite[2.3]{bd}. From the same ideas, it follows that any map from $\Ran(X)$ to a scheme $S$ must be constant, see Remark~\ref{constant}. 

\subsubsection{Discussion of \textbf{(2)}} \label{discuss-2}

In analogy with the weak contractibility of $\Ran(X)$ in the topological setting (see~\ref{ran-top}), there are various results in algebraic geometry which state that $\Ran(X)$ is `contractible' with respect to some cohomology theory: 
\begin{enumerate}[label=(\alph*)]
	\item Let $k$ be an algebraically closed field of characteristic zero, and let $X$ be a connected separated $k$-scheme. Then the de Rham cohomology of $\Ran(X)$ is trivial, in the sense that $H_{\bullet}(\Ran(X)) \simeq k$. \cite[Prop.\ 4.3.3]{bd} (cf.\ \cite[Thm.\ 1.6.5]{g})
	\item Let $k$ be an algebraically closed field, and let $X$ be a connected $k$-scheme. Then the $\ell$-adic cohomology of $\Ran(X)$ is trivial, in the sense that $H_\bullet(\Ran(X); \BQ_\ell) \simeq \BQ_\ell$. \cite[Thm.\ 2.4.5]{gl}\footnote{Their notation `$\Ran^u(X)$' is what we call $\Ran(X)$. For more information on comparing these notations, see~\ref{second-def} and \ref{iv}.} 
	\item Let $k$ be any field, let $X$ be a smooth $k$-scheme, and let $S$ be a $k$-scheme. Then every regular function on $S \times \Ran(X)$ is the pullback of a regular function on $S$. \cite[Prop.\ 4.3.10(1)]{z}
\end{enumerate}
When $X$ is a curve, these facts can be applied to study moduli problems which involve the datum of a dense open subset $U \subset X$ which is allowed to vary. (For example, one could study rational maps $X \dashrightarrow G$ with a domain of definition $U \subset X$ -- this example plays a central role in~\cite{g}.)  For such moduli problems, there is a semigroup action by $\Ran(X)$ where a subset $S \subset X$ acts by subtracting $S$ from $U$, and the above `contractibility' results are applied to show that taking the quotient by this action does not affect the cohomology of the moduli stack. As explained in \cite[4.3]{z}, this idea enters in proving local-to-global principles which relate the moduli stack of $G$-bundles on $X$ to the Beilinson--Drinfeld affine Grassmannian $\Gr_{G, \Ran(X)}$. In fact, the present article was motivated by an application of this kind, see \cite[Prop.\ 1.4]{tz}. 

In comparison with these known results, the main difficulty of proving point \textbf{(2)} lies in the fact that the `cohomology theory' sending $X \rightsquigarrow \Pic(X)$ does not satisfy a K\"unneth theorem, so the proofs of (a) and (b) do not directly apply. The relevant substitute is the Theorem of the Cube~ \cite[\href{https://stacks.math.columbia.edu/tag/0BF4}{Tag 0BF4}]{stacks}. This theorem requires 2 out of 3 varieties to be proper, so the hypothesis (C) in point \textbf{(2)} is designed to allow us to reduce to the proper case, at least when $S$ is regular (e.g.\ a field). 

Note that, if $\on{char}(k) = 0$, then every smooth geometrically connected $k$-variety $X$ satisfies hypothesis (C). This follows from Nagata's compactification theorem and Hironaka's theorem on resolution of singularities. 

\begin{rmk}
	The proof of (c) given in \cite{z} only works when $k$ is of characteristic zero, because this is the generality in which \cite[Lem.\ 4.3.11]{z} holds. For more discussion of this point, see Remark~\ref{exercise}. In the course of proving \textbf{(2)}, we will supply a proof of (c) that works in arbitrary characteristic, at least when $S$ is locally Noetherian, see Lemma~\ref{func1}. 
\end{rmk}

\subsubsection{Discussion of \textbf{(3)}} \label{discuss-3}

In Goodwillie calculus, a functor $G : \mc{C}_1 \to \mc{C}_2$ between $(\infty, 1)$-categories is called $n$-excisive if $G$ sends strongly cocartesian $(n+1)$-hypercubes in $\mc{C}_1$ to weakly cartesian $(n+1)$-hypercubes in $\mc{C}_2$. Here `strongly cocartesian' means that every 2-dimensional face is a cocartesian square, while `weakly cartesian' means that the initial vertex is the limit of the remaining part of the hypercube. The paradigm of Goodwillie calculus is to study an arbitrary functor $G : \mc{C}_1 \to \mc{C}_2$ by its $n$-excisive approximations for successively larger $n$. The analogy with calculus arises because these approximations behave like Taylor series.

As indicated in \textbf{(3)}, we consider $n$-excisive functors $\fset_* \to \mc{C}$ where $\mc{C}$ is an abelian category or an $(\infty, 1)$-category. Let us indicate the two examples of this notion which will be relevant in this paper. The following lemmas, which can be treated as exercises, say that the functors of `polynomials of degree $\le d$' and `line bundles' are $d$-excisive and 2-excisive, respectively: 

\begin{lem}
	For each $1 \le i \le d+1$, let $L_i \subset \BA^{d+1}_{k}$ denote the $i$-th coordinate hyperplane, defined by $x_i = 0$. The space of polynomials of degree $\le d$ on $\BA^{d+1}_k$ is isomorphic to the space parameterizing the following data: 
	\begin{itemize}
		\item For each $1 \le i \le d+1$, we have a polynomial of degree $\le d$ on $L_i$, denoted $f_i$. 
		\item For each pair $i, j$ with $1 \le i < j \le d+1$, we require that $f_i|_{L_i \cap L_j} = f_j|_{L_i \cap L_j}$. 
	\end{itemize}
\end{lem}
\begin{proof}
	This is equivalent to the assertion that the functor $G: \fset_*\to \on{Vect}_k$ is $d$-excisive, where $G$ sends a set $\{*\} \sqcup I$ to the space of polynomials of degree $\le d$ on $\BA^I_k$. The proof follows by considering one monomial function on $\BA^{d+1}_k$ at a time, see Lemma~\ref{rnd-cart}. 
\end{proof}

\begin{lem}
	Let $X$ be a proper, geometrically connected $k$-variety with a basepoint $x_0 \in X(k)$. Then $\Pic(X^3)$ is equivalent to the category whose objects are described as follows: 
	\begin{itemize}
		\item We have a one-dimensional $k$-vector space $\mc{F}$. 
		\item For each $i \in \{1, 2, 3\}$, we have a line bundle $\mc{E}_i \in \Pic(X)$ and an isomorphism 
		\[
		q_i : \mc{E}_i|_{x_0} \xrightarrow{\sim} \mc{F}.
		\]
		\item For each pair $i, j$ with $1 \le i < j \le 3$, we have a line bundle $\mc{L}_{i, j} \in \Pic(X^2)$ and isomorphisms 
		\e{
			\sigma_{i,j} : \mc{L}_{i, j}|_{X \times \{x_0\}} &\xrightarrow{\sim} \mc{E}_i \\
			\tau_{i,j} : \mc{L}_{i, j}|_{\{x_0\} \times X} &\xrightarrow{\sim} \mc{E}_j.
		}
		\item These data are subject to the condition that, for any $i < j$ as above, we have 
		\[
		q_i \circ \sigma_{i, j}|_{x_0} = q_j \circ \tau_{i, j}|_{x_0}
		\]
		as maps $\mc{L}_{i, j}|_{(x_0, x_0)} \xrightarrow{\sim} \mc{F}$. 
	\end{itemize}
\end{lem}
\begin{proof}
	This is a special case of the assertion that the functor $G : \fset_* \to \on{Grpd}$ is 2-excisive,\footnote{Note that $G$ lands in the category of strictly commutative Picard groupoids, which identifies with the extension-closed subcategory of $\on{D}^b(\on{AbGrp})$ consisting of objects concentrated in cohomological degrees $[-1, 0]$. Thus, this example can be analyzed in the context of functors which land in the stable $(\infty, 1)$-category $\on{D}^b(\on{AbGrp})$. However, note that the composition of $G$ with the full embedding into $\on{D}^b(\on{AbGrp})$ need not be 2-excisive.} where $G$ sends a set $\{*\} \sqcup I$ to $\Pic(X^I)$. This 2-excision statement is deduced in~\ref{cube-improved}. Alternatively, it is not difficult to deduce this lemma from the standard statement of the Theorem of the Cube~\cite[\href{https://stacks.math.columbia.edu/tag/0BF4}{Tag 0BF4}]{stacks}. 
\end{proof}

Suitably amplified, these facts allow us to apply \textbf{(3)} and deduce the triviality of regular functions and line bundles on $\Ran(X)$. 

\begin{rmk}
	The notion of $n$-excisive functor on $\fset_*$ has been studied by Berger~\cite{berger}. Also, the interpretation of the Theorem of the Cube as saying that $I \mapsto \Pic(X^I)$ is a `quadratic' functor is well-known, see~\cite[Sect.\ 3]{cast} and \cite{kivimae} for example. We emphasize that this  notion of `quadratic functor' is really the same as the aforementioned notion of 2-excisive functor. 
\end{rmk}

\subsection{Overview of the paper} 

In Section~\ref{sec-poly}, we develop the notion of $n$-excisive functor $\fset_* \to \mc{C}$ and prove some vanishing results for limits of such functors over $\fsetsurjne$, including point \textbf{(3)}. We introduce various categories of finite sets~(\ref{ssec-fset}) and study $n$-excisive functors to abelian categories~(\ref{ssec-poly-ab}) and $(\infty, 1)$-categories~(\ref{ssec-poly-stab}). Lastly, we investigate analogous results for functors which are only defined on sets $I$ whose size is bounded above by some fixed integer $N$, see~\ref{ssec-fin}. One such `finite limit' result is needed for the proof of Lemma~\ref{func2}, which is an enhanced version of the triviality of functions on $\Ran(X)$. 

In Section~\ref{ran-def}, which is purely expository, we review the definition of the Ran space and of quasicoherent sheaves and line bundles on it. In~\ref{sec-defs}, we introduce various categories of prestacks and explain why several possible definitions of the Ran space coincide (\ref{ran-set}). In~\ref{sec-sheaves}, we define various categories of sheaves on $\Ran(X)$ and we explain how the derived and abelian categories relate to each other. 

In Section~\ref{sec-canon}, we prove point \textbf{(1)}. We define infinitesimal analogues of the Ran space~(\ref{ran-inf-def}) and prove that flat quasicoherent sheaves on the infinitesimal Ran space are canonically trivial (Proposition~\ref{prop-triv}). Next, in~\ref{ran-partial-def}, we use partially-labeled Ran spaces to bootstrap this result from the infinitesimal Ran space to the completed Ran space which Beilinson and Drinfeld introduced in order to prove the existence of a canonical connection on a unital factorization algebra (see~\ref{discuss-1}). Finally, we complete the proof in~\ref{ssec-canon-proof}. 

In Section~\ref{sec-pic-triv}, we prove point \textbf{(2)} in the case when $S$ is a field. In~\ref{ssec-pic-triv-1}, we use the Theorem of the Cube to show that \textbf{(3)} can be applied, and at the same time we deduce an improvement of the Theorem of the Cube, see Corollary~\ref{cube-improved}. Although the proof of this case of \textbf{(2)} uses the existence of a $k$-rational point and a smooth compactification of $X$, we use Galois descent in~\ref{ssec-pic-triv-2} to weaken these assumptions to the hypothesis (C) stated in point \textbf{(2)}. 

In Section~\ref{sec-rel-pic}, we finally prove point \textbf{(2)} when $S$ is any locally Noetherian $k$-scheme. In~\ref{rigid}, we introduce the notion of a rigidified line bundle, i.e.\ a line bundle equipped with trivialization at a basepoint. In~\ref{ssec-func}, we prove that functions on $S \times \Ran(X)$ are pulled back from $S$. In~\ref{formal-sec}, we use this to prove the result when $S$ is Artinian, and then in~\ref{final} we bootstrap to the case when $S$ is arbitrary. This relies on Beauville--Laszlo gluing for regular functions (rather than quasicoherent sheaves), which we discuss in~\ref{bl-sec}. In these subsections, we sometimes require the hypothesis that $X$ is affine, but we remove this hypothesis in~\ref{ssec-general} by bootstrapping from the affine case. This proof loosely follows the strategy of `reduction to the local Artinian case' introduced in~\cite[I, p.\ 8--9]{ega}. For more discussion of this point, see Remark~\ref{rmk-ega}.  

\subsection{Notations} 

\subsubsection{General notations} 

In this paper, $k$ is a field, and we never impose any assumptions on the characteristic of $k$. Let $\on{Vect}_k$ be the abelian category of vector spaces over $k$. 

We will consider an abelian category $\mc{A}$ and an $(\infty, 1)$-category $\mc{C}$. All derived categories in this paper will be indicated as such, e.g.\ $\on{D}^b(\on{AbGrp})$. In particular, $\QCoh(-)$ refers to the \emph{abelian} category of quasicoherent sheaves. 

All of our categories will be locally small. We say that a category is \emph{complete} if it admits limits indexed by arbitrary small (equivalently, essentially small) diagrams. A limit in an $(\infty, 1)$-category is always to be understood as a homotopy limit.

We employ the usual derived functor notation, e.g.\ $\lim^i$, $\on{R}^i\Gamma$. All derived functors and $t$-structures will be subject to \emph{cohomological} indexing, with the exception of $\pi_1, \pi_0, \ldots$ which has homological indexing. 

We consider a scheme $X$ over $k$, and we denote its base change to $\ol{k}$ by $X_{\ol{k}}$. By a `prestack over $k$' we mean a contravariant functor on $\on{Sch}_{k}^{\on{aff}, \on{ft}}$, which denotes the category of affine finite type schemes over $k$ -- see~\ref{prestacks} for more details. If a product $\times$ or tensor product $\otimes$ appears with no subscript, it is taken over $k$. 

A boldface $\Pic(-)$ refers to the (strictly commutative) Picard groupoid of line bundles, whereas an ordinary $\on{Pic}(-)$ indicates the Picard group. An underline $\ul{\Pic}(-)$ indicates a (relative) Picard functor. 

When an object in a category $c \in \mc{C}$ appears with an underline $\ul{c}$, this refers to a constant functor with value $c$. (Exception: in~\ref{prop-triv}, an underline $\ul{V}$ denotes a free quasicoherent sheaf with fiber $V$.) 

A hat $\wh{(-)}$ or $(-)^\wedge$ will always denote completion, and $(-)^{\wedge, \mf{a}}$ means completion with respect to the ideal $\mf{a}$. 

\subsubsection{Specialized notations} 

In~\ref{ssec-fset}, we will introduce various categories of finite sets, including $\fset_*$ and $\fsetsurjne$. We denote the faithful embedding $\fsetsurjne \hra \fset_*$ by $\iota$, so that $\iota(I) := \{*\} \sqcup I$. In~\ref{ssec-fin}, we introduce full subcategories consisting of finite sets bounded in size by some integer $N$, e.g.\ $\fset_{*, \le N}$ and $\fset^{\on{surj}}_{\on{n.e.}, \le N}$. Note that the `size' of a pointed set $\{*\} \sqcup I$ is defined to be $|I|$. The letters $I$ and $J$ will always refer to finite sets. We let $[n] := \{1, 2, \ldots, n\}$, so that $[0] = \emptyset$. For pointed finite sets, we have the smash product and wedge sum operations, which satisfy 
\e{
	(\{*\} \sqcup I) \wedge (\{*\} \sqcup J) &\simeq \{*\} \sqcup I \times J \\
	(\{*\} \sqcup I) \vee (\{*\} \sqcup J) &\simeq \{*\} \sqcup I \sqcup J. 
}

In this paper, any limit is taken over $\fsetsurjne, \fset^{\on{surj}}_{\on{n.e.}, \le N}$ for some $N$, $\BZ_{\ge 0}$, or a hypercube diagram. In the first two cases, we make the restriction to surjective maps between nonempty subsets explicit by precomposing the functor in question by $\iota$, whose domain is $\fsetsurjne$, see~\ref{ssec-fset}. There are no limits which are taken over $\fset_*$, although we often consider functors defined on this larger category. 

The notation $\Xi_{(I_b)_{b \in B}}$ refers to a `special' hypercube diagram in $\fset_*$, and it is introduced in~\ref{ssec-poly-stab}. 

In~\ref{ran-inf-def}, we introduce the infinitesimal Ran space $\raninf^n$, the Artinian Ran space $\ranf{d}^n$, and associated functors $\mc{P}_n, \mc{P}_{n, d} : \fset_* \to \on{Vect}_k$ corresponding to their sheaves of regular functions. In~\ref{ran-partial-def}, we define $\Ran^I(-), \Ran_{\wh{\Delta}}^I(-)$, and $(-)^I_{\wh{\Delta}}$. 

In~\ref{rigid}, we define the notion of `rigidified' object, and denote it with an $e$-superscript, e.g.\ $\Pic^e(-)$ and $\Gamma^e(-, \oh)$.

\subsection{Acknowledgments} 

The author would like to thank Dennis Gaitsgory for suggesting the problem of showing that $\Ran(X)$ is Pic-contractible, Yifei Zhao for providing a wealth of helpful comments on an earlier draft, and Charles Fu for providing comments and corrections on a later draft. Much of the paper in its current form was inspired by Yifei Zhao's suggestion that the proof of triviality of $\Pic(\Ran(X))$ could be interpreted as a vanishing theorem for limits of polynomial functors. This work was supported by the National Science Foundation Graduate Research Fellowship. 

\section{$n$-excisive functors on finite pointed sets} \label{sec-poly}

\subsection{Categories of finite sets} \label{ssec-fset}

Let $\fset_*$ denote the category of pointed finite sets. Let $\fsetsurj$ denote the category whose objects are (possibly empty) finite sets and whose morphisms are surjective maps. Note that $\fsetsurj = \{\emptyset\} \sqcup \fsetsurjne$ where the second term is the subcategory consisting of nonempty finite sets. We have a faithful embedding $\iota : \fsetsurjne \to \fset_*$ which sends $I \mapsto \{*\} \sqcup I$. The symbol $*$ will always denote the basepoint of an object in $\fset_*$.

\begin{rmk}
	Note that the functor $\iota$ changes the underlying set. To remedy this, it is perhaps better to think of $\fset_*$ as the category whose objects are (non-pointed) finite sets and whose morphisms are \emph{partially-defined} maps. We will not use this interpretation so as to avoid introducing separate notation for the phrase `the map $\psi : I \to J$ is not defined on the element $i \in I$.' Instead, we can just write $\psi(i) = *$. 
\end{rmk}

Define $[n] := \{1, 2, \ldots, n\}$ for integers $n \ge 0$.

\subsubsection{Inclusions and projections} 
For a subset $S \subset I$, define the maps 
\begin{cd}
	\{*\} \sqcup S \ar[r, shift left = 1.5, "\phi_{S, I}"] & \{*\} \sqcup I \ar[l, shift left = 0.5, "\psi_{S , I}"] 
\end{cd}
as follows: $\phi_{S , I}$ is induced by the inclusion $S \hra I$, and $\psi_{S , I}$ is the identity on $S$ and sends $I \setminus S$ to $\{*\}$. For $S \subset S' \subset I$, we clearly have the identities
\e{
	\psi_{S , I} \circ \phi_{S , I} &= \id_{\{*\} \sqcup S} \\
	\psi_{S, S'} \circ \psi_{S', I} &= \psi_{S,I} \\
	\phi_{S', I} \circ \phi_{S,S'} &= \phi_{S, I}. 
}

\subsection{Polynomial functors to an abelian category} \label{ssec-poly-ab}

In this subsection, we define the notion of a \emph{polynomial functor of degree $\le n$} taking values in an abelian category, and we prove a vanishing theorem for its limit over $\fsetsurjne$. Here are the main ideas: 

Given a functor $G : \fset_* \to \mc{A}$ where $\mc{A}$ is an abelian category, we use the retracts $\psi_{S, [d]} \circ \phi_{S, [d]} = \id_S$ to extract the `homogeneous parts' of $G$, which constitute the values of another functor $\on{Prim}(G) : \fsetsurj \to \mc{A}$. Roughly speaking, the $d$-th homogeneous part of $G$ consists of those elements of $G([d])$ which do not come from $G([d-1])$. This procedure bears a mild resemblance to the passage from the \emph{unnormalized chain complex} to the \emph{normalized chain complex} in the context of the Dold-Kan correspondence; compare \cite[Def.\ 1.2.3.9]{ha} with Lemma~\ref{lem-poly}(ii) below. 

We show that $G$ is determined by $\on{Prim}(G)$ (Proposition~\ref{prop-indprim}), and we define $G$ to be `polynomial of degree $\le n$' if it has no $d$-th homogeneous pieces for $d > n$ (Definition~\ref{def-poly}). The main result of this subsection is Proposition~\ref{prop1}, which shows that, if $G$ is polynomial of degree $\le n$ for some $n$, then the limit of $G$ over $\fsetsurjne$ is trivial. 

\subsubsection{} Let $\mc{A}$ be an abelian category. Define functors   \label{def-indprim} 
\begin{cd}
	\fun(\fsetsurj, \mc{A}) \ar[r, shift left = 1.5, "\on{Ind}"] & \fun(\fset_*, \mc{A}) \ar[l, shift left = 0.5, "\on{Prim}"] 
\end{cd}
as follows: 
\begin{itemize}
	\item Given $F : \fset^{\on{surj}} \to \mc{A}$, let $G = \on{Ind}(F) : \fset_* \to \mc{A}$ be defined as follows. 
	\begin{itemize}
		\item On objects $\{*\} \sqcup I \in \fset_*$, define $G(\{*\} \sqcup I) = \bigoplus_{S \subset I} F(S)$. 
		\item On morphisms $\xi : \{*\} \sqcup I \to \{*\} \sqcup J$, define the map $G(\xi) : \bigoplus_{S \subset I} F(S) \to \bigoplus_{T\subset J} F(T)$ such that the matrix coefficient $F(S) \to F(T)$ corresponding to $(S, T)$ is $F(S \xrightarrow{\xi} \xi(S))$ if $T = \xi(S)$ and zero otherwise. 
	\end{itemize}
	\item Given $G : \fset_* \to \mc{A}$, let $F = \on{Prim}(G) : \fsetsurj \to \mc{A}$ be defined as follows. 
	\begin{itemize}
		\item On objects $I\in \fsetsurj$, define $F(I) = \bigcap_{S \subsetneq I} \ker(G(\psi_{S, I}))$ as a subobject of $G(\{*\} \sqcup I)$. 
		\item On morphisms $\xi : I \sra J$, define $F(\xi)$ to be induced by $G(\{*\} \sqcup I) \xrightarrow{G(\{*\}\sqcup \xi)} G(\{*\} \sqcup J)$. 
	\end{itemize}
\end{itemize} 
The definition of $\on{Ind}$ and $\on{Prim}$ on morphisms (i.e.\ natural transformations) is obvious. 

\subsubsection{} We show that any functor $G$ is determined by its homogeneous pieces $\on{Prim}(G)$. 
\begin{prop} \label{prop-indprim}
	$\on{Ind}$ and $\on{Prim}$ are mutually inverse equivalences of abelian categories. 
\end{prop}
\begin{proof}
	First, we discuss the composition $\on{Prim} \circ \on{Ind}$. Given $F : \fset^{\on{surj}} \to \mc{A}$, we get the functor $\on{Ind}(F)$ whose value on $\{*\} \sqcup I$ is $\bigoplus_{S \subset I} F(S)$. For a fixed $I' \subsetneq I$, the map 
	\[
		\on{Ind}(F)(\psi_{I', I}) : \bigoplus_{S \subset I} F(S) \to \bigoplus_{S' \subset I'} F(S')
	\]
	has matrix coefficient $F(S) \to F(S')$ given by $F(S) \xrightarrow{\id} F(S)$ if $S = S'$ and zero otherwise. Therefore, the kernel of this map is $\bigoplus_{S \subset I \text{ and } S \not\subset I'} F(S)$. Varying $I'$, this shows that 
	\[
		\bigcap_{I' \subsetneq I} \ker(\on{Ind}(F)(\psi_{I', I})) = F(I). 
	\]
	The resulting isomorphism $\on{Prim}(\on{Ind}(F))(I) \simeq F(I)$ extends to a natural isomorphism $\on{Prim} \circ \on{Ind} \simeq \id$. 
	
	Next, we discuss the composition $\on{Ind} \circ \on{Prim}$. We begin by finding a direct sum decomposition of $G(\{*\} \sqcup I)$. For each $S \subset I$, define the idempotent endomorphism $e_{S, I} = \phi_{S, I} \circ \psi_{S, I}$. By definition, $e_{S, I} \circ e_{S', I} = e_{S\cap S', I}$, so these endomorphisms pairwise commute. Restrict attention to $e_i := e_{I\setminus \{i\}, I}$. Each $e_i$ induces a direct sum decomposition 
	\[
		G(\{*\} \sqcup I) \simeq \ker(G(e_i)) \oplus \Im(G(e_i)), 
	\]
	and commutativity implies that these decompositions for various $i \in I$ are compatible with one another. Therefore we get a direct sum decomposition 
	\[
		G(\{*\} \sqcup I) \simeq \bigoplus_{S \subset I} \left( \bigcap_{i \in I} \Big(\ker(G(e_i)) \text{ if } i \in S \text{ and } \Im(G(e_i)) \text{ otherwise}\Big) \right). 
	\]
	
	For any $S \subset I$, we have $\bigcap_{i \in I \setminus S}(\Im(G(e_i))) = \Im(G(e_{S, I}))$. Furthermore, via the isomorphism $\Im(G(e_{S, I})) \underset{\sim}{\xrightarrow{G(\psi_{S, I})}} G(\{*\} \sqcup S)$, we obtain an isomorphism 
	\[
		\Im(G(e_{S, I})) \cap \left(\bigcap_{i \in S} \ker(e_i) \right) \simeq \bigcap_{i \in S} \ker(G(\psi_{S \setminus \{i\}, S})), 
	\]
	where the right hand side is a subobject of $G(\{*\} \sqcup S)$. 	This is because the idempotent $e_{S, I}$ splits as $\phi_{S, I} \circ \psi_{S, I}$, the maps $\psi_{S, I}$ and $\phi_{S, I}$ both intertwine the idempotent $e_i \in \End(\{*\} \sqcup I)$ with the idempotent $\phi_{S \setminus \{i\}, S} \circ \psi_{S \setminus \{i\}, S} \in \End(\{*\} \sqcup S)$ for all $i \in S$, and the latter idempotent splits as indicated. Thus, the direct sum decomposition rewrites as 
	\e{
		G(\{*\} \sqcup I) &\simeq \bigoplus_{S \subset I}  \left(\bigcap_{i \in S} \ker(G(\psi_{S \setminus \{i\}, S}))\right) \\
		&= \bigoplus_{S \subset I} \on{Prim}(G)(S) \\
		&= \on{Ind}(\on{Prim}(G))(\{*\} \sqcup I). 
	}
	This canonical isomorphism extends to a natural isomorphism $\on{Ind} \circ \on{Prim} \simeq \id$. 
\end{proof}

\subsubsection{} The equivalent conditions of this next lemma will explain what it means for a functor $G$ to be polynomial of degree $\le n$. 
\begin{lem}\label{lem-poly}
	Let $G : \fset_* \to \mc{A}$ be a functor, and let $n \ge 0$ be an integer. The following are equivalent: 
	\begin{enumerate}[label=(\roman*)]
		\item For every $I$ with $|I| > n$, we have 
		\[
			G(\{*\} \sqcup I) = \sum_{i \in I} \Im(G(\phi_{I \setminus \{i\}, I})).
		\]
		\item For every $I$ with $|I| > n$, we have 
		\[
			\bigcap_{i \in I} \ker(G(\psi_{I \setminus \{i\}, I})) = 0. 
		\]
		\item For every $I$ with $|I| > n$, this is an equalizer sequence: 
		\begin{cd}
			G(\{*\} \sqcup I) \ar[r, "f"] & \bigoplus_{i \in I} G(\{*\} \sqcup I \setminus \{i\}) \ar[r, shift left = 1.5, "g_1"] \ar[r, shift right = 0.5, swap, "g_2"] & \bigoplus_{\substack{i,j \in I \\ i \neq j}} G(\{*\} \sqcup I \setminus \{i, j\}).
		\end{cd}
		The sequence is defined with reference to a linear order on $I$. The matrix coefficients of $f$ are given by $G(\psi_{I \setminus \{i\},I})$. For $g_1$, the matrix coefficient
		\[
			G(\{*\} \sqcup I \setminus \{i\}) \to G(\{*\} \sqcup I \setminus \{i, j\})
		\]
		is given by $G(\psi_{I \setminus \{i, j\}, I \setminus \{j\}})$ if $i < j$, and all other matrix coefficients are zero. The matrix coefficients for $g_2$ are defined by the same rule, replacing $i < j$ with $i > j$. 
		\item The functor $\on{Prim}(G)$ sends $I \mapsto 0$ for all $|I| > n$.
	\end{enumerate}
\end{lem}
\begin{proof}
	In view of Proposition~\ref{prop-indprim}, we have $G \simeq \on{Ind}(F)$ for some $F : \fset^{\on{surj}} \to \mc{A}$. We want to show that each of (i), (ii), and (iii) is equivalent to the assertion that $F(I) = 0$ for all $|I| > n$. 
	
	For (i), this follows because $G(\{*\} \sqcup I) = \bigoplus_{S \subset I} F(S)$ by definition, and the right hand side is the sum of $F(S)$ for $S \subsetneq I$. So the equation in (i) asserts that $F(I) = 0$. 
	
	For (ii), the left hand side is $F(I)$ by definition of $F = \on{Prim}(G)$. 
	
	For (iii), the term $G(\{*\} \sqcup I \setminus \{i\})$ has a summand $F(S)$ for each $S$ satisfying $S \subset I \setminus \{i\}$. Likewise, $G(\{*\} \sqcup I \setminus \{i, j\})$ has a summand $F(S)$ for each $S$ satisfying $S \subset I \setminus \{i, j\}$. Examining the maps $g_1$ and $g_2$ on these summands shows that the equalizer identifies with $\bigoplus_{S \subsetneq I} F(S)$. So (iii) asserts that $F(I) = 0$. 
\end{proof}

\subsubsection{} We arrive at the main definition and result of this subsection: 
\begin{defn} \label{def-poly}
	Let $G : \fset_* \to \mc{A}$ be a functor. We say that \emph{$G$ is polynomial of degree $\le n$} if $G$ satisfies the equivalent conditions of Lemma~\ref{lem-poly}. 
\end{defn}

\begin{prop}\label{prop1}
	Assume that $\mc{A}$ is complete. If $G : \fset_* \to \mc{A}$ is polynomial of degree $\le n$, for some $n$, then $\lim(G \circ \iota) \simeq G(\{*\})$.\footnote{The functor $\iota$ was defined in~\ref{ssec-fset} and it appears here to indicate that the limit is over $\fsetsurjne$, not $\fset_*$.}  
\end{prop}

We give two proofs of this proposition. 

\begin{myproof}{Proof 1}
	This proposition can be deduced from the proof of Theorem~\ref{thm2}. One carries out that proof replacing all homotopy limits with ordinary limits, skipping the d\'evissage argument which is used to reduce to the case of a functor landing in an abelian category, and replicating the elementary part of the argument of Lemma~\ref{lem-induct}. This proof is not circular because this proposition is not used in~\ref{ssec-poly-stab}. 
\end{myproof}

\begin{myproof}{Proof 2} 
	Here, we give a direct and elementary proof. Since $\{*\} \in \fset$ is terminal, we have a natural transformation $G \to \ul{G(\{*\})}$ where underline denotes the constant functor. Taking limits over $\fsetsurjne$ yields a map $f : \lim(G \circ \iota) \to G(\{*\})$. By the Yoneda lemma, $f$ is an isomorphism if and only if $\Hom_{\mc{A}}(M, f)$ is an isomorphism for all $M \in \mc{A}$. In this way, we reduce to the case in which $\mc{A}$ is the category of abelian groups. 
	
	An element of the limit is a collection of elements $a_I \in G(\{*\} \sqcup I)$ for all nonempty $I$ which are compatible under the maps $\{*\} \sqcup I \to \{*\} \sqcup J$ induced by surjections $I \sra J$. By hypothesis, we have $G \simeq \on{Ind}(F)$ where $F(I) = 0$ for all $|I| > n$. By definition of $\on{Ind}$, the element $a_I$ corresponds to a tuple of elements $b_{S, I} \in F(S)$ for $S \subset I$. It suffices to show that $b_{S, I} = 0$ if $S$ is nonempty. 
	
	First, considering permutations of $I$ shows that $b_{S, I}$ depends only on the cardinality of $S$. Therefore we can write $b_{S, I}$ as $b_{m, I}$ when $|S| = m$. We shall use downward induction on $m$ to show that $b_{m, I} = 0$ for all $m \ge 1$. The claim for $m > n$ follows from our assumption that $F(I) = 0$ for all $|I| > n$. Assume that the claim holds for $m + 1$, and prove it for $m$ as follows. For any $d \ge m$, consider the map
	\[
		\{*\} \sqcup [d+1] \xrightarrow{\xi} \{*\} \sqcup [d]
	\]
	which sends $i \mapsto \min(i, d)$. Observe that
	\begin{itemize}
		\item If $J\subset [d]$ is an $m$-element subset which does not contain $d$, then there is exactly one $m$-element subset of $[d+1]$ which maps to $J$ under $\xi$. 
		\item If $J \subset [d]$ is an $m$-element subset which contains $d$, then there are exactly two $m$-element subsets of $[d+1]$ which map to $J$ under $\xi$. 
	\end{itemize}
	
	If $d > m$, then there exist $m$-element subsets $J \subset [d]$ which contain (resp.\ do not contain) $d$, so the definition of $\on{Prim}(F)(\xi)$ implies that 
	\[
		b_{m, [d]} = b_{m, [d+1]} = 2b_{m, [d+1]}. 
	\]
	We conclude that $b_{m, [d+1]} = b_{m, [d]} = 0$. 
	
	It remains to show that $b_{m, [m]} = 0$. If $d = m$, then the second bullet point implies 
	\[
		b_{m, [m]} = 2b_{m, [m+1]}, 
	\]
	but we already know that $b_{m, [m+1]} = 0$, so $b_{m, [m]} = 0$, as desired. 
	
	Now, applying the compatibility condition to any isomorphism $I \simeq [d]$ shows that $b_{m, I} = 0$ for all $I$, which proves the inductive step. 
\end{myproof}

\begin{rmk}
	Although the first proof is more efficient, we give the second proof for expository purposes. It will be imitated to prove Lemma~\ref{prop1-fin}, which is an analogue for finite limits. There, unlike here, we do not know how to deduce the result from an analogous result in the homotopic setting, see Remark~\ref{further-fin2}. 
\end{rmk}

\subsubsection{Remark} 

The notion of `polynomial functor' introduced in Definition~\ref{def-poly} is a special case of the notion of $n$-excisive functor to be introduced in Definition~\ref{def-excisive}. We chose a different terminology so as to distinguish between functors landing in abelian categories and $(\infty, 1)$-categories, since the term `excisive' suggests homotopy limits. However, we emphasize that there does not seem to be a relation between `polynomial functor' in our sense and the corresponding term in combinatorics, referring to endofunctors $\fset \to \fset$. 

\subsubsection{Remark}\label{counter} 

We give an example to show that a non-polynomial functor $G$ need not satisfy the vanishing result of Proposition~\ref{prop1}. In fact, the simplest example of a non-polynomial functor suffices. 

\begin{claim}
	There is a functor $G : \fset_* \to \on{AbGrp}$ such that the map $\lim G\iota \to G(\{*\})$ is not an isomorphism. 
\end{claim}
\begin{proof}
	Take $G = \on{Ind}(\ul{\BZ})$ where $\ul{\BZ}$ denotes the constant functor $\fsetsurj \to \on{AbGrp}$. Thus, we have $G(\{*\} \sqcup I) \simeq \BZ^{\oplus 2^I}$, and for every map $\xi: \{*\} \sqcup I \to \{*\}\sqcup J$, the induced map 
	\[
		\bigoplus_{S \subset I} \BZ \simeq G(\{*\} \sqcup I ) \xrightarrow{G(\xi)} G(\{*\} \sqcup J) \simeq \bigoplus_{T \subset J} \BZ
	\]
	has $(S, T)$ matrix coefficient equal to $\id_{\BZ}$ if $\xi(S) = T$ and zero otherwise. 
	
	We will define an element $a_I \in G(\{*\} \sqcup I)$ for every nonempty finite set $I$ such that the $a_I$ are compatible with respect to maps of the form $\{*\} \sqcup f$ where $f : I \sra J$ is a surjection. Namely, define $a_I \in \bigoplus_{S \subset I} \BZ$ to be $1$ on the $S = I$ coordinate and 0 on all other coordinates. The claim that $G(\{*\} \sqcup f)(a_I) = a_J$ follows from the fact that the image of $I \subset \{*\} \sqcup I$ under $\{*\} \sqcup f$ is $J \subset \{*\} \sqcup J$ because $f$ is surjective. 
	
	This defines a nonzero element in $\lim G\iota$. The image of this element under the map to $G(\{*\})$ is zero, because the $S = \emptyset$ coordinate of each $a_I$ is zero. This shows that $\lim G\iota \to G(\{*\})$ is not injective. 		
\end{proof}

\subsubsection{} Here is one way to prove that a functor is polynomial: 
\begin{lem}\label{thicc} 
	In the abelian category $\fun(\fset_*, \mc{A})$, the full subcategory spanned by polynomial functors of degree $\le n$ is closed under subquotients and extensions. 
\end{lem}
\begin{proof}
	By Proposition~\ref{prop-indprim}, it suffices to show that the full subcategory of $\fun(\fsetsurj, \mc{A})$ spanned by functors which send $I \mapsto 0$ if $|I| > n$ is closed under subquotients and extensions. This is true because kernels and cokernels in a functor category are computed pointwise. 
\end{proof}

\subsection{$n$-excisive functors to an $(\infty, 1)$-category}  \label{ssec-poly-stab}

Our next goal is to develop material analogous to~\ref{ssec-poly-ab} for functors landing in an $(\infty, 1)$-category. Thus, let $\mc{C}$ be an $(\infty, 1)$-category. All limits in this subsection are to be understood as homotopy limits. 

For any finite tuple $(I_b)_{b\in B}$ of finite sets, the maps $\psi_{(\sqcup_{b_2 \in B_2} I_{b_2}), (\sqcup_{b_1 \in B_1} I_{b_1})}$ for $B_2 \subset B_1 \subset B$ combine to form a commutative hypercube, denoted $\Xi_{(I_b)_{b \in B}}$. We shall refer to this as a \emph{special hypercube} in $\fset_*$. 
The vertices of this hypercube are indexed by subsets $B' \subset B$, where $B'$ corresponds to the object $\{*\} \sqcup \left(\sqcup_{b' \in B'} I_{b'}\right)$. 

If a commutative hypercube realizes its initial vertex as the limit of the rest of the diagram, we say that it is \emph{weakly cartesian}. If every subsquare is cartesian, then we say that it is \emph{strongly cartesian}. Similar definitions apply for `cocartesian' in place of `cartesian.' Note that each $\Xi_{(I_b)_{b \in B}}$ is strongly cocartesian.

\begin{defn}\label{def-excisive} 
	We say that a functor $G : \fset_* \to \mc{C}$ is \emph{$n$-excisive} if, for any tuple $(I_b)_{b \in B}$ with $|B|>n$, the diagram $G(\Xi_{(I_b)_{b \in B}})$ is weakly cartesian. 
\end{defn}
\begin{rmk}
	This definition, which is due to Clemens Berger, is directly related to the notion of $n$-excisive functor in Goodwillie calculus. See~\cite{berger} for a precise relation between $n$-excisive functors $\fset_* \to \on{Spaces}_*$ in this sense and $n$-excisive functors $\on{Spaces}_* \to \on{Spaces}_*$. 
\end{rmk}

\subsubsection{Remark.} \label{rmk-compare}
	
	By applying the same definition with $\mc{C}$ replaced by an abelian 1-category $\mc{A}$, we obtain the notion of $n$-excisive functor $G : \fset_* \to \mc{A}$. However, this notion is one which we have already studied. Indeed, for a functor $G : \fset_* \to \mc{A}$, the condition of Lemma~\ref{lem-poly}(iii) is easily seen to be equivalent to the requirement that $G(\Xi_{(i)_{i \in I}})$ is weakly cartesian when $|I| > n$, and this is equivalent to the condition that $G(\Xi_{(I_b)_{b \in B}})$ is weakly cartesian for all tuples $(I_b)_{b \in B}$ with $|B| > n$. Therefore, $G$ is polynomial of degree $\le n$ if and only if it is $n$-excisive. 
	
	\vspace{-\mylength}
	
	In general, if $G : \fset_* \to \mc{A}$ is $n$-excisive, the composed functor $\fset_* \xrightarrow{G} \mc{A} \to \on{D}(\mc{A})$ need not be $n$-excisive, because the image in $\on{D}(\mc{A})$ of a weakly cartesian square in $\mc{A}$ need not be weakly cartesian. However, when $n \le 1$, this composed functor is $n$-excisive. Indeed, in this case $G$ is the direct sum of a constant functor $\fset_* \to \mc{A}$ with a monoidal functor $\fset_* \to \mc{A}$ (intertwining the monoidal structures $\vee$ on $\fset_*$ and $\oplus$ on $\mc{A}$), so the composed functor (from $\fset_*$ to $\on{D}(\mc{A})$) sends special cocartesian hypercubes in $\fset_*$ to \emph{strongly} cartesian hypercubes in $\on{D}(\mc{A})$. 
	
\subsubsection{Paring down the hypercube} 

If $G$ is $n$-excisive, then the value $G\iota(\sqcup_{b \in B} I_b)$ is determined by the diagram $G(\Xi_{(I_b)_{b \in B}})$ minus its initial vertex, provided that $|B| > n$. Now, we show that if $|B| > n+1$, then this diagram is partly redundant: to determine the value $G\iota(\sqcup_{b \in B}I_b)$ it suffices to remember only the part of the diagram which is of height $\le n$. This result will be used to deduce an improved version of the Theorem of the Cube, see Corollary~\ref{cube-improved}. 

For $\Xi_{(I_b)_{b \in B}}$ as above, let $\Xi_{(I_b)_{b \in B}}^{\le n}$ denote the full sub-diagram consisting only of those vertices indexed by subsets $B' \subset B$ with $|B'| \le n$. 

\begin{lem} \label{paring} 
	Assume that $G : \fset_* \to \mc{C}$ is $n$-excisive. Then 
	\[
		\lim G(\Xi_{(I_b)_{b \in B}}^{\le n}) \simeq G(\sqcup_{b \in B} I_b). 
	\]
\end{lem}
\begin{proof}
	First, if $|B| \le n$, then the statement is trivial because $\Xi_{(I_b)_{b \in B}}^{\le n} = \Xi_{(I_b)_{b \in B}}$ in this case, and $G(\sqcup_{b \in B} I_b)$ is the initial object of the diagram $G(\Xi_{(I_b)_{b \in B}}^{\le n})$. Therefore, in what follows, we assume that $|B| > n$. 	
	
	We shall prove the following statement by downward induction on $m$: 
	\begin{itemize}
		\item[($\cdot$)] For $m \ge n$, we have $\lim G(\Xi_{(I_b)_{b \in B}}^{\le m}) \simeq G(\sqcup_{b \in B} I_b)$. 
	\end{itemize}
	The case of $m \ge |B|-1$ follows because $G(\Xi_{{(I_b)_{b \in B}}})$ is a limit diagram, by hypothesis. Fix an integer $m \ge n$, and assume that the statement has been proven for $m+1$. We will apply the following observation:
	\begin{claim}
		Let $\mc{D}$ be a small diagram, and let $\mc{D} \xrightarrow{F} \mc{C}$ be a functor. Let $\mc{D}_0 \hra \mc{D}$ be a full sub-diagram of $\mc{D}$, let $\mc{D}_0^{\triangleleft}$ be the left cone\footnote{This is the diagram obtained from $\mc{D}_0$ by freely adjoining an initial object.} over $\mc{D}_0$, and consider the diagram $\mc{D}_1 := \mc{D} \cup_{\mc{D}_0} \mc{D}_0^{\triangleleft} \xrightarrow{F_1} \mc{C}$ which sends the cone point to $\lim_{\mc{D}_0} F|_{\mc{D}_0}$. Then $\lim F \simeq \lim F_1$. 
	\end{claim}
	\begin{proofofclaim} 
		The object $\lim F_1$ represents the functor which sends an object $c \in \mc{C}$ to $\on{Maps}_{\fun(\mc{D}_1, \mc{C})}(\ul{c}, F_1)$. But we have
		\e{
			\on{Maps}_{\fun(\mc{D}_1, \mc{C})}(\ul{c}, F_1) &\simeq \on{Maps}_{\fun(\mc{D}, \mc{C})}(\ul{c}, F) \underset{\on{Maps}_{\fun(\mc{D}_0, \mc{C})}(\ul{c}, F|_{\mc{D}_0})}{\times} \on{Maps}_{\fun(\mc{D}_0^{\triangleleft}, \mc{C})}(\ul{c}, F_1|_{\mc{D}_0^{\triangleleft}}) \\
			&\simeq  \on{Maps}_{\fun(\mc{D}, \mc{C})}(\ul{c}, F) \underset{\on{Maps}_{\fun(\mc{D}_0, \mc{C})}(\ul{c}, F|_{\mc{D}_0})}{\times} \on{Maps}_{\mc{C}}(c, \lim_{\mc{D}_0} F|_{\mc{D}_0}) \\
			&\simeq \on{Maps}_{\fun(\mc{D}, \mc{C})}(\ul{c}, F). 
		} 
		So $\lim F_1$ and $\lim F$ represent the same functor $\mc{C}^{\on{op}} \to \on{\infty}\!\on{-Grpd}$, as desired. 
	\end{proofofclaim}
	By the inductive hypothesis, $\lim G(\Xi_{(I_b)_{b \in B}}^{\le m+1}) \simeq G(\sqcup_{b \in B} I_b)$. The partial hypercube $\Xi^{\le m+1}$ is obtained from the partial hypercube $\Xi^{\le m}$ by adjoining one point for each subset $B' \subset B$ with $|B'| = m+1$. For each such $B'$, the corresponding vertex attaches as a cone point over the full sub-diagram of $\Xi^{\le m}$ given by the vertices indexed by $B''$ with $B'' \subset B'$. Therefore, by the claim, it suffices to prove that each such cone (for a fixed $B'$) maps under $G$ to a limit diagram. But this holds because each sub-diagram (with its limit point $B'$) is the full hypercube $G(\Xi_{(I_b)_{b \in B'}})$, where $|B'| = m+1 > n$ (since $m \ge n$), and $G$ is $n$-excisive. This proves the inductive step. 
\end{proof}
	
\subsubsection{} We now state the main vanishing result for limits of $n$-excisive functors: 

\begin{thm}\label{thm2} 
	Let $\mc{C}$ be a complete stable $(\infty, 1)$-category with a right complete $t$-structure $\mc{C}^{\le 0}$, and let $G : \fset_* \to \mc{C}$ be an $n$-excisive functor. We assume: 
	\begin{itemize}
		\item[\emph{($\bullet$)}] The essential image of $G$ lies in $\mc{C}^{\ge m}$ for some $m$. 
	\end{itemize}
	Then the natural map $\lim (G \circ \iota) \to G(\{*\})$ is an isomorphism. 
\end{thm}

\begin{rmk}
	Before embarking on the proof, which will occupy the rest of this subsection, let us first describe the motivation. This proof follows the same formal pattern as the proof in \cite[Sect.\ 6]{g} of the vanishing of the de Rham cohomology of $\Ran(X)$, see \ref{discuss-2}(a). To see the analogy, suppose $G(\{*\} \sqcup I)$ is of the form $F(X^I)$ for some functor $F : \on{Sch}_{k}^{\on{aff}, \op} \to \mc{C}$. Then the functor $M$ introduced below is given by $M(\{*\} \sqcup S) = F(\Ran(X)^S)$.\footnote{The value of $F : \on{Sch}_{k}^{\on{aff}, \op} \to \mc{C}$ on a prestack is defined by right Kan extension along the fully faithful embedding $\on{Sch}_{k}^{\on{aff}} \hra \fun(\on{Sch}_{k}^{\on{aff, op}}, \on{Set}) = \on{PreStk}_{k}$, so we have \[F(\Ran(X)^S) \simeq \lim_{I \in \fsetsurjne} F(X^{I \times S}) \simeq \lim_{I \in \fsetsurjne} G(\{*\} \sqcup (I \times S)) =: M(\{*\} \sqcup S).\] For the first isomorphism, see~\ref{explain}.} Since $\Ran(X)$ admits a semigroup multiplication $\Ran(X)^S \xrightarrow{\on{mult}} \Ran(X)$, we could apply $F$ to deduce a map $M(\{*, 1\}) \xrightarrow{F(\on{mult})} M(\{*\} \sqcup S)$ for every $S$. This is the map constructed in Lemma~\ref{lem-act}(i) (except there we carry around $I$ `useless' copies of $\Ran(X)$). The assertion of Lemma~\ref{lem-act}(ii) is a consequence of the fact that the composition 
	\[
		\Ran(X) \xrightarrow{\Delta} \Ran(X)^S \xrightarrow{\on{mult}} \Ran(X)
	\]
	is the identity. Lastly, after a d\'evissage argument to reduce to the situation of a functor landing in an abelian category, Lemma~\ref{lem-induct} performs the same cancellation trick as in \cite[Sect.\ 6]{g}. The main difference is that, whereas Gaitsgory used the K\"unneth formula to inductively reduce to dealing with a degree $\le 1$ functor, we need to deal with degree $\le n$ functors. The $I$ `useless' copies mentioned before are used to pin down $(n-1)$ elements of an $n$-element set, thereby reducing essentially to the case of a degree $\le 1$ functor. 
\end{rmk}

\begin{myproof}{Proof of Theorem~\ref{thm2}}
	Consider the smash product functor $\wedge : \fset_* \times \fset_* \to \fset_*$ which satisfies $\iota(I) \wedge \iota(J) = \iota(I \times J)$. We have the composition 
	\begin{cd}
		\fset_* \times \fsetsurjne \ar[r, "\id \times \iota"] & \fset_* \times \fset_* \ar[r, "\wedge"] & \fset_* \ar[r, "G"] & \mc{C}, 
	\end{cd}
	and we take its limit with respect to the second coordinate to get a functor $M : \fset_* \to \mc{C}$. 
	
	By definition, the functor $M$ has the following behavior on objects: 
	\[
		M(\{*\} \sqcup S) \simeq \lim_{I \in \fsetsurjne} G\iota(S \times I). 
	\]
	Applying $M$ to the map $\{*, 1\} \to \{*\}$ yields the map $\lim (G \circ \iota) \to G(\{*\})$ which appears in the statement of the theorem. 
	
	\subsubsection{} The first step is to show the following: 	
	\begin{lem}\label{m-excisive}
		The functor $M$ is $n$-excisive. 
	\end{lem}
	\begin{proof}
		By definition of $M$, for any tuple $(I_b)_{b \in B}$ with $|B| > n$, the diagram $M(\Xi_{(I_b)_{b\in B}})$ in $\mc{C}$ is the limit of the diagram 
		\begin{equation}\label{diag1}
			G(\Xi_{(I_b)_{b \in B}} \wedge \iota(I))  \tag{$\dagger$}
		\end{equation}
		over $I \in \fsetsurjne$. Since the smash product $\wedge$ distributes over the wedge sum $\vee$, we have a natural isomorphism of diagrams 
		\[
			\Xi_{(I_b)_{b \in B}} \wedge \iota(I) \simeq \Xi_{(I_b \times I)_{b \in B}}. 
		\]
		Since $G$ is $n$-excisive, this implies that the diagram (\ref{diag1}) is weakly cartesian. Thus $M(\Xi_{(I_b)_{b\in B}})$ is also weakly cartesian, because limits commute with limits. 	
	\end{proof}
	
	\subsubsection{} We need to use \cite[1.5.3]{g}, which we restate here: 
	
	\begin{lem}\label{lims} 
		Let $\mc{C}$ be a complete category, and let $\mc{D}_1, \mc{D}_2$ be essentially small categories. Given a functor $\mc{D}_2 \xrightarrow{G} \mc{C}$, we obtain a functor $\fun(\mc{D}_1, \mc{D}_2) \to \mc{C}_{(\lim G)/}$ which sends a functor $\mc{D}_1 \xrightarrow{F} \mc{D}_2$ to a map 
		\[
			\lim G \xrightarrow{\phi_{F}} \lim (G \circ F_1). 
		\] 
	\end{lem}
	\begin{proof}
		Let $\mc{C} \xrightarrow{\Delta} \fun(\mc{D}_1, \mc{C})$ be the functor sending $d$ to the constant functor with value $d$, and similarly for $\mc{D}_2$ in place of $\mc{D}_1$. We have a strictly commutative diagram of functors 
		\begin{cd}[column sep = 0.7in]
			\fun(\mc{D}_2, \mc{C}) \times \fun(\mc{D}_1, \mc{D}_2) \ar[d, swap, "\on{compose}"] & \mc{C} \times \fun(\mc{D}_1, \mc{D}_2) \ar[l, swap, "\Delta \times \id"] \ar[d, "\on{pr}_1"] \\
			\fun(\mc{D}_1, \mc{C}) & \mc{C} \ar[l, swap, "\Delta"] 
		\end{cd} 
		Passing to right adjoints along the horizontal maps, we obtain a \emph{lax} commutative diagram 
		\begin{cd}[column sep = 0.7in]
			\fun(\mc{D}_2, \mc{C}) \times \fun(\mc{D}_1, \mc{D}_2) \ar[d, swap, "\on{compose}"] \ar[r, "\lim \times \id"] & \mc{C} \times \fun(\mc{D}_1, \mc{D}_2) \ar[d, "\on{pr}_1"] \ar[draw=none]{ld}[auto=false, rotate=-45]{\Downarrow}\\
			\fun(\mc{D}_1, \mc{C}) \ar[r, "\lim"] & \mc{C} 
		\end{cd} 
		This yields the desired functor upon precomposing by the embedding 
		\[
			\fun(\mc{D}_1, \mc{D}_2) \hra \fun(\mc{D}_2, \mc{C}) \times \fun(\mc{D}_1, \mc{D}_2)
		\]
		which sends $F \mapsto (G, F)$. 
	\end{proof}
	
	\subsubsection{} Using the previous lemma, we discover that $M$ satisfies some new functoriality which was not visible prior to taking the limit. Morally, this functoriality corresponds to the semigroup multiplication $\Ran(X)^S \xrightarrow{\on{mult}} \Ran(X)$, which is not visible at any finite stratum of $\Ran(X)$. 
	
	\begin{lem}\label{lem-act} 
		Let $I$ be a finite set. 
		\begin{enumerate}[label=(\roman*)]
			\item There is a functor $\fset_{\on{n.e.}}^{\on{surj}} \to \mc{C}_{M(\{*\} \sqcup I \sqcup \{1\})/}$ which sends a set $S$ to a map
			\[
				M(\{*\} \sqcup I \sqcup \{1\}) \xrightarrow{\phi_S} M(\{*\} \sqcup I \sqcup S). 
			\]
			\item For any $S$, let $c : S \to \{1\}$ be the unique map. The composition 
			\begin{cd}[column sep = 0.7in] 
				M(\{*\} \sqcup I \sqcup \{1\}) \ar[r, "\phi_S"] & M(\{*\} \sqcup I \sqcup S) \ar[r, "M(\id \sqcup c)"] & M(\{*\} \sqcup I \sqcup \{1\})
			\end{cd}
			is homotopic to the identity map, and these data are compatible with the automorphisms of $M(\{*\} \sqcup I \sqcup S)$ induced by permutations of $S$.  
		\end{enumerate}
	\end{lem}
	\begin{proof}
		Applying Lemma~\ref{lims} to the functor $G\iota((I \sqcup \{1\}) \times (-))$ yields a functor 
		\[
			\fun(\fsetsurjne, \fsetsurjne) \to \mc{C}_{M(\{*\} \sqcup I \sqcup \{1\})/}. 
		\]
		Precompose this by the functor $\fset_{\on{n.e.}}^{\on{surj}} \to \fun(\fsetsurjne, \fsetsurjne)$ which sends $S \mapsto S \times (-)$. (The functor $S \times (-)$ lands inside $\fsetsurjne$ only if $S$ is nonempty.) We obtain a functor 
		\[
			\fset_{\on{n.e.}}^{\on{surj}} \to \mc{C}_{M(\{*\} \sqcup I \sqcup \{1\})/}
		\]
		which sends a nonempty finite set $S$ to a map 
		\e{
			M(\{*\} \sqcup I \sqcup \{1\}) \to & M(\{*\} \sqcup ((I \sqcup \{1\}) \times S)) \\
			&\quad = M(\{*\} \sqcup (I \times S) \sqcup S). 
		} 
		Because $\{1\}$ is terminal in $\fset_{\on{n.e.}}^{\on{surj}}$, the maps $M(\{*\} \sqcup (I \times S) \sqcup S) \xrightarrow{(\id_I \times c) \sqcup \id_S} M(\{*\} \sqcup I \sqcup S)$ are functorial in $S$, and post-composing the above map by this one yields the map $\phi_S$. 
		
		To prove part (ii), note that part (i) identifies the composition in question with $\phi_{\{1\}}$, and it does so functorially in $S$. The claim follows because $\phi_{\{1\}} \simeq \id$. 
	\end{proof}

	\subsubsection{} We are now ready to complete the proof of Theorem~\ref{thm2}: 
	
	\begin{lem}\label{lem-induct} 
		For every $d$, the functor $\tau^{\le d} M$ is isomorphic to the constant functor with value $\tau^{\le d} M(\{*\})$. 
	\end{lem}
	\begin{proof}
		Since limits are left exact, our assumption ($\bullet$) implies that the essential image of $M$ lies in $\mc{C}^{\ge m}$. This proves the lemma for $d < m$. By induction, we may assume that $d \ge m$ and that the result has been proven for all smaller $d$. We have a map of exact triangles of functors $\fset_* \to \mc{C}$ as follows: 
		\begin{cd}
			\tau^{\le d-1} M \ar[r] \ar{d}[swap, rotate=90, anchor=south]{\sim} & \tau^{\le d} M \ar[r] \ar[d] & \tau^{\ge d}\tau^{\le d} M \ar[d] \\
			\tau^{\le d-1} \ul{M(\{*\})} \ar[r] & \tau^{\le d} \ul{M(\{*\})} \ar[r] & \tau^{\ge d} \tau^{\le d} \ul{M(\{*\})}
		\end{cd}
		Our desired statement is that the middle vertical arrow is a natural isomorphism. Therefore, it suffices to prove that the right vertical arrow is a natural isomorphism. In this way, we reduce to studying functors landing in the abelian 1-category $\mc{A} := \mc{C}^{\heartsuit}$. 
		
		Consider the functor $N : \fset_* \to \mc{A}$ which corresponds to $\tau^{\ge d} \tau^{\le d} M$. We prove that $N$ is a polynomial functor of degree $\le n$ in the sense of Definition~\ref{def-poly}. Let $\Xi := \Xi_{(I_b)_{b \in B}}$ be any special hypercube in $\fset_*$ with $|B| > n$, so that $M(\Xi)$ is a limit diagram by Lemma~\ref{m-excisive}. Consider the exact triangle of functors $\fset_* \to \mc{C}$ as follows: 
		\begin{cd}
			\tau^{\le d-1} M(\Xi) \ar[r] & M(\Xi) \ar[r] & \tau^{\ge d} M(\Xi)
		\end{cd}
		By the inductive hypothesis, $\tau^{\le d-1} M(\Xi)$ is a constant diagram, so it is a limit diagram.\footnote{If it were not a constant diagram, the proof would break at this point because, although it is a limit diagram in $\mc{C}^{\le d-1}$, the inclusion $\mc{C}^{\le d-1} \hra \mc{C}^{\le d}$ is a left adjoint and does not preserve limits.} Since $\mc{C}$ is stable, and limits commute with limits, we conclude that $\tau^{\ge d} M(\Xi)$ is also a limit diagram. Since $\tau^{\le d}$ preserves limits, this implies that $\tau^{\le d} \tau^{\ge d} M(\Xi)$ is a limit diagram in $\mc{C}^{\le d}$. Now it follows from Remark~\ref{rmk-compare} that $N$ is polynomial of degree $\le n$. 
		
		By Lemma~\ref{lem-poly}, we have $N = \on{Ind}(F)$ for some $F : \fset^{\on{surj}} \to \mc{A}$ which sends $J \mapsto 0$ if $|J| > n$. We shall prove the following statement by downward induction on $w \ge 1$: if $|J| = w$, then $F(J) = 0$. Thus, we fix $w \ge 1$ and assume that the statement has been proved for all integers larger than $w$. 
		
		Let $I$ be a finite set with $|I| = w-1$. In view of Lemma~\ref{lem-act}, we know that for $S \in \fset_{\on{n.e.}}^{\on{surj}}$ there is a map 
		\[
			N(\{*\} \sqcup I \sqcup \{1\}) \xrightarrow{\phi_S} N(\{*\} \sqcup I \sqcup S)
		\]
		which is functorial in $S$. By definition of $\on{Ind}(F)$, this map rewrites as 
		\[
			\bigoplus_{T \subset I \sqcup \{1\}} F(T) \xrightarrow{\phi_{S}} \bigoplus_{T' \subset I \sqcup S} F(T'). 
		\]
		Define $\phi_{S'} := \pr \circ \phi_{S} \circ i$ as shown: 
		\begin{cd}
			\displaystyle\bigoplus_{T \subset I \sqcup \{1\}} F(T) \ar[r, "\phi_{S}"]  & \displaystyle\bigoplus_{T' \subset I \sqcup S} F(T') \ar[d, twoheadrightarrow, "\pr"]\\
			F(I \sqcup \{1\}) \ar[r, "\phi_{S}'"] \ar[u, hookrightarrow, "i"] & \displaystyle\bigoplus_{\substack{I \subset T' \subset I \sqcup S \\ |T'| = w}} F(T') 
		\end{cd}
		In this digram, $i$ is the inclusion of the summand $F(I \sqcup \{1\})$, and $\pr$ is the projection onto the indicated summands. 
		
		In order to formulate an analogous functoriality of $\phi_S'$ with respect to $S$, we first explain the functoriality of the target of $\phi_{S'}$ with respect to $S$. Given a surjection $\xi : S_1 \sra S_2$, we claim that there is a commutative diagram 
		\begin{cd}[column sep = 0.7in]
			N(\{*\} \sqcup I \sqcup S_1) \ar[dash]{d}[rotate=90, anchor=north]{\sim} \ar[r, "{N(\id_I \sqcup \xi)}"] & N(\{*\} \sqcup I \sqcup S_2) \ar[dash]{d}[rotate=90, anchor=north]{\sim} \\
			\displaystyle\displaystyle\bigoplus_{T' \subset I \sqcup S_1} F(T') \ar[r] \ar[d, twoheadrightarrow, "\pr"] & \displaystyle\bigoplus_{T'' \subset I \sqcup S_2} F(T'') \ar[d, twoheadrightarrow, "\pr"] \\
			\displaystyle\bigoplus_{\substack{I \subset T' \subset I \sqcup S_1 \\ |T'| = w}} F(T') \ar[r, "\eta_\xi"] & \displaystyle\bigoplus_{\substack{I \subset T'' \subset I \sqcup S_2 \\ |T''| = w}} F(T'')
		\end{cd}
		I.e., the map $N(\id_I \sqcup \xi)$ descends to a map $\eta_\xi$ between the quotients. This follows from the definition of $\on{Ind}(F)$. Indeed, the $T''$-coordinate of the map $N(\id_I \sqcup \xi)$ only depends on the summands $F(T')$ with $(\id_I \sqcup \xi)(T') = T''$, and if $T''$ satisfies the conditions $I \subset T''$ and $|T''| = w$, then $(\id_I \sqcup \xi)(T') = T''$ implies that $I \subset T'$ and $|T'| \ge w$. Since the summands for $|T'| > w$ are zero (by the inductive hypothesis), the claim follows. The upshot is that these quotients are functorial in $S$. 
		
		Now, the functoriality of the maps $\phi_S$ with respect to $S$ implies the functoriality of the maps $\phi_{S}'$ with respect to $S$. The previous paragraph explains how the target of $\phi_S'$ is functorial with respect to $S$. 
		
		\begin{rmk}
			Note that the bottom row of the previous diagram rewrites as 
			\[
			\bigoplus_{s_1 \in S_1} F(I \sqcup \{s_1\}) \xrightarrow{\eta_\xi} \bigoplus_{s_2 \in S_2} F(I \sqcup \{s_2\})
			\]
			because the subsets $T' \subset I \sqcup S_1$ with $I\subset T'$ and $|T'| = w$ are in bijection with elements of $S_1$ (this uses that $|I| = w-1$) and similarly for $S_2$ in place of $S_1$. Furthermore, the definition of $\on{Ind}(F)$ shows that the matrix coefficient $F(I \sqcup \{s_1\}) \to F(I \sqcup \{s_2\})$ of $\eta_\xi$ is an isomorphism if $s_2 = \xi(s_1)$ and zero otherwise.
		\end{rmk}
		 		
		Next, for notational convenience, we write $V := F(w\text{-element set})$. This is safe because the subsequent manipulations do not involve any permutations on the $w$-elements sets in question; indeed, $w-1$ of the elements are `pinned down' by their identification with elements of $I$. Thus, the map $\phi_S'$ rewrites as 
		\[
			V \xrightarrow{\phi_S'} V^{\oplus S}
		\]
		By the remark, permutations of $S$ act on $V^{\oplus S}$ by permuting the summands, so the functoriality of $\phi_S'$ with respect to $S$ implies that $\phi_{S}'$ lands inside the subgroup $(V^{\oplus S})^{\Sigma_S}$ of symmetric group invariants. 
		
		Next, let $r \ge 3$ and consider the surjection $\xi : [r] \to [r-1]$ defined by $\xi(q) = \min(r-1, q)$. In this case, functoriality of $\phi_{S}'$ with respect to $S$ yields this commutative diagram:  
		\begin{cd}
			V \ar[r, "\phi_{[r]}'"] \ar[rd, swap, "\phi_{[r-1]}'"] & V^{\oplus r} \ar[d, "\id \oplus \cdots \oplus \id \oplus \on{sum}"] \\
			& V^{\oplus (r-1)}
		\end{cd}
		The description of the vertical map follows from the remark. By the previous paragraph, the rightward maps land inside $(V^{\oplus r})^{\Sigma_r}$ and $(V^{\oplus (r-1)})^{\Sigma_{r-1}}$, respectively. These facts imply that the maps $\phi_{[r]}'$ and $\phi_{[r-1]}'$ are equal to zero. This is because any $v \in V$ maps to some $(v_1, \ldots, v_1) \in V^{\oplus r}$ and then to $(v_1, \ldots, v_1, 2v_1) \in V^{\oplus (r-1)}$, and since this is cyclically invariant we conclude that $v_1 = 0$. 
		
		Lastly, we apply the special case of functoriality summarized in Lemma~\ref{lem-act}(ii). This says that the composition
		\[
			N(\{*\} \sqcup I \sqcup \{1\}) \xrightarrow{\phi_{[r]}} N(\{*\} \sqcup I \sqcup [r]) \xrightarrow{N(\id \sqcup c)} N(\{*\} \sqcup I \sqcup \{1\})
		\]
		is equal to the identity map. Passing to $\phi_{[r]}'$ as before, this yields a diagram 
		\[
			V \xrightarrow{\phi_{[r]}'} V^{\oplus r} \xrightarrow{\on{sum}} V
		\]
		where the description of the second map as the addition map follows from the remark, and where the composition is $\phi_{[1]}' = \on{id}_V$. On the other hand, assuming $r \ge 2$, this composition equals zero because $\phi_{[r]}' = 0$ (by the previous paragraph), so we conclude that $V = 0$. This implies that $F$ vanishes on all $w$-element sets, so the inductive step is proved. 
		
		We have just shown that $N$ is isomorphic to the constant functor with value $N(\{*\}) \simeq F(\emptyset)$, so the analogous statement for $\tau^{\ge d} \tau^{\le d} M$ also holds.  
	\end{proof}
	
	By Lemma~\ref{induct}, the map $M(\{*, 1\}) \to M(\{*\})$ becomes an isomorphism upon applying $\tau^{\le d}$ for arbitrarily large $d$. Since the $t$-structure on $\mc{C}$ is right-complete, we conclude that this map is itself an isomorphism. This concludes the proof of Theorem~\ref{thm2} 
\end{myproof}

\subsection{Analogues for finite limits}  \label{ssec-fin}

We prove two analogous results about `vanishing of the limit' which apply to functors defined only on the subcategory of pointed finite sets whose size is bounded above by some integer $N$. Although only the first result Lemma~\ref{prop1-fin} will be used in this paper, the second result Lemma~\ref{thm2-fin} may be useful in other situations. In general, even if one has a functor defined on all of $\fset_*$, one may nevertheless want to restrict to looking at finite limits in order to use commutativity with filtered colimits. 

A heuristic reason why one should expect analogues for finite limits is that any elementary proof of `vanishing of the limit,' such as the second proof of Lemma~\ref{prop1}, will show that each term in a compatible system of elements is zero by using compatibilities which come in the form of equations. And it should only require finitely many equations to show that a particular element is zero, because linear algebra does not allow taking infinite linear combinations. Therefore, the same proof would apply to any \emph{partially defined} system of elements as long as the partial system includes those elements which appear in the finitely many equations; this would allow one to deduce a partial vanishing result for finite limits. Lemma~\ref{thm2-fin} is included to illustrate that a partial vanishing result can hold even in the absence of an elementary proof. 

\subsubsection{}

Define $\fset_{*, \le N}$ to be the full subcategory of $\fset_*$ consisting of sets $\{*\} \sqcup I$ for which $|I| \le N$. (Note that the set $\{*\} \sqcup I$ is considered to have size $|I|$, not $|I|+1$. This agrees with the interpretation of $\fset_*$ mentioned in Remark~\ref{ssec-fset}.) Similarly, let $\fset_{\le N}$ be the full subcategory of $\fset$ consisting of sets $I$ with $|I| \le N$. 

\subsubsection{} Let $\mc{A}$ be an abelian category. 

\begin{prop} \label{prop-fin1} 
	Let $N \ge 0$ be a fixed integer. 
	\begin{enumerate}[label=(\roman*)]
		\item We have inverse equivalences $\on{Ind} : \fun(\fsetsurj_{\le N}, \mc{A}) \rightleftarrows \fun(\fset_{*, \le N}, \mc{A}) : \on{Prim}$. 
		\item These functors fit into a commuting diagram with the ones from~\ref{def-indprim}: 
		\begin{cd}
			\fun(\fsetsurj, \mc{A}) \ar[r, shift left = 1.5, "\on{Ind}"]  \ar[d, "\on{Res}"] & \fun(\fset_{*}, \mc{A}) \ar[l, shift left = 0.5, "\on{Prim}"]  \ar[d, "\on{Res}"] \\
			\fun(\fsetsurj_{\le N}, \mc{A}) \ar[r, shift left = 1.5, "\on{Ind}"] & \fun(\fset_{*, \le N}, \mc{A}) \ar[l, shift left = 0.5, "\on{Prim}"] 
		\end{cd}
		\item Let $G : \fset_{*, \le N} \to \mc{A}$ be a functor, and let $n \ge 0$ be an integer. Then the four conditions of Lemma~\ref{lem-poly} apply to $G$ (after restricting to  $I$ such that $|I| \le N$) and they are equivalent. 
	\end{enumerate} 
\end{prop}
\begin{proof}
	The proofs are entirely similar to the analogous statements for unbounded finite sets. This follows by noting that the definitions of Ind and Prim (\ref{def-indprim}), the proof that they are mutually inverse equivalences (\ref{prop-indprim}), and the proof of Lemma~\ref{lem-poly} are carried out by considering a fixed set $I$ and comparing it with strictly smaller subsets. 
\end{proof}

\subsubsection{} \label{def-poly-fin}

The definitions of `polynomial' and `$n$-excisive' can be applied to a functor which is defined only on $\fset_{*, \le N}$ because those definitions relate the value of the functor on a given set $I$ to its value on smaller sets (subsets of $I$), but not larger sets. 

\begin{defn}
	Let $N \ge 0$ be a fixed integer. 
	\begin{enumerate}[label=(\roman*)]
		\item Let $\mc{A}$ be an abelian category, and let $G : \fset_{*, \le N} \to \mc{A}$ be a functor. We say that $G$ is \emph{polynomial of degree $\le n$} if it satisfies the equivalent conditions of Proposition~\ref{prop-fin1}(iii) for that $n$. 
		\item Let $\mc{C}$ be an $(\infty, 1)$-category, and let $G : \fset_{*, \le N} \to \mc{C}$ be a functor. We say that $G$ is \emph{$n$-excisive} if $G(\Xi_{(I_b)_{b \in B}})$ is weakly cartesian, for any tuple $(I_b)_{b \in B}$ of sets such that their disjoint union is $\le N$. 
	\end{enumerate} 
\end{defn}

\begin{rmk}
	Any functor $G : \fset_{*, \le N} \to \mc{A}$  (resp.\  $\mc{C}$) is polynomial of degree $\le N$ (resp.\ $N$-excisive) for trivial reasons. 
\end{rmk}

\subsubsection{} We arrive at our first vanishing result for finite limits. Note that the first proof of Proposition~\ref{prop1} does not obviously apply here.

\begin{lem}\label{prop1-fin}
	Let $\mc{A}$ be a complete abelian category, and let $n, \ell, N$ be nonnegative integers with $n + 2 \le \ell \le N$. If $G : \fset_{*, \le N} \to \mc{A}$ is polynomial of degree $\le n$, then the map 
	\[
		\lim_{\fset^{\on{surj}}_{\on{n.e.}, \le \ell}} G \iota \to G(\{*\})
	\]
	is an isomorphism. 
\end{lem}
\begin{proof}
	Since a functor is polynomial of degree $\le n$ only if it is polynomial of degree $\le \ell - 2$, we may replace $n$ by $\ell-2$ and thereby assume that $\ell = n+2$. The assumption that $n +2 \le N$ is still in force. 	
	
	We copy the second proof of Proposition~\ref{prop1}. As before, we reduce to the case in which $\mc{A} = \on{AbGrp}$, so an element of the limit is a compatible family of elements $a_I \in G(\{*\} \sqcup I)$ for $|I| \le n+2$. Writing $G = \on{Ind}(F)$ as in Proposition~\ref{prop-fin1}, the element $a_I$ corresponds to a tuple of elements $b_{S, I} \in F(S)$ for $S \subset I$. Our goal is to show that $b_{S, I} = 0$ if $S$ is nonempty. 
	
	Considering permutations of $I$ shows that $b_{S, I}$ depends only on the cardinality of $S$, so we write $b_{S, I}$ as $b_{m, I}$ when $|S| = m$. We use downward induction on $m$ to show that $b_{m, I} = 0$ for all $m \ge 1$ and $|I| \le n+2$. The claim for $m \ge n+1$ follows from our assumption that $G$ is polynomial of degree $\le n$ because this means that $F(S) = 0$ for $|S| \ge n+1$. Therefore, we may assume that $1 \le m \le n$ and that the claim has been proven for all larger $m$. For any $d \le n+1$, consider the map 
	\[
	\{*\} \sqcup [d+1] \xrightarrow{\xi} \{*\} \sqcup [d]
	\]
	which sends $i \mapsto \min(i, d)$. Since $d +1 \le n+2 \le N$, we may apply $G$ to these sets.  If $d > m$, then there exist $m$-element subsets of $[d]$ which do (resp.\ do not) contain $d$, and the same argument as in Proposition~\ref{prop1} shows that 
	\[
	b_{m, [d]} = b_{m, [d+1]} = 2b_{m, [d+1]}, 
	\]
	so $b_{m, [d+1]} = b_{m, [d]} = 0$. Take $d = m+1, \ldots, n+1$ (since $m \le n$, this list is nonempty) to conclude that $b_{m, I} = 0$ for all $I$ with $m < |I| \le n+2$. In particular, $b_{m, [m+1]} = 0$. It remains to show that $b_{m, [m]} = 0$. From the same map for $d = m$, we conclude that $b_{m, [m]} = 2b_{m, [m+1]}$, and the claim follows. 
\end{proof}

\subsubsection{} It is interesting to ask whether the (non-elementary) proof of Theorem~\ref{thm2} can be modified to give a vanishing result for finite limits. This is indeed possible: 

\begin{lem} \label{thm2-fin}
	Let $\mc{C}$ be a complete stable category. Let $G : \fset_{*, \le N} \to \mc{C}$ be a $1$-excisive functor with $G(\{*\}) = 0$. Then this map is zero:\footnote{The notation $\le N/3$ is shorthand for $\le \lfloor N/3 \rfloor$.}  
	\[
		\lim_{\fset^{\on{surj}}_{\on{n.e.}, \le N}} G \iota  \to \lim_{\fset^{\on{surj}}_{\on{n.e.}, \le N/3}} G \iota
	\]
	The map is the one obtained by applying Lemma~\ref{lims} to $\fset^{\on{surj}}_{\on{n.e.}, \le N/3} \hra \fset^{\on{surj}}_{\on{n.e.}, \le N}$. 
\end{lem}
\begin{proof}
	We emulate the proof of Theorem~\ref{thm2}. Since $G(\{*\}) = 0$, the $1$-excisive condition says that $G$ is monoidal, in the sense that it sends wedge products (of pointed finite sets) to direct sums (in $\mc{C}$). First, consider the composition 
	\begin{cd}
		\fset_{*, \le 3} \times \fset^{\on{surj}}_{\on{n.e.}, \le N/3} \ar[r, "\id \times \iota"] & \fset_{*, \le 3} \times \fset_{*, \le N/3} \ar[r, "\wedge"] & \fset_{*, \le N} \ar[r, "G"] & \mc{C}
	\end{cd}
	and take its limit with respect to the second coordinate to get a functor $M : \fset_{*, \le 3} \to \mc{C}$. For $S$ such that $|S|\le 3$, we have 
	\e{
		M(\{*\} \sqcup S) &\simeq \lim_{I \in \fset^{\on{surj}}_{\on{n.e.}, \le N/3}} G\iota(S \times I) \\
		&\simeq \lim_{I \in \fset^{\on{surj}}_{\on{n.e.}, \le N/3}} (G\iota(I))^{\oplus S} \\
		&\simeq \left(\lim_{I \in \fset^{\on{surj}}_{\on{n.e.}, \le N/3}} G\iota(I)\right)^{\oplus S}
	} 
	where we have used the 1-excisive property of $G$. In particular, $M(\{*\}) = 0$. In this way, we see that $M$ is 1-excisive (cf.\ Lemma~\ref{m-excisive}). In particular, this implies that the map $M(\{*\} \sqcup S) \to M(\{*, 1\})$ corresponding to the unique map $S \sra \{1\}$ identifies with the map 
	\begin{equation}\label{diamond}
		\left(\lim_{I \in \fset^{\on{surj}}_{\on{n.e.}, \le N/3}} G\iota(I)\right)^{\oplus S} \xrightarrow{\on{sum}} \lim_{I \in \fset^{\on{surj}}_{\on{n.e.}, \le N/3}} G\iota(I)\tag{$\diamondsuit$}
	\end{equation}
	
	To ease the notation, we write $\lim_{\le N} G\iota := \lim_{\fset^{\on{surj}}_{\on{n.e.}, \le N}} G \iota$. Applying Lemma~\ref{lims} to the functor $G\iota$ yields a functor 
	\[
		\fun(\fset^{\on{surj}}_{\on{n.e.}, \le N/3}, \fset^{\on{surj}}_{\on{n.e.}, \le N}) \to \mc{C}_{(\lim_{\le N} G\iota) /}
	\]
	Precomposing by the functor $\fset^{\on{surj}}_{\on{n.e.}, \le 3} \to \fun(\fset^{\on{surj}}_{\on{n.e.}, \le N/3}, \fset^{\on{surj}}_{\on{n.e.}, \le N})$ which sends $S \mapsto S \times (-)$, we obtain a functor 
	\[
		\fset^{\on{surj}}_{\on{n.e.}, \le 3} \to \mc{C}_{(\lim_{\le N} G\iota) /}
	\]
	which sends a finite set $S$ with $|S| \le 3$ to a map 
	\[
		\lim_{\le N} G\iota \xrightarrow{\phi_S} M(\{*\} \sqcup S),
	\]
	and this is functorial in $S$ (cf.\ Lemma~\ref{lem-act}). 
	
	We want to prove that the map $\lim_{\le N} G\iota \xrightarrow{\phi_{[1]}} M(\{*, 1\}) \simeq \lim_{\le N/3} G \iota$ is homotopic to zero. We apply the functor of the previous paragraph to the sequence of maps 
	\[
		[3] \xrightarrow{\xi_3} [2] \xrightarrow{\xi_2} [1]
	\]
	where $\xi_3(i) = \min(i, 2)$. The resulting diagram is 
	\begin{cd}
		\lim_{\le N} G\iota \ar[r, "\phi_{[3]}"] \ar[rd, swap, "\phi_{[2]}"]  \ar[rdd, bend right = 20, swap, "\phi_{[1]}"]  & (\lim_{\le N/3} G \iota)^{\oplus 3} \ar[d, " M(\xi_3)"] \\
		& (\lim_{\le N/3} G \iota)^{\oplus 2} \ar[d, "M(\xi_2)"] \\
		& \lim_{\le N/3} G \iota
	\end{cd}
	From the reasoning of (\ref{diamond}), we know that $M(\xi_3) = \id \oplus \on{sum}$ and $M(\xi_2) = \on{sum}$.  By considering permutations on the set $S = [3]$, we conclude that $\phi_{[3]} = (\psi_3, \psi_3, \psi_3)$ for some map $\psi_3$. Therefore $\phi_{[2]} = (2\psi_3, \psi_3)$. On the other hand, considering permutations on the set $S = [2]$ shows that $\phi_{[2]} = (\psi_2, \psi_2)$ for some map $\psi_2$. This yields a homotopy from $\psi_3$ to $2\psi_3$, which shows that $\psi_3$ is homotopic to zero. Thus, $\psi_2$ is also homotopic to zero. Thus $\phi_{[1]} = 2\psi_{2}$ is homotopic to zero, as desired. 
\end{proof}

\subsubsection{Remark} \label{further-fin}

We explain a sense in which the result of Lemma~\ref{prop1-fin} is the best possible. In that lemma, the main hypothesis is that degree $n$ of the functor is two less than the bound $\ell$ for the finite limit. Therefore, it is natural to ask if any vanishing result is possible for $\ell \le n$. (We shall not discuss the case $n = \ell-1$.) For sake of definiteness, we restrict attention to the following template:
\begin{itemize}
	\item[(Q1)] Fix positive integers $\ell'\le \ell \le n \le N$. For any abelian category $\mc{A}$, if $G : \fset_{*, \le N} \to \mc{A}$ is a polynomial functor of degree $\le n$ with $G(\{*\}) = 0$, then the map 
	\[
		\lim_{\fset^{\on{surj}}_{\on{n.e.}, \le \ell}} G \iota \to \lim_{\fset^{\on{surj}}_{\on{n.e.}, \le \ell'}} G \iota
	\]
	is equal to zero. 
\end{itemize}
Our question, then, is whether there exist $\ell', \ell, n, N$ such that (Q1) is true. 

The main observation is that of Remark~\ref{def-poly-fin}: every functor $\fset_{*, \le \ell} \to \mc{A}$ is polynomial of degree $\ell$. This can be amplified to prove the following: 

\begin{claim}
	For fixed $\ell', \ell, n, N$, the statement (Q1) is equivalent to the following statement (for the same $\ell', \ell$): 
	\begin{itemize}
		\item[(Q2)] Consider positive integers $\ell' \le \ell$. For any abelian category $\mc{A}$ and any functor $G : \fset_{*, \le \ell} \to \mc{A}$ with $G(\{*\}) = 0$, the map 
		\[
		\lim_{\fset^{\on{surj}}_{\on{n.e.}, \le \ell}} G\iota \to \lim_{\fset^{\on{surj}}_{\on{n.e.}, \le \ell'}} G\iota
		\]
		is equal to zero. 
	\end{itemize}
\end{claim}
\begin{proof}
	We clearly have (Q2) $\Rightarrow$ (Q1), so we focus on proving the converse. Assume (Q1) holds. Let $G : \fset_{*, \le \ell} \to \mc{A}$ be any functor, which we may express as $G = \on{Ind}(F)$ for some functor $F : \fset^{\on{surj}}_{\le \ell} \to \mc{A}$. Let $\wt{F} : \fset^{\on{surj}}_{\le N} \to \mc{A}$ be the extension of $F$ which sends $I \mapsto 0$ whenever $|I| > \ell$. Then $\on{Ind}(\wt{F}) : \fset_{*, \le N} \to \mc{A}$ extends $G$. Furthermore, $\on{Ind}(\wt{F})$ is polynomial of degree $\le \ell$ which is $\le n$, so we may apply (Q1) to conclude that the map
	\[
		\lim_{\fset^{\on{surj}}_{\on{n.e.}, \le \ell}} \on{Ind}(\wt{F})\iota \to \lim_{\fset^{\on{surj}}_{\on{n.e.}, \le \ell'}} \on{Ind}(\wt{F})\iota
	\]
	is zero. But $\on{Ind}(\wt{F})$ restricted to $\fset_{*, \le \ell}$ coincides with $G$, and the claim follows. 
\end{proof}

The example of Remark~\ref{counter} is easily modified to give a counterexample to (Q2) for any $\ell', \ell$. Indeed, one ensures $G(\{*\}) = 0$ by taking $G = \on{Ind}(F)$ where $F : \fset^{\on{surj}} \to \on{AbGrp}$ is given by $\emptyset \mapsto 0$ and $F|_{\fset^{\on{surj}}_{\on{n.e.}}} = \ul{\BZ}$. The element of $\lim G\iota$ constructed in the remark maps to a nonzero element of $\lim_{\le m} G\iota$ for every integer $m \ge 1$. Taking $m = \ell, \ell'$, we obtain a commutative diagram 
\begin{cd}
	\lim G \iota \ar[d] \ar[rd] \\
	\lim_{\le \ell} G \iota \ar[r] & \lim_{\le \ell'} G\iota
\end{cd}
which implies that the horizontal map cannot be zero.

By the claim, we also conclude that (Q1) is not true for any positive integers $\ell' \le \ell \le n \le N$. Therefore, whenever $\ell \le n$, it is not reasonable to expect any general vanishing result for limits over $\fset^{\on{surj}}_{\on{n.e.}, \le \ell}$ of arbitrary polynomial functors of degree $\le n$. In particular, the vanishing statement considered in Remark~\ref{exercise} does not follow from a general theorem; the fact that it is true when $R$ is a ring over a characteristic zero field is a phenomenon specific to that functor.

\subsubsection{Remark}\label{further-fin2} 

One could ask whether there is an analogue of Theorem~\ref{thm2} for $n$-excisive functors with $n \ge 2$. The proof of Theorem~\ref{thm2} proceeds by reducing to the case of functors landing in an abelian category and then carrying out an elementary argument which shows that each successive homogeneous piece of a polynomial functor is zero. This is a poor man's version of the method, employed in Goodwillie calculus, of understanding a functor by studying its homogeneous pieces. If one could prove Theorem~\ref{thm2} without resorting to $t$-structures, then it is likely that one could answer this question affirmatively. 

\section{Interlude: Sheaves on the Ran space} \label{ran-def} 

In this short section, we review several equivalent definitions of the Ran space which appear in the literature,  discuss the relation between various definitions of `structure on the Ran space,' and explain what conventions we use in this paper. The definitions which we actually use appear in~\ref{prestacks} and \ref{defs}. 

\subsection{Definitions of $\Ran(X)$} \label{sec-defs} 

\subsubsection{Prestacks} \label{prestacks}

Let $\on{Sch}^{\on{aff}, \on{ft}}_{k}$ denote the category of affine finite-type schemes over $k$. This category is essentially small. There are at least four notions of `prestack' which appear in the literature: 
\begin{enumerate}[label=(\roman*)]
	\item Prestacks valued in sets: $\fun(\on{Sch}^{\on{aff, ft, op}}_{k}, \on{Set})$
	\item Prestacks valued in groupoids: $\fun(\on{Sch}^{\on{aff, ft, op}}_{k}, \on{Grpd})$
	\item Prestacks valued in $\infty$-groupoids: $\fun(\on{Sch}^{\on{aff, ft, op}}_{k}, \infty\!\on{-Grpd})$
	\item Prestacks valued in $(\infty, 1)$-categories: $\fun(\on{Sch}^{\on{aff, ft, op}}_{k}, (\infty, 1)\!\on{-Cat})$ 
\end{enumerate}
Here (iv) means the $(\infty, 2)$-category of functors (i.e. pseudofunctors) from $\on{Sch}^{\on{aff, ft}}_k$ to the $(\infty, 2)$-category of $(\infty, 1)$-categories. The datum of such a functor is equivalent to the datum of a functor from $\on{Sch}^{\on{aff, ft}}_k$ to the $(\infty, 1)$-category of $(\infty, 1)$-categories (obtained by discarding all natural transformations which are not natural isomorphisms). 

In this paper, we shall denote (iii) by $\on{PreStk}_k$ and call its objects `prestacks.' Point (i) is the usual notion of `presheaf on the big Zariski site of finite type schemes over $k$' which is used in~\cite[0.3.1]{z}. Point (iii) is the notion of `prestack' used in~\cite[0.4.1]{g}. And point (iv) is the notion of `prestack' used in~\cite[2.3.8]{gl}. Note that there are canonical functors (i) $\to$ (ii) $\to$ (iii) $\to$ (iv). 

Let $X^{(-)} : \fsetsurjneop \to \on{Sch}^{\on{aff, ft}}_k$ be the functor which sends $I \mapsto X^I$ and which sends a map $\xi : I \sra J$ to the generalized diagonal map $\Delta_\xi : X^J \hra X^I$ whose $i$-th coordinate is the $\xi(i)$-th coordinate projection of $X^J$, for all $i \in I$. 

One could reasonably define $\Ran^{(a)}(X) := \colim_{\fsetsurjneop} X^{(-)}$ where the colimit is evaluated in the category $(a)$ for any choice of $a = $ i, ii, iii. However, we shall see in~\ref{ran-set} that the versions for (i), (ii), and (iii) are isomorphic as prestacks valued in $\infty$-groupoids. In view of this, we define $\Ran(X)$ to be any of these three equivalent prestacks. On the other hand, see~\ref{iv} for a discussion about evaluating the colimit in the category (iv). 

\subsubsection{} \label{ran-set} 
The colimit definition of $\Ran(X)$ does not depend on whether the colimit is evaluated in the category (i), (ii), or (iii). This follows from the next lemma. 
\begin{lem}
	The colimit $\colim_{\fsetsurjneop} X^{(-)}$ evaluated in $\fun(\on{Sch}^{\on{aff, ft, op}}_{k}, \infty\!\on{-Grpd})$ is a functor whose values in $\infty\!\on{-Grpd}$ are equivalent to sets. 
\end{lem}
\begin{proof}
	Since colimits in functor categories are computed pointwise, we want to show that
	\[
		\colim_{I \in \fsetsurjneop} \on{Maps}_{\on{Sch}_k}(Y, X^I), 
	\]
	when evaluated in $\infty$-Grpd, is equivalent to a set. To see this, write $S = \on{Maps}_{\on{Sch}_k}(Y, X)$ and note that 
	\e{
		\{\on{pt}\} \sqcup \colim_{I \in \fsetsurjneop} \on{Maps}_{\on{Sch}_k}(Y, X^I) &\simeq \{\on{pt}\} \sqcup \colim_{I \in \fsetsurjneop} S^I \\
		&\simeq \colim_{I \in \fsetsurjop} S^I. 
	} 
	Let $\mc{C}$ be the category whose objects are pairs $(I, \phi)$ where $\phi : I \to S$ is any map (not necessarily surjective) and whose morphisms from $(I_1, \phi_1)$ to $(I_2, \phi_2)$ are commutative diagrams 
	\begin{cd}
		I_1 \ar[rr, twoheadleftarrow] \ar[rd, swap, "\phi_1"] & & I_2 \ar[ld, "\phi_1"] \\
		& S
	\end{cd}
	where the horizontal map is surjective.\footnote{Remark: $\mc{C}$ is the fiber over $Y$ of the category defined in~\cite[Def.\ 2.4.9]{gl} and denoted `$\Ran(X)$' therein.} The forgetful functor $F : \mc{C} \to \fsetsurjop$ which sends $(I, \phi) \mapsto I$ is a cocartesian fibration whose fiber $\mc{C}|_I$ over $I \in \fsetsurjop$ identifies with the set $S^I$. The previous colimit rewrites as follows: 
	\e{
		\colim_{I \in \fsetsurjop} S^I &\simeq \colim_{I \in \fsetsurjop} \colim_{\mc{C}|_I} \on{pt} \\
		&\simeq \colim_{\mc{C}} \on{pt}, 
	} 
	where the last line uses Thomason's theorem for colimits, see~\cite[Prop.\ 26.5]{w} for example. 
	
	Next, view the power set $\on{Pow}(S)$ as a discrete category and define the functor $G : \on{Pow}(S) \to \mc{C}$ by sending $S_0 \subset S$ to the pair $(S_0, S_0 \hra S)$. We claim that $G$ is final, in the $(\infty, 1)$-categorical sense. By~\cite[Thm.\ 4.1.3.1]{htt}, it suffices to show that, for any $(I, \phi) \in \mc{C}$, the comma category $\on{Pow}(S)_{(I, \phi)/}$ is weakly contractible. But this category has a terminal object corresponding to $\phi(I) \in \on{Pow}(S)$, so the claim follows. 
	
	Now we conclude that $\colim_{\mc{C}} \on{pt} \simeq \on{Pow}(S)$, so that the original colimit is isomorphic to $\on{Pow}(S) \setminus \{\emptyset\}$, as desired. 
\end{proof}

\subsubsection{A second definition of $\Ran(X)$} \label{second-def} 
The following concrete description, which appears as~\cite[Def.\ 3.3.1]{z}, follows from the previous lemma: 

\begin{cor}
	$\Ran(X)$ is equivalent to the Set-valued presheaf on $\on{Sch}^{\on{aff, ft}}_{k}$ which sends $S \mapsto \on{Pow}(\on{Maps}_{\on{Sch}_k}(S, X)) \setminus \{\emptyset\}$. 
\end{cor}

This is also the definition of the `unlabeled Ran space' $\Ran^u(X)$ in \cite[Def.\ 2.4.2]{gl}, so $\Ran^u(X)$ in that paper coincides with the $\Ran(X)$ considered here. 

\subsubsection{What about version (iv)?} \label{iv} 

When considering prestacks valued in $(\infty, 1)$-categories, which is category (iv) defined in~\ref{prestacks}, one could consider two versions of $\Ran(X)$: 
\begin{itemize}
	\item Define $\Ran^{(\on{iv, strict})}(X) := \colim_{\fsetsurjneop} X^{(-)}$ where the colimit is the \emph{strict} colimit in the $(\infty, 2)$-category (iv). The strict colimit is equivalent to the ordinary colimit in the $(\infty, 1)$-category obtained from (iv) by discarding the noninvertible 2-morphisms. 
	\item Define $\Ran^{(\on{iv, lax})}(X) := \colim_{\fsetsurjneop}^{\on{lax}} X^{(-)}$ where the colimit is the \emph{lax} colimit in the $(\infty, 2)$-category (iv). 
\end{itemize}
There is a map $\Phi: \Ran^{(\on{iv, lax})}(X) \to \Ran^{(\on{iv, strict})}(X)$. 

\begin{lem}
	The prestack $\Ran^{(\on{iv, strict})}$ is equivalent to the image of $\Ran(X)$ under $\infty\!\on{-Grpd} \hra (\infty, 1)\!\on{-Cat}$. 
\end{lem}
\begin{proof}
	This follows from the fact that $\infty\!\on{-Grpd} \hra (\infty, 1)\!\on{-Cat}'$ is a left adjoint and therefore preserves colimits. Here $(\infty, 1)\!\on{-Cat}'$ is the $(\infty, 1)$-category obtained from $(\infty, 1)\!\on{-Cat}$ by discarding the noninvertible 2-morphisms. The right adjoint is given by sending an $(\infty, 1)$-category $\mc{C}$ to the $\infty$-groupoid obtained by discarding all noninvertible 1-morphisms of $\mc{C}$. 
\end{proof}

However, the map $\Phi$ is not an equivalence. In fact, $\Ran^{(\on{iv, lax})}(X)$ coincides with the prestack defined in~\cite[Def.\ 2.4.9]{gl} for which they use the notation `$\Ran(X)$.' We emphasize that the prestack $\Ran(X)$ studied in the present paper is what they denote `$\Ran^u(X)$.' 

\subsection{Sheaves on $\Ran(X)$} \label{sec-sheaves} 

\subsubsection{} \label{structures} 

Let $\mc{C}$ be any complete $(\infty, 1)$-category, and let $F : \on{Sch}^{\on{aff, ft, op}}_{k} \to \mc{C}$ be any functor. To extend the definition of $F$ to $\on{PreStk}_k$, one takes the right Kan extension along the Yoneda embedding $i : \on{Sch}^{\on{aff, ft, op}}_{k} \hra \on{PreStk}_k$. 

This allows us to define quasicoherent sheaves (and functions, $\D$-modules, etc.)\ on $\Ran(X)$. For definiteness, we take up the problem of defining the derived category of quasicoherent sheaves on $\Ran(X)$, which we shall denote $\on{D}_{\QCoh}(\Ran(X))$. 

Let $\mc{C}$ be the $(\infty, 1)$-category of presentable stable $(\infty, 1)$-categories, and let $F$ be the functor $S \mapsto \on{D}(\QCoh(S))$, with functoriality given by (derived) $*$-pullbacks. Then we define $\on{D}_{\QCoh}(\Ran(X)) := \on{RKE}_i(F)(\Ran(X))$. By the construction of right Kan extensions via limits, and the construction of limits of categories as sections of the corresponding cartesian fibration which send arrows to cartesian arrows (see~\cite[Cor.\ 3.3.3.2]{htt}), we arrive at a more concrete description of the objects of this category: 
\begin{itemize}
	\item For every map $S \to \Ran(X)$ where $S \in \on{Sch}^{\on{aff, ft, op}}_k$, we have a complex $\mc{F}_S \in \on{D}(\QCoh(S))$. 
	\item For every diagram of maps 
	\begin{cd}
		S_1 \ar[rr, "f"] \ar[rd] & & S_2 \ar[ld] \\
		& \Ran(X)
	\end{cd}
	we have an isomorphism $\sigma_f : f^* \mc{F}_{S_2} \xrightarrow{\sim} \mc{F}_{S_1}$. 
	\item These isomorphisms $\sigma_f$ satisfy (higher) cocycle conditions. 
\end{itemize}

One could alternatively define $\on{D}_{\QCoh}(\Ran(X))$ to be the limit 
\[
	\lim_{I \in \fsetsurjne} \on{D}(\QCoh(X^I)). 
\]
Concretely, an object of this category is described as follows: 
\begin{itemize}
	\item For every nonempty finite set $I$, we have a complex $\mc{F}_I \in \on{D}(\QCoh(X^I))$. 
	\item For every surjective map $f : I \sra J$, giving rise to the generalized diagonal map $\Delta_f : X^J \hra X^I$, we have an isomorphism $\sigma_f : (\Delta_f)^*\mc{F}_I \xrightarrow{\sim} \mc{F}_J$. 
	\item These isomorphisms $\sigma_f$ satisfy (higher) cocycle conditions. 
\end{itemize}
(For example, the definition of `category of $\D$-modules on $\Ran(X)$' used in~\cite[Sect.\ 2.1]{fg} follows this pattern.) Since it is this second definition that will be directly used in this paper, we provide an explanation of why it is equivalent to the first one:

\subsubsection{} \label{explain} 
Let $\mc{C}, F, i$ be as before. Define a functor $G$ by the composition 
\[
G : \mc{C} \xrightarrow{\on{Yoneda}} \fun(\mc{C}, \infty\!\on{-Grpd}) \xrightarrow{(-)\, \circ F} \on{PreStk}_k. 
\]
\begin{lem}
	The functors $G, \on{RKE}_i(F)$ are an adjoint pair. 
\end{lem}
\begin{proof}
	This is a standard fact from category theory. The version for $(\infty, 1)$-categories is deduced from the proof of~\cite[Prop.\ 5.2.6.3]{htt}. 
\end{proof}
\begin{cor}
	The functor $\on{RKE}_i(F)$ preserves limits. 
\end{cor}
Since $\Ran(X) = \colim_{\fsetsurjneop} X^{(-)}$ by definition, where the colimit is taken in $\on{PreStk}_k$, the corollary implies that the two definitions of $\on{D}_{\QCoh}(\Ran(X))$ given in~\ref{structures} are equivalent. This is parallel to the discussion in~\cite[2.1.3]{crys} for $\D$-modules. 

\subsubsection{} \label{defs} 

In this paper, we work in the underived setting, i.e.\ we define 
\e{
	\QCoh(\Ran(X)) &= \lim_{S \to \Ran(X)} \QCoh(S) \simeq \lim_{I \in \fsetsurjne} \QCoh(X^I) \\
	\Pic(\Ran(X)) &= \lim_{S \to \Ran(X)} \Pic(S) \simeq \lim_{I \in \fsetsurjne} \Pic(X^I), 
} 
where the isomorphisms follow from~\ref{explain}. The limits in these definitions involve 1-categories, so the cocycle conditions (see~\ref{structures}) are easy to describe: there is one cocycle condition for every commutative triangle, and no higher cocycle conditions. For the reason, line bundles and quasicoherent sheaves on $\Ran(X)$ are very concrete objects. 

We also define 
\[
	\QCoh(\Ran(X))_{\on{flat}} = \lim_{S \to \Ran(X)} \QCoh(S)_{\on{flat}} \simeq \lim_{I \in \fsetsurjne} \QCoh(X^I)_{\on{flat}}, 
\]
where $\QCoh(S)_{\on{flat}}$ is the category of quasicoherent sheaves which are flat over $S$. This is the full subcategory of $\QCoh(\Ran(X))$ consisting of objects $(\mc{F}_S, \sigma_f)$ for which each $\mc{F}_S$ is flat over $S$.

\subsubsection{} 

One could ask how $\QCoh(\Ran(X))$ relates to the category $\on{D}_{\QCoh}(\Ran(X))$ considered in~\ref{structures}. There is no map in either direction; all one can say in general is that both categories map to a third one: 
\begin{cd}
	\QCoh(\Ran(X)) \ar{d}[anchor=south, rotate=90]{\sim} & \on{D}_{\QCoh}(\Ran(X)) \ar{d}[anchor=south, rotate=90]{\sim}\\
	\lim_{I \in \fsetsurjne} \QCoh(X^I) \ar[d] & \lim_{I \in \fsetsurjne} \on{D}(\QCoh(X^I)) \ar[d] \\
	\lim_{I \in \fsetsurjne}^{\on{lax}} \QCoh(X^I) \ar[r] & \lim_{I \in \fsetsurjne}^{\on{lax}} \on{D}(\QCoh(X^I))
\end{cd}
where `lax' denotes lax colimits in the $(\infty, 2)$-category of $(\infty, 1)$-categories.\footnote{Remark: These lax colimits admit a less \emph{ad hoc} interpretation. Namely, $\lim_{I \in \fsetsurjne}^{\on{lax}} \QCoh(X^I)$ is equivalent to the category of natural transformations (i.e.\ pseudonatural transformations)  $\Ran^{\on{(iv, lax)}}(X) \Rightarrow \QCoh(-)$ of functors $\on{Sch}^{\on{aff, ft, op}}_k \to (\infty, 1)\!\on{-Cat}$. One can replace $\QCoh(-)$ by another functor to $(\infty, 1)\!\on{-Cat}$.} The horizontal map arises from the fact that $\on{R}^\cdot f^*$ is right-exact.  

Informally, suppose we are given an object of $\lim_{I \in \fsetsurjne}^{\on{lax}} \QCoh(X^I)$ which is represented by data $(\mc{F}_I, \sigma_\xi)$ as in~\ref{structures}, where lax-ness means that $\sigma_\xi$ need not be an isomorphism. For any map $\xi : I \sra J$, the map $\sigma_\xi : (\Delta_\xi)^*\mc{F}_I \to \mc{F}_J$ yields a map $\on{R}^\cdot (\Delta_\xi)^* \mc{F}_i \to \mc{F}_J$, and in this way one obtains an object of $\lim_{I \in \fsetsurjne}^{\on{lax}} \on{D}(\QCoh(X^I))$. 

\subsubsection{Remark}  \label{delta-flat}
	On a full subcategory of $\lim_{I \in \fsetsurjne} \QCoh(X^I)$, one can say more. If each $\mc{F}$ is flat along the diagonals $\Delta_\xi$ and each $\sigma_\xi$ is an isomorphism, then the maps $\on{R}^\cdot (\Delta_\xi)^* \mc{F}_I \to \mc{F}_J$ constructed in the previous paragraph are isomorphisms, so the resulting object upgrades to one in the strict limit $\lim_{I \in \fsetsurjne} \on{D}(\QCoh(X^I))$. Thus, there is a map
	\[
		\lim_{I \in \fsetsurjne} \QCoh(X^I)_{\Delta\!\on{-flat}} \to \lim_{I \in \fsetsurjne} \on{D}(\QCoh(X^I))
	\]
	where the subscript `$\Delta$-flat' indicates objects which are flat on all the diagonals in $X^I$. This is the generality adopted in Beilinson and Drinfeld's discussion of factorization algebras (see~\cite[3.4.2 and Lem.\ 3.4.3]{bd}), and it is motivated by the observation that these are the quasicoherent sheaves on $\Ran(X)$ in the abelian sense which nevertheless make sense as objects in the derived category. 
	
	Of course, this hypothesis of being flat on all the diagonals is satisfied by flat quasicoherent sheaves and in particular by line bundles. Therefore, the sheaves which are actually studied in this paper do not leave the generality adopted by Beilinson and Drinfeld.

\section{Canonical connection on quasicoherent sheaves over $\Ran(X)$} \label{sec-canon}

The goal of this section is to prove the following theorem, which says that a flat quasicoherent sheaf on $\Ran(X)$ has a unique $\D$-module structure.

\begin{thm} \label{thm-dr} 
	Let $X$ be a smooth $k$-scheme. The pullback functor $\QCoh(\Ran(X)_\dr)_{\on{flat}} \to \QCoh(\Ran(X))_{\on{flat}}$ is an equivalence of categories. 
\end{thm}

As this statement indicates, we adopt the interpretation of $\D$-modules as quasicoherent sheaves on the de Rham prestack set forth in~\cite{crys}. The category $\QCoh(\ran(X))_{\on{flat}}$ was defined in~\ref{defs}, and $\QCoh(\Ran(X)_{\dr})_{\on{flat}}$ is defined in a similar way. 

\subsection{Infinitesimal Ran space}  \label{ran-inf-def}
To place $\D$-module structures on sheaves on $\Ran(X)$, we will need to introduce analogues of $\Ran(X)$ which are defined using formal neighborhoods of diagonal maps. 

\subsubsection{} \label{inf-ran1}
Recall that the category of formal schemes over $k$ is a full subcategory of $\on{PreStk}_k$ (see~\ref{prestacks}) which is closed under products. 

Fix an integer $n \ge 0$. Let $\disk^n$ be the formal neighborhood of $\BA^n$ at the origin, which has structure sheaf $k\bb{x_1, \ldots, x_n}$.\footnote{As the notation suggests, we have $\disk^n \simeq (\disk)^{\times n}$.} For any finite set $I$, let $(\disk^n)^I$ be the $I$-fold product. This assignment is functorial in $\{*\} \sqcup I \in \fset_*^\op$ because $\{*\} \sqcup I \mapsto (\BA^n)^I$ defines a functor to $k$-schemes, and taking formal neighborhoods at the origin is functorial. Therefore, we may define 
\[
	\Ran^n_{\on{inf}} := \colim_{I \in \fsetsurjneop} (\disk^n)^I 
\]
where the colimit is taken in the category of prestacks. 

We also introduce some notation for the rings of functions on these formal neighborhoods. Let $S = \Spec R$, and define the functor 
\[
	\mc{P}_{S, n} : \fset_* \to R\!\on{-mod}
\]
by sending 
\[
	\{*\} \sqcup I \mapsto \oh(S \times (\disk^n)^I) \simeq R\bb{(x_{m, i})_{m \in [n], i \in I}}. 
\]
The map $\xi : \{*\} \sqcup I \to \{*\} \sqcup J$ induces the map on complete rings given by $x_{m, i} \mapsto x_{m, \xi(i)}$ if $\xi(i) \neq *$ and $x_{m, i} \mapsto 0$ otherwise. 

When $S = \Spec k$, the subscript $S$ will be omitted from the notation $\mc{P}_{S, n}$. 

\subsubsection{Artinian Ran space}  \label{art-ran}
We introduce finite-length analogues of the constructions made in~\ref{inf-ran1}. Retain the integer $n > 0$ from before, and fix another integer $d \ge 0$. Let $\disk_d^n$ be the $d$-th infinitesimal neighborhood of $\BA^n$ at the origin, which has structure sheaf $k[x_1, \ldots, x_n] / (x_1, \ldots, x_n)^{d+1}$. (Note that $\disk_d^n$ is \emph{not} isomorphic to $(\disk_d^1)^{\times n}$.) Similarly, define $\disk^{n, I}_d$ to be the $d$-th infinitesimal neighborhood of $(\BA^n)^I$ at the origin. (Note that $\disk^{n, I}_d$ is \emph{not} isomorphic to $(\disk_d^n)^I$.) We define 
\[
	\Ran^n_{\la d \ra} := \colim_{I \in \fsetsurjneop} \disk^{n, I}_d, 
\]
where the colimit is taken in the category of prestacks. 

For $S = \Spec R$, we define the functor 
\[
	\mc{P}_{S, n, \le d} : \fset_* \to R\!\on{-mod}
\]
by sending 
\[
	\{*\} \sqcup I \mapsto \oh(S \times \disk^{n, I}_d) \simeq R[(x_{m, i})_{m \in [n], i \in I}] / \mf{m}^{d+1}, 
\]
where $\mf{m}$ is the maximal ideal at the origin. The right hand side consists of polynomials of degree $\le d$ in the variables $x_{m, i}$. Similarly, there is the functor $\mc{P}_{S, n, d}$ whose values are spaces of homogeneous polynomials of degree $d$.

\begin{lem} \label{limran}
	We have $\raninf^n \simeq \colim_d \Ran_{\la d \ra}^n$
\end{lem}
\begin{proof}
	Since colimits commute with colimits, this follows from the fact that $(\disk^n)^I \simeq \colim_d \disk_d^{n, I}$ for all $I$. 
\end{proof}

\subsubsection{} 
As one would expect, the functor $\mc{P}_{S, n, d}$ of `homogeneous polynomials of degree $d$' is polynomial of degree $d$ in the sense of Definition~\ref{def-poly}. Moreover, we prove that $\mc{P}_{S, n, d}$ is polynomial of degree $d$, even in the derived sense (Definition~\ref{def-excisive}): 
\begin{lem}\label{rnd-cart} 
	Let $S = \Spec R$ be an affine $k$-scheme. The functor $\mc{P}_{S, n, d}$, viewed as a functor to $\on{D}(\on{AbGrp})$, sends special hypercubes $\Xi_{(I_b)_{b \in B}}$ with $|B| > d$ to \emph{strongly} cartesian diagrams. In particular, $\mc{P}_{S, n, d}$ is $d$-excisive. 
\end{lem}
\begin{proof}
	For any tuple $(I_b)_{b \in B}$ of finite sets with $|B| > d$, we will decompose the diagram $\mc{P}_{S, n, d}(\Xi_{(I_b)_{b \in B}})$ as the direct sum of several strongly cartesian diagrams, each corresponding to one monomial basis element in $\mc{P}_{S, n, d}(\{*\} \sqcup I)$, where $I := \sqcup_{b \in B} I_b$. To start, recall that $\mc{P}_{S, n, d}(\{*\} \sqcup I)$ is a free $R$-module of finite rank with basis given by monomials of degree $\le d$ in the variables $x_{m, i}$, where $m \in [n]$ and $i\in I$, see~\ref{art-ran}. Each such monomial $m$ splits uniquely as a product $m = \prod_{b \in B} m_b$ where $m_b$ consists of all the variables $x_{m, i}$ for which $i \in I_b$. Let $B_m \subset B$ be the subset of $b$ for which $m_b \neq 1$. Since $|B| > d$, we have $B_m \subsetneq B$. 
	
	By definition, the \emph{hypercube diagram spanned by $m$}, denoted $\Xi_m$, is the sub-diagram of $\mc{P}_{S, n, d}(\Xi_{(I_b)_{b \in B}})$ spanned by $m$ in $\mc{P}_{S, n, d}(\{*\} \sqcup I)$ and the images of $m$ in the other vertices of $\mc{P}_{S, n, d}(\Xi_{(I_b)_{b \in B}})$. Recall that vertices of the hypercube are parameterized by subsets $B' \subset B$. Observe that $\Xi_m$ consists of $k$ at the vertices for which $B_m \subset B'$ and $0$ at all other vertices, and the map between any two copies of $k$ is the identity map. This description shows that every square of $\Xi_m$ is cartesian, so it is a strongly cartesian diagram. 
	
	To finish, note that the original hypercube $\mc{P}_{S, n, d}(\Xi_{(I_b)_{b \in B}})$ is the direct sum of $\Xi_m$ for all degree $d$ monomials $m \in \mc{P}_{S, n, d}(\{*\} \sqcup I)$. 
\end{proof}

The lemma immediately implies the analogous $d$-excision statement for the functor $\mc{P}_{S, n, \le d}$.

\subsubsection{Remark} \label{exercise} 

Consider the following statement: 
\begin{itemize}
	\item The map $\lim_{\fset^{\on{surj}}_{\on{n.e.}, \le 3}} \mc{P}_{S, 1, d} \circ \iota \to \mc{P}_{S, 1, d}(\{*, 1\})$ is zero. 
\end{itemize}
For $S = \Spec R$, the statement is equivalent to the following: 
\begin{itemize}
	\item Let $f(x, y, z) \in R[x, y, z]$ be a symmetric polynomial which is homogeneous of degree $d$. Assume that $g(x, y) := f(x, x, y)$ is symmetric (i.e.\ $f(x, x, y) = f(x, y, y)$). Then $h(x) := f(x, x, x)$ is zero. 
\end{itemize}

This statement in the special case $R = \BC$ appears as~\cite[Exercise~3.4.2]{bd}, and it is stated for arbitrary $R$ as~\cite[Lem.\ 4.3.11]{z}. Unfortunately, this statement does not hold when the characteristic of the field $k$ is nonzero. We give a counterexample when $\on{char}(k) = p$ and $d = p$, for any prime $p \ge 5$. Let 
\[
	f(x, y, z) := xyz \left( x^{p-3} + y^{p-3} + z^{p-3} + \left( \frac{p+1}{2} \right) \sum_{\on{cyc}} (x^{p-4}y + \cdots + xy^{p-4}) \right), 
\]
where the subscript cyc denotes cyclic summation with respect to $x, y, z$. We have 
\e{
	f(x, x, y) &= x^2 y \left( 2x^{p-3} + y^{p-3} + 2 \left( \frac{p+1}{2} \right) (x^{p-4}y + \cdots + xy^{p-4})  + \left( \frac{p+1}{2} \right) (p-4)x^{p-3} \right) \\
	&= x^{p-2}y^2 + x^{p-3}y^3 + \cdots + x^{2} y^{p-2}, 
} 
so $g(x, y)$ is symmetric. Then 
\[
	g(x, x)= (p-1)x^p
\]
is nonzero, i.e.\ $h(x)$ is nonzero. This yields our desired counterexample. 

We do not know whether this statement becomes true (in characteristic $p$) if $\le 3$ is replaced by $\le 4$, although it seems doubtful. As discussed in Remark~\ref{further-fin}, it is unlikely that such a statement could be deduced from a `vanishing of the limit' result which holds for arbitrary polynomial functors.

\subsubsection{}
Consider the maps 
\begin{cd}
	\raninf^n \ar[r, shift left = 1.5, "p"] & \Spec k =: \on{pt} \ar[l, shift left = 0.5, "q"] 
\end{cd}
which satisfy $p \circ q = \id$. For $V \in \on{Vect}_k$, write $\ul{V} := p^*V$ for notational convenience. The next result says that quasicoherent sheaves on the infinitesimal Ran space are canonically trivial. 
\begin{prop} \label{prop-triv} 
	We have mutually inverse equivalences of symmetric monoidal categories 
	\begin{cd}
		\QCoh(\raninf^n)_{\on{flat}} \ar[r, shift right = 0.5, swap, "q^*"] & \QCoh(\on{pt}) = \Vect_{k} \ar[l, shift right = 1.5, swap, "p^*"] 
	\end{cd}
	(The left hand side category is defined in analogy with~\ref{defs}.) 
\end{prop}
\begin{proof}
	Let $V$ be a fixed vector space over $k$. It suffices to show that any pair $(\mc{E}, \sigma)$ with $\mc{E} \in \QCoh(\raninf^n)_{\on{flat}}$ and $\sigma : q^*\mc{E} \simeq V$ admits a \emph{unique} trivialization $\tau : \mc{E} \simeq \ul{V}$ such that $q^*\tau = \sigma$. In view of Lemma~\ref{limran}, it suffices to prove this analogous statement for all $d$: 
	\begin{itemize}
		\item[(P)] For any pair $(\mc{F}, \sigma)$ with $\mc{F} \in \QCoh(\ranf{d}^n)_{\on{flat}}$ and $\sigma : q^*\mc{F} \simeq V$, there exists a unique trivialization $\tau : \mc{F} \simeq \ul{V}$ such that $q^*\tau = \sigma$. 
	\end{itemize}
	The statement is trivial for $d = 0$, so we fix $d \ge 1$ and assume that it has been proven for all smaller $d$. Thus, we start with $(\mc{F}, \sigma)$ as above, and we know that $(\mc{F}|_{\ranf{d-1}^n}, \sigma)$ admits a unique trivialization $\ol{\tau}$ (by the inductive hypothesis). We want to show that there is a unique extension of $\ol{\tau}$ to $\ranf{d}^n$, i.e.\ that the category of such pairs $(\mc{F}, \ol{\tau})$ is trivial. 
	
	For any finite set $I$, any flat quasicoherent sheaf on $\disk^{n, I}_d$ is (noncanonically) free, and any trivialization on $\disk^{n,I}_{d-1}$ (noncanonically) extends to one on $\disk^{n, I}_d$. Let $\mc{C}_d(I)$ be the category of pairs $(\mc{G}, \pi)$ where $\mc{G} \in \QCoh(\disk^{n, I}_d)_{\on{flat}}$ and $\pi$ is an isomorphism of $\mc{G}$ with $\ul{V}$ on $\disk^{n,I}_{d-1}$. The first sentence of this paragraph implies that $\mc{C}_d(I)$ is equivalent to a groupoid with one object. Its automorphism group is the group of automorphisms of $\ul{V}$ on $\disk^{n, I}_d$ which are trivial on $\disk^{n,I}_{d-1}$. 
	
	\begin{claim}
		The automorphism group $\pi_1(\mc{C}_d(I))$ identifies with the abelian group $\mc{P}_{n, d}(\{*\} \sqcup I) \otimes_k \End_k(V)$ where $\End_k(V)$ is regarded as a group via its \emph{additive} structure. 
	\end{claim}
	\begin{proof}
		The endomorphism $k$-algebra of $\ul{V}$ on $\disk^{n, I}_d$ is 
		\e{
			\End_{\mc{P}_{n, \le d}(\{*\} \sqcup I)}(V \otimes_k \mc{P}_{n, \le d}(\{*\} \sqcup I)) &\simeq \Hom_k (V, V \otimes_k \mc{P}_{n, \le d}(\{*\} \sqcup I)) \\
			&\simeq \Hom_k (V, V) \otimes_k \mc{P}_{n, \le d}(\{*\} \sqcup I), 
		} 
		where the last line follows because $\mc{P}_{n, \le d}(\{*\} \sqcup I)$ is finite dimensional. The restriction of endomorphisms to $\disk^{n,I}_{d-1}$ is given by the surjection 
		\[
			\Hom_k (V, V) \otimes_k \mc{P}_{n, \le d}(\{*\} \sqcup I) \sra \Hom_k (V,V) \otimes_k \mc{P}_{n, \le d-1}(\{*\} \sqcup I)
		\]
		induced by the quotient map $\mc{P}_{n, \le d}(\{*\} \sqcup I) \sra \mc{P}_{n, \le d-1}(\{*\} \sqcup I)$. The desired automorphism group is the preimage of $1$ under this map, with the group structure induced by ring multiplication. The kernel clearly identifies with $\Hom_k (V, V) \otimes_k \mc{P}_{n, d}(\{*\} \sqcup I)$, and the group structure is as described because the ideal $\mc{P}_{n, d}(\{*\} \sqcup I) \subset \mc{P}_{n, \le d}(\{*\} \sqcup I)$ is square-zero. 
	\end{proof}
	
	We have a functor $G : \fset_* \to \on{AbGrp}$ defined by $\{*\} \sqcup I \mapsto \mc{P}_{n, d}(\{*\} \sqcup I) \otimes_k \End_k(V)$ with the additive group structure. The category of pairs $(\mc{F}, \ol{\tau})$ as above\footnote{Namely, $\mc{F}$ is a sheaf on $\ranf{d}^n$ and $\ol{\tau}$ is an isomorphism with $\ul{V}$ on $\ranf{d-1}^n$.} is equivalent to $\lim_{I \in \fsetsurjne} \mc{C}_d(I)$ for tautological reasons, so the previous claim implies that the set of isomorphism classes of objects is in bijection with $\lim_{\fsetsurjne}^1 G\iota$ and the zero object has automorphism group $\lim_{\fsetsurjne}^0 G\iota$. To finish the proof, it suffices to show that both of these groups are trivial. 
	
	\begin{claim}
		The functor $G$, viewed as a functor to $\on{D}(\on{AbGrp})$, is $d$-excisive. 
	\end{claim}
	\begin{proof}
		Since $\End_k(V)$ splits as the (infinite) direct sum of one-dimensional $k$-vector spaces, and tensor products commute with direct sums, the functor $G$ is the direct sum of copies of $\mc{P}_{n, d}$. Thus, the claim follows from Lemma~\ref{rnd-cart} and the fact that infinite direct sums preserve cartesian squares because we are working in a stable category. 
	\end{proof}
	
	Now Theorem~\ref{thm2} applies (because $G$ takes values in cohomological degrees $\ge 0$), and we conclude that $\lim_{\fsetsurjne}^i G\iota$ is zero for all $i$. This shows that the category of pairs $(\mc{F}, \ol{\tau})$ is equivalent to a point, which completes the proof of the inductive step. 
\end{proof}

\subsubsection{} The benefit of working in the generality of quasicoherent sheaves rather than coherent sheaves is that, by considering algebra objects, we can deduce an analogous statement relative to an arbitrary scheme over $k$. 
\begin{cor}\label{cor-triv}
	Let $Y$ be a $k$-scheme. Then we have mutually inverse equivalences of categories 
	\begin{cd}[column sep = 0.7in] 
		\QCoh(Y \times \raninf^n)_{\on{fl.}} \ar[r, shift right = 0.5, swap, "(\id \times q)^*"] & \QCoh(Y) \ar[l, shift right = 1.5, swap, "(\id \times p)^*"]
	\end{cd}
	where the subscript $\on{fl.}$ indicates quasicoherent sheaves flat over $\raninf^n$. 
\end{cor}
\begin{proof}
	Since quasicoherent sheaves form a stack in the Zariski topology, we reduce to the case in which $Y = \Spec A$ is affine. Now the claim follows from Proposition~\ref{prop-triv} because the concept of $A$-module is expressed in terms of the monoidal structure on QCoh.  
\end{proof}

\subsection{Partially labeled Ran spaces}  \label{ran-partial-def}

Given a $k$-scheme $Y$, let $Y^2_{\wh{\Delta}}$ be the formal neighborhood of $Y^2$ along the diagonal. Roughly speaking, a $\D$-module on $Y$ consists of a quasicoherent sheaf $\mc{F}$ on $Y$, along with an identification of the pullbacks of $\mc{F}$ to $Y^2_{\wh{\Delta}}$ along the two projections $\pr_1, \pr_2 : Y^2_{\wh{\Delta}} \rightrightarrows Y$. There are two ways in which these projection maps differ from the generalized diagonal maps $\Delta_{I \sra J} : Y^J \hra Y^I$ which appear in the colimit definition of $\Ran(Y)$: 
\begin{itemize}
	\item The spaces $Y^J$ and $Y^I$ are not completed along the diagonal. This issue can be easily resolved by pulling back to a completed version of the Ran space, defined as the colimit over the maps $\Delta_{I \sra J} : Y^J_{\wh{\Delta}} \to Y^I_{\wh{\Delta}}$, see~\ref{rani-d}. 
	\item In terms of labeled subsets of $Y$, the projection maps correspond to \emph{deletion} of points, while the generalized diagonal maps correspond to doubling of points. This issue is much more serious, and overcoming it is the crux of the proof. The main idea is to find a way to delete points (in some sense) using only the generalized diagonal maps. See Remark~\ref{rani-p} for further discussion of this intuition. 
\end{itemize}

In this subsection, $Y$ is a fixed $k$-scheme. 

\subsubsection{} \label{rani} 

Define the functor $\Ran^{(-)}(Y) : \fset^{\op} \to \on{PreStk}_{k}$ by 
\[
	\Ran^I(Y) := Y^I \times \Ran(Y)
\]
for $I \in \fset$, and for a map $I_1 \xrightarrow{\xi} I_2$ the corresponding map 
\[
	\Ran^{I_2}(Y)  \xrightarrow{\Ran^\xi(Y)} \Ran^{I_1}(Y)
\]
is given by the composition 
\begin{cd}
	\Ran^{I_2}(Y) =  Y^{I_2} \times \Ran(Y) \ar[r, "\sim"] & Y^{\Im(\xi)} \times Y^{I_2 \setminus \Im(\xi)} \times \Ran(Y) \ar[d, "{\Delta_\xi \times i \times \id_{\Ran(Y)}}"] \\
	& Y^{I_1} \times \Ran(Y) \times \Ran(Y) \ar[d, "\id \times \on{mult}"] \\
	& Y^{I_1} \times \Ran(Y)
\end{cd}
Here $\Delta_{\xi} = \Delta_{I_1 \xrightarrow{\xi} \Im(\xi)}$ is the diagonal map $Y^{\Im(\xi)} \to Y^{I_1}$ associated to $\xi$, and $i : Y^{I_2 \setminus \Im(\xi)} \to \Ran(X)$ is the standard map given by the colimit expression for $\Ran(X)$. In particular, if $I_1 = \emptyset$, then $\Im(\xi)= \emptyset$ so the map $\Delta_\xi$ is just the map $\on{pt} \to \on{pt}$. 

\subsubsection{} \label{rani-p}
We have a map $\Ran^I(Y) \xrightarrow{p^I} Y^I$ given by projecting onto the first factor, and this is functorial in $I \in \fset^{\op}$, so $p^{(-)}$ gives a natural transformation $\Ran^{(-)}(Y) \Longrightarrow Y^{(-)}$. Because this is the most important construction in this subsection, let us spell it out in more detail. For any map $I_1 \xrightarrow{\xi} I_2$, we get a commuting square 
\begin{cd}
	\Ran^{I_2}(Y) \ar[r, "p^{I_2}"] \ar[d, swap, "\ran^\xi(Y)"] & Y^{I_2} \ar[d, "Y^\xi"] \\
	\Ran^{I_1}(Y) \ar[r, "p^{I_1}"] & Y^{I_1}
\end{cd}
This is not as trivial as it looks because $\Ran^\xi(Y)$ is not equal to the map $Y^\xi \times \id_{\Ran(Y)}$. 

\begin{rmk}
	We explain the significance of this commuting square. Let $\xi$ be the inclusion $\{1\} \hra \{1, 2\}$. A $k$-point of $\Ran^{\{1, 2\}}(Y)$ is a triple $(y_1, y_2, C)$ where $y_1, y_2$ are points of $Y$ and $C$ is some finite set of points of $Y$. The mapping diagram is as follows: 
	\begin{cd}
		(y_1, y_2, C) \ar[r, mapsto] \ar[d, mapsto] & (y_1, y_2) \ar[d, mapsto] \\
		(y_1, y_2 \cup C) \ar[r, mapsto] & y_1
	\end{cd}
	The right vertical map is the projection $\pr_1 : Y^2 \to Y$, while the left vertical map is defined solely in terms of the generalized diagonal maps. (In particular, the semigroup multiplication on $\Ran(X)$ is ultimately defined using the generalized diagonal maps; this is made more explicit in~\ref{rani-d}.) 
\end{rmk}

\subsubsection{} \label{rani-p2} 

For any finite set $I$, there is also a map $d^I : Y^I \to \Ran^I(Y)$ defined as 
\[
	Y^I \xrightarrow{\on{diag}} Y^I \times Y^I \xrightarrow{\id_{Y^I} \times i} Y^I\times \Ran(Y) = \Ran^I(Y)
\]
where `diag' is the diagonal map and $i : Y^I \hra \Ran(Y)$ is the map in the colimit diagram for $\Ran(Y)$. The maps $d^I$ are \emph{not} functorial in $I$, but they  satisfy two useful properties: 
\begin{enumerate}[label=(\roman*)]
	\item For any $I$, the following diagram commutes: 
	\begin{cd}
		\Ran^I(Y) \ar[rd, swap, "\Ran^{c}(Y)"] & & Y^I \ar[ll, swap, "d^I"] \ar[ld, "i"]\\
		& \Ran(Y)
	\end{cd}
	Here $c : \emptyset \to I$ is the unique map, and the target of $\Ran^c(Y)$ is $\Ran^{\emptyset}(Y) = \Ran(Y)$. 
	\item For any $I$, we have $p^I \circ d^I = \id_{Y^I}$. 
\end{enumerate}

\subsubsection{Completions along the diagonal} \label{rani-d}

Define 
\[
\Ran_{\wh{\Delta}}(Y) = \colim_{J \in \fsetsurjneop} Y^J_{\wh{\Delta}}, 
\]
where the subscript $\wh{\Delta}$ denotes the formal neighborhood of the (small) diagonal $Y \hra Y^J$. 

\begin{rmk}
	This version of the Ran space was used by Beilinson and Drinfeld in establishing the existence of the canonical connection on a factorization algebra on $\Ran(X)$, see \cite[Prop.\ 3.4.7]{bd}. They denote $Y^J_{\wh{\Delta}}$ by $Y^{<J>}$. 
\end{rmk}

Similarly, define 
\[
	\Ran_{\wh{\Delta}}^I(Y) = \colim_{J \in \fsetsurjneop} Y^{I \sqcup J}_{\wh{\Delta}}.
\]
We enhance the assignment $I \rightsquigarrow \Ran_{\wh{\Delta}}^I(Y)$ into a functor for $I \in \fset^{\op}$ in analogy with $\Ran^I(Y)$ as defined in~\ref{rani}. Namely, given a map $I_1 \xrightarrow{\xi} I_2$, the corresponding map 
\[
	\Ran_{\wh{\Delta}}^{I_2}(Y) \xrightarrow{\Ran_{\wh{\Delta}}^\xi (Y) } \Ran_{\wh{\Delta}}^{I_1}(Y)
\]
is induced by the map 
\e{
	Y^{I_2 \sqcup J}_{\wh{\Delta}} &\simeq (Y^{\Im(\xi)} \times Y^{(I_2 \setminus \Im(\xi)) \sqcup J})_{\wh{\Delta}} \\
	&\to (Y^{I_1} \times Y^{(I_2 \setminus \Im(\xi)) \sqcup J})_{\wh{\Delta}} \\
	&\simeq Y^{I_1 \sqcup \big((I_2 \setminus \Im(\xi)) \sqcup J\big)}_{\wh{\Delta}}
} 
by taking the colimit with respect to $J$, and using the colimit version of Lemma~\ref{lims} for the functor $\fsetsurjne \to \fsetsurjne$ given by $J \mapsto (I_2 \setminus \Im(\xi)) \sqcup J$. When $I_1 = \emptyset$, the last sentence of~\ref{rani} explains how to define the above map. 

\subsubsection{} \label{rani-d2} 

As in~\ref{rani-p}, we have a natural transformation $p^{(-)} : \Ran^{(-)}_{\wh{\Delta}}(Y) \to Y^{(-)}_{\wh{\Delta}}$ of functors $\fset^{\op}\to \on{PreStk}_k$. The map $p^I$ is obtained from the projection map $Y^{I \sqcup J}_{\wh{\Delta}} \to Y^I_{\wh{\Delta}}$ by taking colimits with respect to $J \in \fsetsurjneop$ and applying the colimit version of Lemma~\ref{lims}. 

As in~\ref{rani-p2}, we have maps $d^I : Y^I_{\wh{\Delta}} \to \Ran^{I}_{\wh{\Delta}}(Y)$ satisfying properties analogous to (i) and (ii) in~\ref{rani-p2}. The map $d^I$ is defined to be the composition $Y^I_{\wh{\Delta}} \xrightarrow{\on{diag}} Y^{I \sqcup I}_{\wh{\Delta}} \xrightarrow{i} \Ran^I_{\wh{\Delta}}(Y)$ where $i$ is part of the colimit diagram for $\Ran^I_{\wh{\Delta}}(Y)$. 

\subsubsection{} We prove a technical result which will be used in Lemma~\ref{lem-twist}. The idea is that, given a coordinate system on $Y$, the datum of $|I|+1$ nearby points on $Y$ is the same as the datum of the first point, along with an $|I|$-tuple of displacements. 
\begin{lem}\label{subtract} 
	Let $Y$ be a $k$-scheme equipped with an \'etale map $f : Y \to \BA^n$. Then, for each $I \in \fset$ there is an isomorphism 
	\[
		Y^{\{0\} \sqcup I}_{\wh{\Delta}} \simeq Y \times (\disk^n)^I, 
	\]
	and this is functorial in $I$. 
\end{lem}
\begin{proof}
	Consider the commutative diagram 
	\begin{cd}
		& Y^{\{0\} \sqcup I} \ar[d, "\id_Y \times (f)^I"] \\
		Y \ar[ru, "\Delta"] \ar[r, swap, "\Gamma_{(f)^I}"] & Y \times (\BA^n)^I
	\end{cd}
	Since the vertical map is \'etale, the formal neighborhood of $Y^{\{0\} \sqcup I}$ along $Y$ (via the diagonal map) is isomorphic to the formal neighborhood of $Y \times (\BA^n)^I$ along $Y$ (via the horizontal map). The graph $\Gamma_{(f)^I}$ defines an automorphism of $Y \times (\BA^n)^I$ via the group structure of $\BA^n$, and this automorphism allows us to replace the horizontal map by 
	\[
		Y \xrightarrow{(\id_Y, 0, \ldots, 0)} Y \times (\BA^n)^I. 
	\]
	But the formal neighborhood associated to this map is evidently $Y \times (\disk^n)^I$. 
\end{proof}

\subsubsection{} We are now in a position to deduce from Corollary~\ref{cor-triv} a triviality result which relates flat quasicoherent sheaves on the two sides of the crucial diagram of~\ref{rani-p}. 
\begin{lem}\label{lem-twist}
	Assume that $Y$ is smooth over $k$. For any $I \in \fsetne$, the functors
	\begin{cd}
		\QCoh(\Ran_{\wh{\Delta}}^{I}(Y))_{\on{flat}} \ar[r, shift right = 0.5, swap, "(d^I)^*"] & \QCoh(Y^I_{\wh{\Delta}})_{\on{flat}} \ar[l, shift right = 1.5, swap, "(p^I)^*"] 
	\end{cd}
	yield mutually inverse equivalences of categories, where the subscript `\emph{flat}' indicates quasicoherent sheaves flat over $\Ran_{\wh{\Delta}}^{I}(Y)$ and $Y^I_{\wh{\Delta}}$, respectively. 
\end{lem}
\begin{proof}
	This proof will show that $(p^I)^*$ is an equivalence of categories. This implies the statement about $(d^I)^*$ because we have $p^I \circ d^I = \id_{Y^I_{\wh{\Delta}}}$. 	
	
	First, we argue that both sides are global sections of Zariski sheaves of categories on $Y$. 
	\begin{enumerate}
		\item For any scheme $\wt{Y}$ equipped with an isomorphism $\sigma : \wt{Y}^{\on{red}} \simeq Y$, the category $\QCoh(\wt{Y})$ is the global sections of a Zariski sheaf of categories on $Y$ whose value on an open subscheme $U \subset Y$ is $\QCoh(\wt{U})$, where $\wt{U} \subset \wt{Y}$ is the unique open subscheme of $\wt{Y}$ whose underlying reduced scheme identifies with $U$ via $\sigma$. 
		\item This statement remains true when $\wh{Y}$ is replaced by any prestack $\mc{Z}$ of the form $\colim_{c \in \mc{C}} Z_c$, where $Z_{(-)}$ is a functor from some category $\mc{C}$ to the category of pairs $(\wt{Y}, \sigma)$ as in point (1). This follows from (1) because $\QCoh(\mc{Z}) = \lim_c \QCoh(Z_c)$ and limits of sheaves are computed pointwise. 
		\item Finally, observe that $\Ran^I_{\wh{\Delta}}(Y)$ and $Y^I_{\wh{\Delta}}$ are prestacks of the form considered in (2). 
	\end{enumerate}
	Since the functor $(p^I)^*$ arises from a map of sheaves in this sense, it suffices to prove the lemma after replacing $Y$ by $U$, where $U$ ranges over a Zariski open cover of $Y$. In this way, we may assume that $Y$ has a globally-defined coordinate system, i.e.\ an \'etale map $Y \to \BA^n$. 
	
	Although the functor $(p^I)^*$ is canonical, we shall prove that it is an equivalence by doing something noncanonical. Since $I$ is assumed to be nonempty, we may pick an element $i \in I$ and thereby obtain an isomorphism 
	\e{
		Y^{I \sqcup J}_{\wh{\Delta}} &\simeq Y \times (\disk^n)^{(I \setminus \{i\}) \sqcup J} \\
		&\simeq Y \times (\disk^n)^{I \setminus \{i\}} \times (\disk^n)^J
}
 by Lemma~\ref{subtract}. Taking colimits with respect to $J \in \fsetsurjneop$, we obtain an isomorphism 
	\[
		\phi : \Ran^I_{\wh{\Delta}}(Y) \simeq Y \times (\disk^n)^{I \setminus \{i\}} \times \raninf^n. 
	\]
	Similarly, we get an isomorphism $\phi' : Y^I_{\wh{\Delta}} \simeq Y \times (\disk^n)^{I \setminus \{i\}}$, and it is not difficult to see that this diagram commutes: 
	\begin{cd}
		\Ran^I_{\wh{\Delta}}(Y) \ar[r, "p^I"] \ar{d}[anchor=north, rotate=90]{\sim}[swap]{\phi} & Y^I_{\wh{\Delta}} \ar{d}[anchor=north, rotate=90]{\sim}[swap]{\phi'} \\
		Y \times (\disk^n)^{I \setminus \{i\}} \times \raninf^n \ar[r, "\pr_{12}"] & Y \times (\disk^n)^{I \setminus \{i\}}
	\end{cd}
	Therefore, it suffices to show that $(\pr_{12})^*$ induces an isomorphism on categories of flat quasicoherent sheaves. Since limits preserve isomorphisms, it suffices to show the analogous statement for the maps 
	\[
		Y \times (\disk^n_d)^{I \setminus \{i\}} \times \raninf^n \xrightarrow{\pr_{12}} Y \times (\disk^n_d)^{I \setminus \{i\}}
	\]
	for integers $d \ge 0$. But this statement follows from Corollary~\ref{cor-triv}. 
\end{proof}

\subsubsection{} Following through with the plan articulated at the start of this subsection, we interpret~\ref{rani-p} and Lemma~\ref{lem-twist} as saying that, as far as flat quasicoherent sheaves are concerned, projection maps can be realized in terms of generalized diagonal maps. This allows us to begin constructing $\D$-module structures: 
\begin{lem} \label{lem-psi}
	There is a functor $\Psi: \QCoh(\ran(Y))_{\on{flat}} \to \QCoh(Y_\dr)_{\on{flat}}$ such that this diagram strictly commutes: 
	\begin{cd}
		\QCoh(\ran(Y)_\dr)_{\on{flat}} \ar[r] \ar[d] & \QCoh(\ran(Y))_{\on{flat}} \ar[d] \ar[ld, "\Psi"] \\
		\QCoh(Y_{\dr})_{\on{flat}} \ar[r] & \QCoh(Y)_{\on{flat}}
	\end{cd}
	The four unlabeled functors are induced by pullback along $Y \hra \ran(Y)$ and along the map from a prestack to its associated de Rham prestack. 
\end{lem}
\begin{proof}
	Given $\mc{F} \in \QCoh(\ran(Y))_{\on{flat}}$, we need to equip $\mc{F}_Y \in \QCoh(Y)_{\on{flat}}$ with the datum of descent along $Y \to Y_{\dr}$, i.e.\ an action of the infinitesimal groupoid of $Y$. Recall that the infinitesimal groupoid of $Y$ is a functor $L_Y : \Delta^{\op} \to \on{PreStk}_{k}$ for which $L_Y([0]) = Y$, $L_Y([0 \to 1]) = Y_{\wh{\Delta}}^{2}$, and the two maps $[0] \rightrightarrows [0 \to 1]$ correspond to the two projections $Y_{\wh{\Delta}}^2 \rightrightarrows Y$.\footnote{It is unfortunate that the notation for the simplicial category $\Delta$ clashes with the notation for the small diagonal $\Delta : Y \hra Y^I$. In this proof, the latter meaning is intended only when $\Delta$ appears as a subscript $(-)_{\wh{\Delta}}$ indicating completion.} More precisely, if $\tau : \Delta \to \fsetne$ is the functor sending $[0 \to \cdots \to n]$ to the underlying set $\{0, \ldots, n\}$, then we have $L_Y \simeq Y_{\wh{\Delta}}^{(-)} \circ \tau$. 
	
	(In the rest of this proof, we omit the subscript `flat' for notational convenience. Every instance of $\QCoh$ is meant to be $\QCoh_{\on{flat}}$.) 
	
	\begin{rmk}
		The category of quasicoherent sheaves equipped with action by the infinitesimal groupoid is $\lim (\QCoh \circ L_Y)$ by definition. Here, as motivation, we explain a naive attempt to produce an object in this category and why it fails. By restricting $\mc{F}$ to $Y_{\wh{\Delta}}^I$ for each $I \in \fsetsurjne$, we obtain an object 
		\[
			\mc{F}' \in \lim (\QCoh \circ Y_{\wh{\Delta}}^{(-)} \circ \iota). 
		\]
		(This is just the restriction of $\mc{F}$ along $\Ran_{\wh{\Delta}}(Y) \to \Ran(Y)$, and the category displayed above is just $\QCoh(\Ran_{\wh{\Delta}}(Y))$ by definition.) If we had an object in $\lim (\QCoh \circ Y_{\wh{\Delta}}^{(-)})$, then we could precompose by $\tau$ in the limit and thereby obtain an object in the desired category. The presence of $\iota$ above indicates that $\mc{F}'$ does not have sufficient functoriality for this to work: it defines a family of sheaves on the $Y_{\wh{\Delta}}^I$ which is not equipped with functoriality with respect to nonsurjective maps $I \to J$. The next construction fixes this problem.
	\end{rmk} 
	
	Since $\emptyset \in \fset$ is initial, we get maps $\Ran_{\wh{\Delta}}^I(Y) \to \Ran_{\wh{\Delta}}^{\emptyset}(Y) \simeq \Ran_{\wh{\Delta}}(Y)$ which are functorial in $I \in \fset$. Composing with the map $\Ran_{\wh{\Delta}}(Y) \to \Ran(Y)$ and pulling back $\mc{F}$ to each $\Ran_{\wh{\Delta}}^I(Y)$ yields an object
	\[
		\mc{F}' \in \lim_{\fset}(\QCoh \circ \Ran_{\wh{\Delta}}^{(-)}(Y)). 
	\]
	Since Lemma~\ref{lem-twist} is functorial in $I \in \fsetne$, it gives a natural equivalence between $\QCoh \circ \Ran_{\wh{\Delta}}^{(-)}(Y)$ and $\QCoh \circ Y_{\wh{\Delta}}^{(-)}$ as functors on $\fsetne$. Under this equivalence, $\mc{F}'$ corresponds to an object
	\[
		\mc{F}'' \in \lim_{\fsetne}(\QCoh \circ Y_{\wh{\Delta}}^{(-)} ). 
	\]
	Precomposing by $\tau$, we obtain an object 
	\[
		\mc{F}''' \in \lim_{\Delta}(\QCoh \circ Y_{\wh{\Delta}}^{(-)} \circ \tau). 
	\]
	By the first paragraph of this proof, $\QCoh \circ Y_{\wh{\Delta}}^{(-)} \circ \tau$ coincides with $\QCoh \circ L_Y$, so $\mc{F}'''$ is a quasicoherent sheaf on $Y$ equipped with action by $L_Y$, and we set $\Psi(\mc{F}) := \mc{F}'''$. It is straightforward to amplify this to yield a definition of $\Psi$ as a functor. 
	
	Let us equip the lower-right triangle with the datum of commutativity. This is done via the commutative diagram 
	\begin{cd}
		\Ran_{\wh{\Delta}}^{\{1\}} \ar[rd, swap, "\Ran^c_{\wh{\Delta}}(Y)"] & & Y \ar[ll, swap, "d^{\{1\}}"] \ar[ld, "i"] \\
		& \Ran(Y)
	\end{cd}
	which results from the property (ii) mentioned in~\ref{rani-d2}. By the construction of $\Psi$, the pullback of $\Psi(\mc{F})$ to $Y$ is obtained by pulling back $\mc{F}$ along $\Ran^c_{\wh{\Delta}}(Y)$ and then applying the equivalence of Lemma~\ref{lem-twist}. Since that equivalence occurs via $(d^{\{1\}})^*$, we obtain a natural isomorphism from this sheaf to $i^*\mc{F} =: \mc{F}_Y$, as desired. 
	
	For the upper-left triangle, we assume that $\mc{F}$ is the pullback to $\ran(Y)$ of some sheaf $\mc{G}$ on $\ran(Y)_{\dr}$. We have to show that the crystal structure on $\mc{F}_Y$ arising from $\Psi$ is equal to the crystal structure coming from $\mc{G}$. Since the map from a prestack to its de Rham prestack is functorial, we have a natural transformation 
	\[
		\QCoh \circ \Ran_{\wh{\Delta}}^{(-)}(Y)_{\dr} \Longrightarrow \QCoh \circ \Ran_{\wh{\Delta}}^{(-)}(Y). 
	\]
	Moreover, all the steps in the construction of $\Psi$ can be carried through for the de Rham versions of the prestacks involved, and the two versions are related by a natural transformation as above. In more detail, we have a commutative diagram 
	\begin{cd}
		\QCoh(\ran(Y)_{\dr}) \ar{d} \ar[r] & \QCoh(\ran(Y)) \ar[d] \\
		\displaystyle\lim_{\fset}(\QCoh \circ \Ran_{\wh{\Delta}}^{(-)}(Y)_{\dr}) \ar[r]  \ar{d}[rotate=90, anchor=south]{\sim} & \displaystyle\lim_{\fset}(\QCoh \circ \Ran_{\wh{\Delta}}^{(-)}(Y)) \ar[d] \\
		\displaystyle\lim_{\fsetne}(\QCoh \circ Y^{(-)}_{\wh{\Delta}}) \ar[r]  \ar{d}[rotate=90, anchor=south]{\sim} & \displaystyle\lim_{\fsetne}(\QCoh \circ Y^{(-)}_{\wh{\Delta}}) \ar[d] \\
		\displaystyle\lim_{\Delta}( \QCoh \circ L_{Y_{\dr}}) \ar[r, "\sim"] & \displaystyle\lim_{\Delta}( \QCoh \circ L_Y) 
	\end{cd}
	Tracing $\mc{G} \in \QCoh(\ran(Y)_{\dr})$ through the upper composition yields $\Psi(\mc{F})$ by definition. For the lower composition, note that the functors $\Ran_{\wh{\Delta}}^{(-)}(Y)_\dr$, $Y^{(-)}_{\wh{\Delta}}$, and $L_{Y_{\dr}}$ are constant with value $Y_{\dr}$, so the lower three terms in the left column are just $\QCoh(Y_{\dr})$, and with these identifications the labeled arrows become identity functors. Hence, $\mc{G}$ maps under the lower composition to $\mc{G}_{Y_{\dr}} \in \QCoh(Y_{\dr})$, so $\Psi(\mc{F}) \simeq \mc{G}_{Y_{\dr}}$ as desired. 
\end{proof}

\subsection{Proof of Theorem~\ref{thm-dr}}  \label{ssec-canon-proof}

We construct an inverse functor $\Phi : \QCoh(\ran(X))_{\on{flat}} \to \QCoh(\Ran(X)_{\dr})_{\on{flat}}$ as follows. Since Lemma~\ref{lem-psi} is functorial with respect to $Y$, we may take $Y = X^I$ to obtain a diagram 
\begin{cd}
	\QCoh(\ran(X^I)_\dr)_{\on{flat}} \ar[r] \ar[d] & \QCoh(\ran(X^I))_{\on{flat}} \ar[d] \ar[ld, "\Psi_{I}"] \\
	\QCoh(X^I_{\dr})_{\on{flat}} \ar[r] & \QCoh(X^I)_{\on{flat}}
\end{cd}
which is functorial with respect to $I \in \fsetsurjne$. We have an identification $(X^{I})^{(-)} \simeq X^{(I \times (-))}$ of functors $\fsetsurjne \to \on{PreStk}_{k}$, and combining this with the colimit version of Lemma~\ref{lims} yields a map  
\begin{align*}
	\ran(X^I) &= \colim (X^{I})^{(-)}  \\
	&\simeq \colim X^{(I \times (-))} \\
	& \quad \to \colim X^{(-)} \\
	& \quad \phantom{\to\ } = \ran(X) 
\end{align*}
which is functorial with respect to $I$. There is a similar map for the correpsonding de Rham prestacks, and pulling back along these maps allows us to expand the previous diagram: 
\begin{cd}
	\QCoh(\ran(X)_{\dr})_{\on{flat}} \ar[r] \ar[d] & \QCoh(\ran(X))_{\on{flat}} \ar[d] \\
	\QCoh(\ran(X^I)_\dr)_{\on{flat}} \ar[r] \ar[d] & \QCoh(\ran(X^I))_{\on{flat}} \ar[d] \ar[ld, "\Psi_{I}"] \\
	\QCoh(X^I_{\dr})_{\on{flat}} \ar[r] & \QCoh(X^I)_{\on{flat}}
\end{cd}
In fact, we focus solely on the outer square: 
\begin{cd}
	\QCoh(\ran(X)_{\dr})_{\on{flat}} \ar[r] \ar[d] & \QCoh(\ran(X))_{\on{flat}} \ar[d] \ar[ld, "\Phi_I"]  \\
	\QCoh(X^I_{\dr})_{\on{flat}} \ar[r] & \QCoh(X^I)_{\on{flat}}
\end{cd}
Here $\Phi_I$ is the composition of $\Psi_I$ with the upper-right vertical map. Taking the limit of this diagram over $I \in \fsetsurjne$, and using that $\lim_{\fsetsurjne}(\QCoh(X^{(-)})_{\on{flat}}) \simeq \QCoh(\Ran(X)_{\on{flat}})$, we obtain a diagram 
\begin{cd}
	\QCoh(\ran(X)_{\dr})_{\on{flat}} \ar[r] \ar[d] & \QCoh(\ran(X))_{\on{flat}} \ar[d] \ar[ld, "\Phi"]  \\
	\QCoh(\ran(X)_{\dr})_{\on{flat}} \ar[r] & \QCoh(\Ran(X))_{\on{flat}}
\end{cd}
where the vertical maps are the identity functors and the horizontal maps are both given by pullback along $\Ran(X) \to \Ran(X)_{\dr}$. Since this diagram is strictly commutative, we conclude that $\Phi$ is an inverse to the horizontal maps. This concludes the proof of Theorem~\ref{thm-dr}.

\section{Triviality of $\Pic(\Ran(X))$} \label{sec-pic-triv}

\subsection{The Picard groupoid as a quadratic functor} \label{ssec-pic-triv-1}
We use the results of Section~\ref{sec-poly} to show that line bundles on $\Ran(X)$ are trivial. For this, we need $I \rightsquigarrow \Pic(X^I)$ to be an $n$-excisive functor defined on $\fset_*$. In Proposition~\ref{prop-pic-1}, hypothesis (i) will allow us to pick a basepoint on $X$, which ensures this functor is defined on $\fset_*$. Hypothesis (ii) will allow us to apply the Theorem of the Cube to conclude that this functor is quadratic. Unfortunately for our purposes, this theorem is usually stated for the Picard group rather than the Picard groupoid, so we will need to discuss the $\pi_1$ and $\pi_0$ terms of the groupoid separately. However, as a corollary of this discussion, we deduce the corresponding statement for Picard groupoids, which can be interpreted as an improved version of the Theorem of the Cube, see Remark~\ref{cube-improved}. 

\subsubsection{} Here is the main result of this subsection: 
\begin{prop}\label{prop-pic-1} 
	Let $k$ be a field, and let $X$ be an algebraic variety over $k$ which satisfies the following properties: 
	\begin{enumerate}[label=(\roman*)]
		\item $X(k)$ is nonempty. 
		\item $X$ admits an open embedding into a smooth proper geometrically integral $k$-variety. 
	\end{enumerate} 
	Then the pullback functor $\Pic(\Spec k) \to \Pic(\Ran(X))$ is an equivalence. 
\end{prop}
\begin{proof}
	In~\ref{defs}, we noted the equivalence
	\[
		\Pic(\Ran(X)) \simeq \lim_{I \in \fsetsurjne} \Pic(X^I). 
	\]
	Recall that the category of strictly commutative Picard groupoids is equivalent to $\on{D}(\on{AbGrp})^{[-1, 0]}$, which is the full subcategory consisting of complexes in degrees $-1$ and $0$.\footnote{For details regarding this point, see~\cite[Sect.\ 1.4]{sga43}.} In $\on{D}(\on{AbGrp})$, we have 
	\[
		\lim_{I \in \fsetsurjne} \Pic(X^I) \simeq \tau^{\le 0} \left(\holim_{I \in \fsetsurjne} \on{R}^\cdot \Gamma(X^I, \oh_{X^I}^\times)[1] \right), 
	\]
	and similarly $\Pic(\Spec k) \simeq k^\times[1]$. Therefore, to prove the proposition it suffices to show the following two statements:   
	\begin{itemize}
		\item $\mathrm{H}^0 \holim_{I \in \fsetsurjne} \on{R}^\cdot\Gamma(X^I, \oh_{X^I}^\times)  \simeq k^\times$
		\item $\mathrm{H}^1  \holim_{I \in \fsetsurjne} \on{R}^\cdot\Gamma(X^I, \oh_{X^I}^\times) \simeq 0$ 
	\end{itemize}
	Applying the spectral sequence of \cite[7.1]{bk}, we reduce to proving these statements: 
	\begin{enumerate}
		\item[(A1)] $\nlim{0}_{I \in \fsetsurjne} \on{R}^0\Gamma(X^I, \oh_{X^I}^\times) \simeq k^\times$. 
		\item[(A2)] $\nlim{1}_{I \in \fsetsurjne} \on{R}^0\Gamma(X^I, \oh_{X^I}^\times) \simeq 0$. 
		\item[(B)] $\nlim{0}_{I \in \fsetsurjne} \on{R}^1\Gamma(X^I, \oh_{X^I}^\times) \simeq 0$
	\end{enumerate}
	This is accomplished in~\ref{van1} and \ref{van2}. 
\end{proof}

\subsubsection{Setup} By assumption (i), we may choose a basepoint $x_0 \in X(k)$. This gives a functor $\fset_*^{\op} \to \on{Sch}_{k}$ which sends $\{*\} \sqcup I \mapsto X^I$. 

\subsubsection{Proof of (A1) and (A2)} \label{van1} 
Define the functor $G : \fset_* \to \on{AbGrp}$ by the assignment $\{*\} \sqcup I \mapsto \Gamma(X^I, \oh_{X^I}^\times)$. 

\begin{lem} \label{want-linear} 
	The functor $G$ is linear in the sense of Definition~\ref{def-poly}. 
\end{lem}
\begin{proof}
	By the criterion of Lemma~\ref{lem-poly}(i), it suffices to show that, for every $I$ with $|I| > 1$, every function in $\Gamma(X^I, \oh_{X^I}^\times)$ is a product of pullbacks of functions in $\Gamma(X^{I \setminus i}, \oh_{X^{I \setminus i}}^\times)$ for $i \in I$. This follows from the well-known fact that an invertible function on a product of algebraic varieties is an external product of invertible functions on each factor, see~\cite[Lem.\ 5.1.15]{cg}.
\end{proof}

Now Remark~\ref{rmk-compare} tells us that $G$ is $1$-excisive, when viewed as a functor to $\on{D}(\on{AbGrp})$. Applying Theorem~\ref{thm2}, we conclude that 
\e{
	\holim_{I \in \fsetsurjne} \on{R}^0 \Gamma(X^I, \oh_{X^I}^\times)&\simeq \on{R}^0\Gamma(\Spec k, \oh^\times) \\
	&\simeq k^\times
} 
as complexes of abelian groups. This proves (A1) and (A2). 

\subsubsection{Proof of (B)} \label{van2}

Let $X \hra \ol{X}$ be a compactification of $X$ where $\ol{X}$ is smooth and geometrically integral, as guaranteed by the assumption (ii). Define a functor $\wt{F} : \fset_* \to \on{AbGrp}$ by $I \mapsto \on{Pic}(\ol{X}^I)$. 

\begin{lem}
	The functor $\wt{F}$ is quadratic in the sense of Definition~\ref{def-poly}. 
\end{lem}
\begin{proof}
	Fix $I$ with $|I| > 2$. We can write $I = \{1, 2\} \sqcup J$ for nonempty $J$. Applying the Theorem of the Cube~ \cite[\href{https://stacks.math.columbia.edu/tag/0BF4}{Tag 0BF4}]{stacks} to the product decomposition $\ol{X}^I \simeq \ol{X} \times \ol{X} \times \ol{X}^J$, we conclude that an element of $\on{Pic}(\ol{X}^I)$ is zero if and only if its images in 
	\[
		\begin{array}{c}
		\on{Pic}(\Spec k \times \ol{X} \times \ol{X}^J) \\
		\on{Pic}(\ol{X} \times \Spec k \times \ol{X}^J) \\
		\on{Pic}(\ol{X} \times \ol{X} \times \Spec k)
		\end{array}
	\]
	under pullback are all zero. (This is where we use the hypothesis that $\ol{X}$ is proper and geometrically integral.) This implies the criterion of Lemma~\ref{lem-poly}(ii) for $n=2$. 
\end{proof}

Let $F : \fset_* \to \on{AbGrp}$ be defined by $I \mapsto \on{Pic}(X^I)$. 

\begin{cor} \label{want-quadratic} 
	The functor $F$ is quadratic. 
\end{cor}
\begin{proof}
	Restriction of line bundles from $\ol{X}^I$ to $X^I$ defines a natural transformation $\wt{F} \Rightarrow F$. Because $\ol{X}$ is smooth, every line bundle on $X^I$ extends to one on $\ol{X}^I$. In other words, this natural transformation is a surjection, and Lemma~\ref{thicc} proves the corollary. 
\end{proof}

Now Proposition~\ref{prop1} implies that $\nlim{0} (F \circ \iota) \simeq \on{Pic}(\Spec k) = 0$, and this proves (B). 

\subsubsection{Remark} 

By a one-step d\'evissage argument similar to that of Lemma~\ref{lem-induct}, we conclude from Lemma~\ref{van1} and Corollary~\ref{van2} that the functor 
\[
	\fset_* \to (\text{strictly commutative Picard groupoids}) \simeq \on{D}(\on{AbGrp})^{[-1, 0]}
\]
defined by $I \mapsto \Pic(X^I)$ is $2$-excisive. (But this does not imply that the functor is $2$-excisive if its target is taken to be $\on{D}(\on{AbGrp})$.) Applying Proposition~\ref{paring} yields the following improvement of the Theorem of the Cube, under the hypotheses of Proposition~\ref{prop-pic-1}: 
\begin{cor} \label{cube-improved} 
	Let $n \ge 3$. The category of line bundles on $X^n$ is equivalent to the category consisting of the following data: 
	\begin{itemize}
		\item We have a one-dimensional $k$-vector space $\mc{F}$. 
		\item For each $i \in [n]$, we have a line bundle $\mc{E}_i \in \Pic(X)$ and an isomorphism 
		\[
			q_i : \mc{E}_i|_{x_0} \xrightarrow{\sim} \mc{F}.
		\]
		\item For each pair $i, j$ with $1 \le i < j \le n$, we have a line bundle $\mc{L}_{i, j} \in \Pic(X^2)$ and isomorphisms 
		\e{
			\sigma_{i,j} : \mc{L}_{i, j}|_{X \times \{x_0\}} &\xrightarrow{\sim} \mc{E}_i \\
			\tau_{i,j} : \mc{L}_{i, j}|_{\{x_0\} \times X} &\xrightarrow{\sim} \mc{E}_j.
		}
		\item These data are subject to the condition that, for any $i < j$ as above, we have 
		\[
			q_i \circ \sigma_{i, j}|_{x_0} = q_j \circ \tau_{i, j}|_{x_0}
		\]
		as maps $\mc{L}_{i, j}|_{(x_0, x_0)} \xrightarrow{\sim} \mc{F}$. 
	\end{itemize}
\end{cor}
These conditions correspond to the $\binom{n}{2}$ square facets in a partial hypercube $\Xi^{\le 2}$ in dimension $n$. 

It is straightforward to modify the proof of this corollary to apply to a product $\prod_{i=1}^n X_i$ of possibly different varieties $X_i$, each satisfying the hypotheses of Proposition~\ref{prop-pic-1}. This yields an improved Theorem of the Cube for noncompact smooth varieties. We explain the idea below: 

\begin{myproof}{Sketch of the proof}
	Let $\mc{C}$ be the category whose objects are pairs $(\{*\} \sqcup I, \phi)$ where the first element is a pointed finite set and the second element is an injective map $I \overset{\phi}{\hra} [n]$. A morphism from $(\{*\} \sqcup I_1, \phi_1)$ to $(\{*\} \sqcup I_2, \phi_2)$ is a map $\xi : \{*\} \sqcup I_1 \to \{*\} \sqcup I_2$ satisfying the following property: 
	\begin{itemize}
		\item For every $i \in I_1$, if $\xi(i) \neq *$, then $\phi_1(i) = \phi_2(\xi(i))$. 
	\end{itemize}
	The category $\mc{C}$ along with the object $(\{*\} \sqcup [n], \id_{[n]}) \in \mc{C}$ is the universal example of a category with an object which has $n$ commuting split idempotent endomorphisms. There is a faithful embedding $\mc{C} \hra \fset_{*, \le n}$ given by forgetting $\phi$. 
	
	The datum of the varieties $(X_1, \ldots, X_n)$ and their chosen basepoints yields a functor $\mc{C}^{\op} \to \on{Sch}_{k}$ which sends 
	\[
		(\{*\} \sqcup I, \phi) \mapsto \prod_{i \in I} X_{\phi(i)}. 
	\]
	Post-composing with $\Pic$, we obtain a functor $G : \mc{C} \to \on{D}(\on{AbGrp})^{[-1,0]}$. 
	
	One can formulate the notion of $n$-excision for functors defined on $\mc{C}$ by transferring Definition~\ref{def-poly-fin}. The arguments of~\ref{van1} and \ref{van2} work just as well in this setting, because the usual Theorem of the Cube does not require the varieties to be the same, and this proves that $G$ is 2-excisive. Then the analogue of Proposition~\ref{paring} finishes the proof. The reason these definitions and results transfer to the setting of functors defined on $\mc{C}$ is because they only use the split endomorphisms in $\fset_{*, \le n}$ and not the permutations.
\end{myproof}

These comments also apply when each $X_i$ is proper and geometrically integral over $k$ but not necessarily smooth (these are the usual hypothesis for the Theorem of the Cube). Moreover, in this case, we can take one of the $X_i$ to be an arbitrary $k$-variety. 

\subsection{The hypothesis (C)} \label{ssec-pic-triv-2}
Using Galois descent for line bundles, we weaken the hypotheses of Proposition~\ref{prop-pic-1} as follows: 

\begin{cor}\label{cor-pic-1} 
	Let $k$ be a field, and let $X$ be an algebraic variety over $k$ which satisfies the following property: 
	\begin{itemize}
		\item[(C)] $X_{\ol{k}}$ admits an open embedding into a smooth proper connected $\ol{k}$-variety. 
	\end{itemize}
	Then the pullback functor $\Pic(\Spec k) \to \Pic(\Ran(X))$ is an equivalence. 
\end{cor}
\begin{proof}
	Let $i : X_{\ol{k}} \hra Y$ be the open embedding guaranteed by (C), where $Y$ is a smooth, proper, and connected $\ol{k}$-variety. Since $i$ is a map between finite type $\ol{k}$-schemes, there exists a finite sub-extension $E/k$ (where $E \subset \ol{k}$), a variety $Y_0$ over $E$, and an open embedding $i_0 : X_{E} \hra Y_0$ such that the base change of $i_0$ along $E \hra \ol{k}$ identifies with $i$. Furthermore, $Y_0$ is smooth, proper, and geometrically connected over $E$. Since $X(\ol{k})$ is nonempty (by the Nullstellensatz), we can ensure that $X(E)$ is nonempty by choosing $E$ large enough. 
	
	For each finite set $I$, let $(\pi^I)_* \ul{\Pic}(X^I)$ denote the fppf sheaf of groupoids over $\Spec k$ which sends a $k$-scheme $S$ to the Picard groupoid $\Pic(S \times X^I)$. The case $I = \emptyset$ is $\ul{\Pic}(\Spec k)$. We have a map 
	\[
		\psi : \ul{\Pic}(\Spec k) \to \lim_{I \in \fset^{\on{surj}}_{\on{n.e.}}} (\pi^I)_* \ul{\Pic}(X^I), 
	\]
	and taking global sections yields the map $\Pic(\Spec k) \to \Pic(\Ran(X))$. This is because the functor of global sections commutes with limits. 
	
	We now show that $\psi$ is an equivalence. By fppf descent for the cover $\Spec E \to \Spec k$, it suffices to show that $\psi$ induces an equivalence on spaces of sections over any fppf cover of $\Spec k$ of the form 
	\[
		\Spec E \underset{\Spec k}{\times} \Spec E \underset{\Spec k}{\times} \cdots \underset{\Spec k}{\times} \Spec E.
	\]
	Each such cover is isomorphic to a disjoint union of covers of the form $\Spec E'$ where $E' \subset \ol{k}$ is a finite extension of $k$ which contains $E$. Thus, it suffices to show that the map 
	\[
		\Gamma(\Spec E', \psi) : \Pic(\Spec E') \to \lim_{I \in \fset^{\on{surj}}_{\on{n.e.}}} \Pic(X_{E'}^I) \simeq \Pic(\Ran_{E'}(X_{E'}))
	\]
	is an equivalence. (Here $\Ran_{E'}(X_{E'})$ denotes the Ran space construction taken over $E'$ rather than $k$.) The constructions of the first paragraph can be base-changed along $E \hra E'$, and this shows that the hypotheses of Proposition~\ref{prop-pic-1} apply to $X_{E'}$ as a variety over $E'$, so $\Gamma(\Spec E', \psi)$ is an equivalence, as desired. 
\end{proof}

\section{Relative Pic-contractibility of $\Ran(X)$}  \label{sec-rel-pic}

Our goal is to prove the following relative version of Corollary~\ref{cor-pic-1}: 

\begin{thm} \label{main2} 
	Let $k$ be a field, and let $X$ be an algebraic variety over $k$ which satisfies the following property: 
	\begin{enumerate}[label=(\roman*)] 
		\item[(C)] The base change $X_{\ol{k}}$ admits an open embedding into a smooth proper $\ol{k}$-variety.
	\end{enumerate} 
	Then, for any locally Noetherian $k$-scheme $S$, the pullback functor $\Pic(S) \to \Pic(S \times \Ran(X))$ is an equivalence. 
\end{thm}

\subsubsection{Remark} \label{rmk-ega}

Our strategy is a variation on the method of `reduction to the case of an Artinian local ring with separably closed residue field.' Let us explain this method in more detail. When proving a statement about a morphism $f : Y \to Z$, one makes the following reductions: 
\begin{enumerate}
	\item If the statement is suitably local on the base, one can replace $Z$ by an affine open subscheme, and subsequently by $\Spec A$ where $(A, \mf{m})$ is a local ring. 
	\item First, study the problem when $A$ is an Artinian local ring with separably closed residue field. In this case, one can use deformation theory. 
	\item Via the theory of formal schemes, extend from the case of an Artinian local ring to that of a complete local ring (with separably closed residue fields). 
	\item If $A$ is an arbitrary local ring, descend the result of step (3) to the strict henselization $A^{\on{sh}}$ and subsequently to some \'etale $A$-algebra. 
\end{enumerate}
These steps are paraphrased from~\cite[p.\ 8--9]{ega}. Our morphism of interest is the projection 
\[
	\pr_1 : S \times \Ran(X) \to S,
\]
but this lies outside the usual realm of application of the method in several ways: 
\begin{itemize}
	\item This method is suited to proving statements that are `finitely presented' in the sense that they assert the existence of solutions to finitely many polynomial equations. Such statements are nice because if they are true in a filtered colimit (of rings) then they are true at a finite stage in that colimit. Our situation involves an infinite set of data (a section on $X^I$ for each $I$) so this principle does not apply. We circumvent this issue by judiciously considering finite limits or focusing attention on $X^n$ for one $n$ at a time.  
	\item A more serious obstruction is that this method is adapted to proving \emph{properties}, whereas we want to construct \emph{data} consisting of trivializations of some line bundles. An identity of functions on a Noetherian scheme can be checked on formal neighborhoods of closed points, but it is not \emph{a priori} clear how to similarly localize the problem of constructing a function. For this reason, given a closed subscheme $Z' \subset Z$, we set up a framework for gluing two functions defined on the formal neighborhood of $Z'$ in $Z$ and on the complement $Z \setminus Z'$, respectively (see~\ref{bl-sec}). The ability to glue functions along formal neighborhoods allows us to dispense with step (4) entirely. 
	\item The method is usually applied when $f$ is proper, which assists in carrying out step (3) because tools like Grothendieck's existence theorem become available for controlling coherent sheaves on $Y$, see~\cite[Thm.\ 8.4.2]{fga}. In our situation, suppose $S$ is affine, let $S' \subset S$ be a closed subscheme, let $S_{\on{inf}}$ be the formal neighborhood of $S'$ in $S$, and let $\wh{S}$ be the Spec of the completed local ring, so there are maps $S_{\on{inf}} \to \wh{S} \to S$. If $X$ were proper, the Grothendieck existence theorem would imply that a section on $S_{\on{inf}} \times X^I$ automatically descends to one on $\wh{S} \times X^I$. 
	
	Since we do not assume that $X$ is proper, this technique is not available. Instead, we assume that $X$ is affine, which allows us to work with the Spec of the ring of functions of $S_{\on{inf}} \times X^I$, denoted $(S \times X^I)^\wedge$. Then a weaker form of Grothendieck's existence theorem implies that a section on $S_{\on{inf}} \times X^I$ descends to one on $(S \times X^I)^\wedge$. The price we pay is that we have to work harder to prove the triviality of functions on the completed version of $S \times \Ran(X)$ (see Lemma~\ref{func2}) and we have to use a covering argument to reduce to the case of affine $X$ (see~\ref{ssec-general}). 	
\end{itemize}

In what follows, one could identify step (1) with Corollary~\ref{sheaf}, step (2) with Lemma~\ref{formal}, step (3) with the application of Grothendieck's existence theorem in Proposition~\ref{induct}, and step (4) does not appear as explained in the second bullet point above. The passage to a separably closed residue field corresponds to the Galois descent trick which was used in Corollary~\ref{cor-pic-1} and which appears again in~\ref{based-based}. 

\subsection{Basepoints and rigidification} \label{based} 

\subsubsection{Basepoints} \label{based-based}

In this section, we will be studying functions and line bundles, both of which satisfy fppf descent. Therefore, by the same trick used in the proof of Corollary~\ref{cor-pic-1}, we may assume that $X(k)$ is nonempty by passing to a finite extension of $k$. 

Now, if $X(k)$ is nonempty, we can pick a basepoint $x_0 \in X(k)$, which gives basepoints $x_0^I \in X^I(k)$ and $x_0^{\on{Ran}} \in \Ran(X)(k)$. 

\subsubsection{Rigidification} \label{rigid}

Upon choosing these basepoints, the association $\{*\} \sqcup I \mapsto X^I$ becomes a functor from $\fset_*^{\op}$ to the category of $k$-schemes equipped with a $k$-rational point. In other words, these basepoints $x_0^I$ are compatible with all generalized diagonal maps, projection maps, and inclusion maps. 

Pulling back along the inclusion maps of the basepoints of $X^I$ gives retracts to the pullback maps $\Pic(\Spec k) \to \Pic(X^I)$ which are functorial in $I$. We define $\Pic^e(X^I)$ to be the kernel of this retract.\footnote{The $e$-superscript notation for rigidified line bundles is inspired by~\cite[Sect.\ 3.4]{z}.} Concretely, $\Pic^e(X^I)$ is the Picard groupoid of line bundles on $X^I$ equipped with trivialization at $x_0^I$. We refer to such a datum as a \emph{rigidified line bundle}. Taking the limit over $\fsetsurjne$, we obtain a retract of the map $\Pic(\Spec k) \to \Pic(\Ran(X))$ and a Picard groupoid $\Pic^e(\Ran(X))$ of rigidified line bundles. 

We also make analogous definitions for the functor $\{*\} \sqcup I \mapsto \Gamma(X^I, \oh)$. Namely, for each $I$, the subspace $\Gamma^e(X^I, \oh) \subset \Gamma(X^I, \oh)$ consists of functions which are zero on $x_0^I$, and similarly for the subspace $\Gamma^e(\Ran(X), \oh) \subset \Gamma(\Ran(X), \oh)$. We refer to functions vanishing on the basepoint as \emph{rigidified functions}. 

If $S$ is any $k$-scheme, we can similarly define $\Pic^e(S \times X^I), \Pic^e(S \times \Ran(X)), \Gamma^e(S \times X^I, \oh)$, and $\Gamma^e(S \times \Ran(X), \oh)$. For example, a rigidified line bundle on $S \times X^I$ is a line bundle on $S \times X^I$ equipped with a trivialization on $S \times \{x_0^I\}$. 

It is clear that $\Pic(S) \to \Pic(S \times \Ran(X))$ is an equivalence if and only if $\Pic^e(S \times \Ran(X))$ is trivial, and there is the analogous statement for $\Gamma(S \times \Ran(X))$. 

\subsection{Triviality of functions} \label{ssec-func}

Our plan is to show that line bundles on $S \times \Ran(X)$ are trivial by first trivializing them over certain completions and then gluing these trivializations. To make this gluing manageable, we need to first prove that the trivializations over the completions are unique. Since automorphisms of the trivial line bundle are given by nonvanishing functions, this amounts to showing that there is a paucity of such functions on $S \times \Ran(X)$ (or a completion thereof).  

\subsubsection{} We need a basic commutative algebra lemma. \label{basic}

\begin{lem}
	Let $S$ be a locally Noetherian $k$-scheme, and let $Y$ be a smooth $k$-scheme. Then the associated points of $S \times Y$ are exactly the generic points of the fibers of $S \times Y \xrightarrow{\pr_1} S$ over the associated points of $S$. 
\end{lem}
\begin{proof}
	We immediately reduce to the case when $S = \Spec A$ is affine. By considering \'etale maps $U \to \BA^n$ for open subschemes $U \subset Y$ covering $Y$, we reduce to the case when $Y = \BA^n$. By inducting on $n$, we reduce to the case when $n = 1$. By localizing at a point of $A$, we reduce to the two following statements: 
	
	Let $(A, \mf{m})$ be a local $k$-algebra.
	\begin{enumerate}[label=(\roman*)]
		\item Let $\mf{m}^e \subset A[x]$ be the extension of $\mf{m}$, i.e.\ the generic point of the fiber over $\mf{m}$. Then $\mf{m}^e$ is an associated point of $A[x]$ if and only if $\mf{m}$ is an associated point of $A$. 
		\item Let $\mf{p} \subset A[x]$ be a point lying over $\mf{m}$ which is not the generic point of the fiber over $\mf{m}$. Then $\mf{p}$ is not an associated point of $A[x]$. 
	\end{enumerate}
	
	For (i), the `only if' follows from the fact that flat maps send associated points to associated points~\cite[Exer.\ 24.2.J]{v}. Conversely, if $\mf{m}$ is an associated point of $A$, then there exists $f \in A$ such that $\on{Ann}_A(f) = \mf{m}$. Then $\on{Ann}_{A[x]}(f) = \mf{m}^e$, so $\mf{m}^e$ is associated. 
	
	For (ii), there exists a \emph{monic} polynomial $P(x) \in A[x]$ whose vanishing locus contains the point $\mf{p}$. If $\mf{p}$ is associated, then $P(x)$ must be a zerodivisor in $A[x]$. However, no monic polynomial in a polynomial ring can be a zerodivisor, contradiction. 
\end{proof}

\subsubsection{}
The following lemma reduces us to studying functions on the infinitesimal Ran space. 

\begin{lem}\label{supp} 
	Let $S = \Spec A$ where $A$ is a Noetherian $k$-algebra. If a regular function $f \in \Gamma(S \times X^I, \oh)$ vanishes on the formal neighborhood of $S \times \{x_0\}$, then $f =0 $. 
\end{lem}
\begin{proof}
	By the Krull intersection theorem (which uses that $A$ is Noetherian), we conclude that $\on{Supp} f$ is disjoint from $S \times \{x_0\}$. However, $\on{Supp} f$ is a union of closures of associated points of $S \times X^I$, and Lemma~\ref{basic} says that all such closures are of the form $Z \times X^I$ for some closed $Z \subset S$. Since each $Z \times X^I$ intersects $S \times \{x_0\}$ if $Z$ is nonempty, the only possibility is that $\on{Supp} f$ is empty, i.e.\ $f = 0$. 
\end{proof}

\subsubsection{} 
We prove the following lemma for expository purposes only. We actually need the stronger version Lemma~\ref{func2}, which applies to a completed version of $S \times \Ran(X)$. 

\begin{lem}\label{func1} 
	Let $S$ be a locally Noetherian $k$-scheme. Then the pullback map $\Gamma(S, \oh) \to \Gamma(S \times \Ran(X), \oh)$ is an isomorphism. 
\end{lem}
\begin{proof}
	In view of~\ref{rigid}, it is equivalent to prove that $\Gamma^e(S \times \Ran(X), \oh) = 0$. We may assume that $S = \Spec A$ for a Noetherian $k$-algebra $A$. 
	
	\begin{claim}
		We have $\Gamma^e(\Ran^n_{\on{inf}, S}, \oh) = 0$. 
	\end{claim}
	\begin{proofofclaim}
		First, observe that 
		\e{
			\Gamma^e(\Ran^n_{\on{inf}, S}, \oh) &\simeq \lim_{I \in \fsetsurjne} \Gamma^e(\Ran^n_{\on{inf}, S}(\{*\} \sqcup I), \oh) \\
			&\simeq \prod_{d=1}^\infty \lim_{I \in \fsetsurjne} \mc{P}_{S, n, d}(\{*\} \sqcup I). 
		} 
		But Lemma~\ref{rnd-cart} tells us that $\mc{P}_{S, n, d}$ is $d$-excisive, so Proposition~\ref{prop1} implies that 
		\[
			\lim_{I \in \fsetsurjne} \mc{P}_{S, n, d}(\{*\} \sqcup I) \simeq \mc{P}_{S, n, d}(\{*\}). 
		\]
		The right hand side is $A$ if $d = 0$ and $0$ otherwise, and the claim follows. 
	\end{proofofclaim}
	
	An element of $\Gamma^e(S \times \Ran(X), \oh)$ is a compatible family of elements $a_I \in \Gamma^e(S \times X^I, \oh)$. The claim tells us that the restriction of $a_I$ to the formal neighborhood of $S \times \{x_0^I\}$ is zero. But then $a_I = 0$ by Lemma~\ref{supp}. 
\end{proof}

\begin{rmk}
	As mentioned in Remark~\ref{discuss-3}, it is notable that this proof works for a field $k$ of arbitrary characteristic. See Remark~\ref{exercise} for more discussion of this point. 
\end{rmk}

\subsubsection{Remark} \label{constant} 

It is not hard to show, using the claim proved in~\ref{func1}, that any map $\Ran(X) \to S$ to any scheme $S$ is constant. The idea is to restrict to various infinitesimal neighborhoods $\Ran^n_{\on{inf}} \to \Ran(X)$. The previous claim implies that the maps $\Ran^n_{\on{inf}}\to S$ are constant, and this allows one to show that the restrictions $X^I \to S$ are constant, which implies that $\Ran(X) \to S$ was constant to begin with. 

\subsubsection{} \label{sheaf} We deduce that the assignment $S \mapsto \Pic^e(S \times \Ran(X))$, which is \emph{a priori} a sheaf valued in Picard groupoids, is actually a sheaf valued in abelian groups. 
\begin{cor}
	Let $S$ be a locally Noetherian $k$-scheme. Then the groupoid $\Pic^e(S \times \Ran(X))$ is equivalent to a set. Equivalently, rigidified line bundles on $S \times \Ran(X)$ admit no nontrivial automorphisms. 
\end{cor}
\begin{proof}
	Define a \emph{rigidified invertible function} on $S \times \Ran(X)$ to be an invertible function on $S \times \Ran(X)$ which is equal to 1 on $S \times \{x_0^{\on{Ran}}\}$. If $f$ is a rigidified invertible function, then $1-f$ is a rigidified function, and Lemma~\ref{func1} proves that $1-f = 0$, so $f = 1$. Therefore, any rigidified invertible function is identically equal to 1. 
	
	An automorphism of a (rigidified) line bundle is given by a (rigidified) invertible function, so the previous paragraph shows that rigidified line bundles have no nontrivial automorphisms, as desired. 
\end{proof}

\subsubsection{Punctured completions} \label{punctured}
Assume that $X$ is affine. This assumption is needed because, in general, the colimits over $m$ in the following paragraph exist in the category of \emph{affine} schemes, but not in the category of all schemes. 

Let $A$ be a Noetherian $k$-algebra, let $f \in A$, and define $S = \Spec A$. Let $T \subset A$ be a multiplicative system each of whose elements maps to a non-zerodivisor in $A_f$,\footnote{Equivalently, no element of $T$ vanishes on an associated point of $A_f$. In fact, we will only use this construction when $f$ vanishes on the embedded points of $A$, so $A_f$ has no embedded points, in which case the requirement on $T$ is simply that its elements do not vanish on any generic point of $A$.} and let ${fT}$ be the multiplicative system generated by $f$ and $T$. Define 
\e{
	\wh{S} &:= \colim_m V(f^m) \\
	(S\times X^I)^\wedge &:= \colim_m V(f^m) \times X^I
}
where the limits are taken in the category of affine schemes. If $X = \Spec R$, then $(S \times X^I)^\wedge$ is the spectrum of the completion of $A \otimes R^{\otimes m}$ at the ideal $(f)$. Define 
\begin{align*} 
\wh{S}_{fT} &:= \text{localization of $\wh{S}$ by ${fT}$} \\
(S \times X^I)^\wedge_{fT} &:= \text{localization of $(S \times X^I)^\wedge$ by ${fT}$}. 
\end{align*} 
There is evidently a map $(S \times X^I)^\wedge_{fT} \to \wh{S}_{fT}$. 

In this way, we obtain a functor $\fset_* \to \on{Sch}^{\on{aff}}_{k}$ which sends $\{*\} \sqcup I \mapsto (S \times X^I)^\wedge_{fT}$. Let
\[
\mc{Y} := \colim_{\fsetsurjneop} (S \times X^I)^\wedge_{fT}, 
\]
evaluated in the category of prestacks. The maps defined at the end of the previous paragraph combine to give a map $\mc{Y} \to \wh{S}_{fT}$. 

\begin{rmk}
	In the proof of the crucial Proposition~\ref{final}, we will need to glue trivializations of a line bundle on $S \times X^I$ using the `cover' of $S \times X^I$ consisting of the localization $S_f \times X^I$ and the completion $(S \times X^I)^\wedge$, for a suitably chosen $f$. This fits the pattern of Beauville-Laszlo gluing, and we summarize the requisite lemmas from commutative algebra in~\ref{bl-sec}. The upshot is that the `overlap' for the gluing procedure is the punctured completion $(S \times X^I)^\wedge_f$ defined above, and this is why we need to study functions on $(S \times X^I)^\wedge_f$ and $\mc{Y}$. 
\end{rmk}

\subsubsection{Reduction} \label{reduction} 

In~\ref{punctured}, the requirement that $T$ maps to non-zerodivisors in $A_f$ is engineered so that, when studying the punctured completions $\mc{Y}$ and $(S \times X^I)^\wedge_{fT}$, we may assume that $f$ and $T$ consist only of non-zerodivisors in $A$. In the next paragraph, we explain how to perform this reduction step: 

Let $\Ann_A(f) \subset A$ be the ideal consisting of elements annihilated by a power of $f$. This is set-theoretically supported on $V(f)$, so we can make the following replacement without changing $\mc{Y}$ or $\wh{S}_{fT}$: 
\begin{itemize} 
	\item Replace $A$ by $A' := A / \Ann_A(f)$. 
	\item Replace $f$ and $T$ by their images in $A'$. 
\end{itemize} 
By performing this replacement, we may assume that $f$ is a non-zerodivisor on $S$. Then $fT$ consists only of non-zerodivisors in $A$.\footnote{In geometric terms, $f$ does not vanish on any associated point of $\Spec A'$, and the hypothesis on $T$ then implies that none of its elements vanish on any associated point of $\Spec A'$.}

If this is the case, then the pullback of $fT$ to $S \times X^I$ and $(S \times X^I)^\wedge$ consists only of non-zerodivisors. Indeed, upon noting that a function is a non-zerodivisor if and only if it does not vanish on any associated point, this follows from Exercise~24.2.J and Theorem~29.2.6(a) in~\cite{v}. (This uses the Noetherian hypothesis on $A$.) Therefore 
\[
	\oh_{(S \times X^I)^\wedge_{fT}} \to \oh_{(S \times X^I)^\wedge}
\]
is an injection. So any function on $(S \times X^I)^\wedge_{fT}$ admits at most one extension to $(S \times X^I)^\wedge$. 

\subsubsection{Rigidification for punctured completions} \label{rigid-variation}

As in~\ref{rigid}, the basepoint $x_0^I \in X^I$ determines maps $\wh{S}_{fT} \to (S \times X^I)^\wedge_{fT}$ which are functorial in $I$, so we can define rigidified line bundles and functions, denoted $\Pic^e((S \times X^I)^\wedge_{fT}), \Pic^e(\mc{Y}), \Gamma^e((S \times X^I)^\wedge_{fT}, \oh)$, and $\Gamma^e(\mc{Y}, \oh)$. These are functors defined on $\fset_*$. 

\subsubsection{} Now we study functions on the punctured completion of $S \times \Ran(X)$. 
\begin{lem} \label{func2}
	If $A$ is Noetherian, then any regular function on $\mc{Y}$ is pulled back from $\wh{S}_{fT}$. 
\end{lem} 
\begin{proof}
	By~\ref{reduction}, we may assume that $fT$ consists only of non-zerodivisors in $A$. In view of~\ref{rigid} and \ref{rigid-variation}, it suffices to prove that $\Gamma^e(\mc{Y}, \oh) = 0$. Consider an element of $\Gamma^e(\mc{Y}, \oh)$, which is a compatible family of elements $a_I \in \Gamma^e((S \times X^I)^\wedge_{fT}, \oh)$ for $I \in \fsetsurjne$. 
	
	Let $N > 0$ be an integer. There exists $g_N \in {fT}$ such that $g_N\, a_I$ extends to $(S \times X^I)^\wedge$ for all $|I| \le N$. If $g_N$ has this property, then so does any multiple of $g_N$. We may therefore assume that, for each pair of integers $N < N'$, there exists an element $r_{N, N'} \in {fT}$ such that $g_{N'} = r_{N, N'} \, g_{N}$. 
	
	In this paragraph, we consider a fixed finite set $I$. Let $\mf{m}$ be the maximal ideal of the closed point $x_0^I \in X^I$. For integers $m, d\ge 0$, we have the following commutative diagram of affine schemes: 
	\begin{cd}[column sep = 0.4in]
		V(f^m) \times V(\mf{m}^{d+1}) \ar[r, "\colim_d"] \ar[d, "\colim_m"] & (V(f^m) \times X^I)^{\wedge, \mf{m}} \ar[d, "\colim_m"] \ar[r] & V(f^m) \times X^I \ar[d, "\colim_m"] \\
		(S \times V(\mf{m}^{d+1}))^{\wedge, (f)} \ar[r, "\colim_d"] & (S \times X^I)^{\wedge, (f) + \mf{m}} \ar[r] & (S \times X^I)^{\wedge, (f)} \ar[r, shift left = 1, dash] \ar[r, dash] &  (S \times X^I)^{\wedge}. 
	\end{cd}
	Each $\wedge$ denotes completion (i.e.\ colimit in the category of affine schemes) and is followed by the ideal with respect to which the completion is performed. The arrows labeled `colim' are members of a colimit diagram; for example, the upper horizontal arrow indicates that $(V(f^m) \times X^I)^{\wedge, \mf{m}} \simeq \colim_d V(f^m) \times V(\mf{m}^{d+1})$. Also, we have 
	\begin{align}
		(S \times V(\mf{m}^{d+1}))^{\wedge, (f)} &\simeq \wh{S} \times V(\mf{m}^{d+1}). \tag{P} \label{prod}  
	\end{align}
	Now, for any function $h \in \Gamma((S \times X^I)^\wedge, \oh)$, we obtain functions denoted as follows: 
	\begin{cd}
		(h)_{m, d}  & (h)_{m, \infty} \ar[l, mapsto] & (h)_{m} \ar[l, mapsto] \\
		(h)_{\infty, d} \ar[u, mapsto] & (h)_{\infty, \infty} \ar[l, mapsto] \ar[u, mapsto] & (h)_\infty \ar[r, shift left = 1, dash] \ar[r, dash] \ar[l, mapsto] \ar[u, mapsto] & h. 
	\end{cd}
	We will apply this to functions of the form $h = g_N\, a_I$. 
	
	In this paragraph, we fix an integer $N > 0$. Recall the functor 
	\[
		\mc{P}_{S, n, \le d} : \fset_{*} \to \on{AbGrp}
	\]
	defined in~\ref{rnd-cart}, for any base scheme $S$. The assignment 
	\[
		I \rightsquigarrow (g_N\, a_I)_{\infty, d}
	\]
	yields an element of $\lim_{|I| \le N} \mc{P}_{\wh{S}, n, \le d}$ where $n = \dim X$. (This uses (\ref{prod}).) Lemma~\ref{rnd-cart} implies that $\mc{P}_{\wh{S}, n, \le d}$ is polynomial of degree $\le d$, so Lemma~\ref{prop1-fin} can be applied when $d \le N-2$. Since we are dealing with \emph{rigidified} functions, i.e.\ the constant terms are zero, we conclude that $(g_N\, a_I)_{\infty, N-2} = 0$.
	
	Now fix an integer $N_0 > 0$, and let $N > N_0$ be any integer. The identities in this paragraph will be valid for $I$ satisfying $|I| \le N_0$. Recall that $g_{N} = r_{N_0, N} \, g_{N_0}$ for $r_{N_0, N} \in {fT}$. The previous paragraph implies that 
	\e{
		0 &= (g_N\, a_I)_{\infty, N-2}  \\
		&= r_{N_0, N} \, (g_{N_0}\, a_I)_{\infty, N-2}.
	} 
	Now, since $\wh{S} \times V(\mf{m}^{N-2})$ has the same associated points as $\wh{S}$,\footnote{Proof: the map $\wh{S} \times V(\mf{m}^{N-2}) \xrightarrow{\pr_1} \wh{S}$ is flat and induces a bijection on the underlying topological spaces, so Exercise 24.2.J in \cite{v} yields the result.} the isomorphism (\ref{prod}) implies that ${fT}$ maps to non-zerodivisors on $(S \times V(\mf{m}^{N-2}))^{\wedge, (f)}$. Therefore the previous equation implies that 
	\[
		(g_{N_0}\, a_I)_{\infty, N-2} = 0. 
	\]
	Keeping $N_0$ fixed and passing to the limit as $N \to \infty$, we conclude that 
	\[
		(g_{N_0}\, a_I)_{\infty, \infty} = 0. 
	\]
	This implies that 
	\[
		(g_{N_0}\, a_I)_{m, \infty} = 0. 
	\]
	This is a function on $(V(f^m) \times X^I)^{\wedge, \mf{m}}$, so we may apply Lemma~\ref{supp} and conclude that 
	\[
		(g_{N_0}\, a_I)_m = 0. 
	\]
	Taking the limit as $m \to \infty$ we find that $g_{N_0}\, a_I = 0$, from which it follows that $a_I = 0$, by definition of ring localization. 
	
	This conclusion applies to $|I| \le N_0$, but we are now free to take $N_0 \to \infty$, so $a_I = 0$ for all $I$, as desired. 
\end{proof}

\subsection{The Artinian case} \label{formal-sec} 

To prove the result in the case when $S$ is Artinian, we bootstrap from the result of Corollary~\ref{cor-pic-1} using deformation theory. 

\subsubsection{} \label{formal}
Here is the desired statement. The proof will rely on the next Lemma~\ref{ga-trivial}. 
\begin{lem} 
	If $S$ is an Artinian scheme over $k$, and $X$ satisfies the hypothesis (C), then $\Pic^e(S \times \Ran(X)) = 0$.  
\end{lem} 
\begin{proof} 
	We may assume that $S = \Spec R$ where $R$ is local and Artinian. Then $R$ has a unique prime ideal $\mf{m}$ with residue field $K := R / \mf{m}$ which is a field extension of $k$. Furthermore, as an $R$-module, $R$ has a filtration
	\[
		0 = I_N \subset I_{N-1} \subset \cdots \subset I_0 = R
	\]
	where $I_{n} / I_{n+1} \simeq K$ for each $n$. The image of $1 \in K$ under this (non-canonical) isomorphism gives an element $\ol{\epsilon}_n \in I_{n} / I_{n+1}$. For each $n$, we have a short exact sequence in $R$-mod: 
	\[
		0 \to K \xrightarrow{\ol{\epsilon}_n} R/I_{n+1} \to R/I_n \to 0. 
	\]
	If we define $S_n := R/I_n$, then we have a sequence of closed embeddings 
	\[
		\Spec K = S_1 \hra S_2 \hra \cdots \hra S_N = S.
	\]
	
	For any $k$-algebra $\wt{R}$, define the following sheaves of abelian groups on the big Zariski site of $\Spec k$: 
	\e{
		(\ga \otimes \wt{R})(T) &= \{\text{regular functions on } T \times \Spec \wt{R}\} \\
		(\gm \otimes \wt{R})(T) &= \{\text{invertible functions on } T \times  \Spec \wt{R}\}
	} 
	where $T$ is an arbitrary affine $k$-scheme. These are just the mapping prestacks $\ul{\Hom}(\Spec \wt{R}, \ga)$ and $\ul{\Hom}(\Spec \wt{R}, \gm)$, respectively. 
	
	For each $n\ge 1$, we have an exact sequence of Zariski sheaves of abelian groups: 
	\[
		0 \to \ga \otimes K \to \gm \otimes R/I_{n+1} \to \gm \otimes R/I_n \to 0.
	\]
	On an affine scheme $T = \Spec A$, the first horizontal map is defined by sending a function $a \in A\otimes_k K$ to the invertible function $1 + a\, \ol{\epsilon}_n \in A \otimes_k R/I_{n+1}$. 
	
	Now fix a rigidified line bundle $\mc{L}$ on $S \times \Ran(X)$. Its restriction to $S_1 \times \Ran(X)$ is trivial by Corollary~\ref{cor-pic-1}. Fix $n \ge 1$ and assume by induction that its restriction to $S_n \times \Ran(X)$ is trivial. We claim that its restriction to $S_{n+1} \times \Ran(X)$ must also be trivial. Indeed, a line bundle on $S_{n+1} \times \Ran(X)$, equipped with trivialization on $S_n \times \Ran(X)$, is equivalent to a torsor for $\gm \otimes R/I_{n+1}$ defined on $\Ran(X)$, equipped with trivialization of the induced $\gm \otimes R/I_n$ torsor. But this is equivalent to a $\ga \otimes K$ torsor on $\Ran(X)$, which is trivial by Lemma~\ref{ga-trivial}. This completes the proof of the inductive step. 	
\end{proof} 

\subsubsection{} \label{ga-trivial} Let $K$ be an arbitrary field extension of $k$, and recall from~\ref{formal} the definition of the big Zariski sheaf $\ga \otimes K$ valued in abelian groups. 
\begin{lem} 
	Any $\ga \otimes K$ torsor on $\Ran(X)$ is trivial. 
\end{lem} 
\begin{proof} 
	Writing $X_K := \Spec K \times_{\Spec k} X$, we have a Cartesian diagram: 
	\begin{cd} 
		\Ran_K(X_K) \ar{r}{q'} \ar{d}[swap]{p'} & \Ran(X) \ar{d}{p} \\
		\Spec K \ar{r}{q} & \Spec k
	\end{cd} 
	(Here, as in Corollary~\ref{cor-pic-1}, the notation $\Ran_K(X_K)$ means the Ran space construction taken over $K$ rather than $k$.) The restriction of $\ga \otimes K$ (as a sheaf on the big Zariski site of $\Spec k$) to $\Ran(X)$ is $q'_* \oh_{\Ran_K(X_K)}$, so isomorphism classes of torsors for this sheaf are in bijection with the cohomology group $H^1(\Ran(X), q'_* \oh_{\Ran_K(X_K)})$. 
	
	Since $\oh_{\Ran_K(X_K)}$ is quasicoherent, and $q_*'$ is affine and hence acyclic for quasicoherent sheaves, we have 
	\[
		H^1(\Ran(X), q'_*\, \oh_{\Ran_K(X_K)}) \simeq H^1(\Ran_K(X_K), \oh_{\Ran_K(X_K)})
	\]
	as $k$-vector spaces. The argument of \cite[Sect.\ 6]{g} shows that 
	\[
	H^i(\Ran_K(X_K), \oh_{\Ran_K(X_K)}) = 0 
	\]for all $i > 0$, as desired. This argument applies because there is a K\"unneth theorem for coherent cohomology~\cite[\href{https://stacks.math.columbia.edu/tag/0BEC}{Tag 0BEC}]{stacks} and because we already know that $H^0(\Ran_K(X_K), \oh) = K$, by Lemma~\ref{func1}. 	
\end{proof}

\subsection{The case of affine $X$} \label{induct} \label{final}

Starting from Lemma~\ref{formal}, which establishes the desired result in the case when $S$ is Artinian, we extend the resulting trivializations of line bundles to the completed version of $S \times \Ran(X)$ introduced in~\ref{punctured}. We need to restrict to affine $X$ because that is the generality in which this construction makes sense. Next, we glue these trivializations across punctured completions. In order to do so, we will use some lemmas in commutative algebras which are gathered in~\ref{bl-sec}. 

\begin{prop} 
	Assume that $X$ is affine and satisfies (C). Then, for any locally Noetherian scheme $S$ over $k$, we have $\Pic^e(S \times \Ran(X)) = 0$. 
\end{prop} 
\begin{proof} 
	We know from~\ref{sheaf} that the assignment $U \mapsto \Pic^e(U \times \Ran(X))$ defined on open subschemes of $S$ is a sheaf valued in abelian groups. Therefore, it suffices to prove the result when $S$ is affine, so we may assume $S = \Spec A$. Since the result when $S$ is a reduced point follows from Corollary~\ref{cor-pic-1}, we may assume by Noetherian induction that the result holds for all strictly smaller closed subschemes of $S$. 
	
	Fix a rigidified line bundle $\mc{L}$ on $S \times \Ran(X)$. 
	
	For notational ease, we replace $\fsetsurjne$ by its skeletal subcategory $\mc{C}$ consisting of the objects $[n]$ for $n \ge 1$. 	
	
	Let $T \subset A$ be the multiplicative system consisting of elements of $A$ not contained in any minimal prime ideal. Then $S_T := \Spec A_T$ is an Artinian scheme, so Lemma~\ref{formal} implies that the restriction of $\mc{L}$ to $S_T \times \Ran(X)$ is trivial. Concretely, this trivialization provides sections $\mathring{a}^{(n)} \in \Gamma(S_T \times X^n, \mc{L})$ which are compatible under the maps corresponding to surjections of finite sets. 
	
	\begin{claim}
		For each $n$, there exists $f_n \in T$ such that $\mathring{a}^{(n)}$ is the restriction of a nonvanishing section $a^{(n)} \in \Gamma(S_{f_n} \times X^n, \mc{L})$, where $S_{f_n} := \Spec A_{f_n}$. 
	\end{claim}
	\begin{proof}
		Fix local trivializations of $\mc{L}|_{S \times X^n}$ with respect to some affine open cover $\{U_\beta\}_{\beta \in B}$ of $S \times X^n$, and write $U_\beta = \Spec A_\beta$ where each $A_\beta$ is an $A$-algebra. Localizing by $T$, this yields local trivializations of $\mc{L}|_{S_T \times X^n}$ with respect to the affine open cover of $S_T \times X^n$ consisting of $(U_\beta)_T := \Spec ((A_\beta)_T)$. The nonvanishing section $\mathring{a}^{(n)}$ can be interpreted as consisting of the following data: 
		\begin{enumerate}
			\item[] There are functions $g_\beta, h_\beta \in \Gamma((U_\beta)_T, \oh)$ for each $\beta \in B$, satisfying two constraints:  
			\item[(F1)] For each $\beta \in B$, we have $g_\beta \cdot h_\beta = 1$.
			\item[(F2)] For each pair $\beta_1, \beta_2 \in B$, the product $g_{\beta_1} \cdot h_{\beta_2}$ on the overlap $(U_{\beta_1})_T \cap (U_{\beta_2})_T$ is equal to the corresponding transition function of $\mc{L}|_{S_T \times X^n}$. 
		\end{enumerate}
		For each $\beta \in B$, we have $g_\beta = t^{-1}\, \wt{g}_\beta$ and $h_{\beta} = s^{-1} \, \wt{h}_{\beta}$ for some functions $\wt{g}_\beta, \wt{h}_\beta$ on $U_\beta$ and some elements $s, t \in T$. If we take $f'_n \in T$ to be the product of all such $s$ and $t$ which arise, then for each $\beta \in B$ the functions $g_{\beta}$ and $h_{\beta}$ are restrictions of functions $g'_{\beta}$ and $h'_{\beta}$ defined on the localizations $(U_\beta)_{f'_n}$. Now, each equation $g'_\beta \cdot h'_\beta = 1$ becomes true after further localizing by some element $u\in T$, and similarly for the equations corresponding to (F2). (Note that the transition function in question is defined on $U_{\beta_1} \cap U_{\beta_2}$, even before localizing by $T$.) Let $f_n$ be the product of $f'_n$ with all the $u$'s which arise. Then the restrictions of $g'_\beta$ and $h'_\beta$ to $(U_\beta)_{f_n}$ for each $\beta \in B$ yield a nonvanishing section of $\mc{L}|_{S_{f_n} \times X^n}$, as desired. 
	\end{proof}
	By multiplying $f_n$ by an additional element of $T$, we may assume that $f_n$ vanishes on all the embedded points of $S$. Hence, its pullback to $S_{f_n} \times X^n$ vanishes on all the embedded points of $S_{f_n} \times X^n$ (see Lemma~\ref{basic}). 
	
	For any $m \ge 1$, the inductive hypothesis gives a trivialization of $\mc{L}$ on $V(f_n^m) \times \Ran(X)$, where $V(f_n^m)$ is the vanishing locus of $(f_n)^m$ on $S$. The uniqueness of trivializations proved in Corollary~\ref{sheaf} implies that these trivializations are compatible under pullback along the maps $V(f_n^m) \hra V(f_n^{m+1})$. In this way, we obtain nonvanishing sections of $\mc{L}$ on the formal neighborhoods of $V(f_n) \subset S \times X^n$ for various $n$. Grothendieck's existence theorem turns these formal sections into nonvanishing sections of $\mc{L}$ on the completions $(S\times X^n)^\wedge$.\footnote{The scheme $(S \times X^n)^\wedge$ is defined as in~\ref{punctured}, completing with respect to $f_n$ in place of $f$.} In this way, we obtain a trivialization of $\mc{L}$ on 
	\[
		\mc{Y}' := \colim_{\mc{C}^\op}(S \times X^n)^\wedge. 
	\]
	Here $\mc{L}|_{\mc{Y}'}$ is taken to be a rigidified line bundle in the sense of~\ref{rigid-variation}. Let $b^{(n)} \in \Gamma((S \times X^n)^\wedge, \mc{L})$ be the section given by this trivialization. 
	
	At this point, we have sections 
	\e{
		& a^{(n)} \in \Gamma(S_{f_n} \times X^n, \mc{L}) \\
		& b^{(n)} \in \Gamma((S \times X^n)^\wedge, \mc{L}). 
	} 
	In what follows, subscripts on $(S\times X^n)^\wedge$ indicate localization. 
	\begin{claim} 
		The pullbacks of $a^{(n)}$ and $b^{(n)}$ to $(S \times X^n)^\wedge_{f_n}$ agree. 
	\end{claim} 
	\begin{proof} 
		We have a commutative diagram\footnote{A comment about notation: since we are in the special case of~\ref{punctured} when $f_n = f \in T$, we can replace the subscripts $fT$ by $T$, and we have done so above.} 
		\begin{cd} 
			(S \times X^n)^\wedge_T \ar{r} \ar{d} & (S \times X^n)^\wedge_{f_n} \ar{d} \ar{r} & (S \times X^n)^\wedge \ar[d] \\
			S_T \times X^n \ar{r} & S_{f_n} \times X^n \ar[r] & S \times X^n
		\end{cd} 
		Since $f_n$ vanishes on all the embedded points of $S$, every element of $T$ restricts to a non-zerodivisor on $S_{f_n}$. Note that the maps 
		\[
			(S \times X^n)^\wedge_{f_n} \to (S \times X^n)_{f_n} \to S_{f_n}
		\]
		are flat, because completion is flat under Noetherian hypotheses~\cite[Thm.\ 29.2.6]{v} and localization preserves flatness. It follows that every element of $T$ restricts to a non-zerodivisor on $(S \times X^n)^\wedge_{f_n}$ as well. Hence, a nonzero function on $(S \times X^n)^\wedge_{f_n}$ pulls back to a nonzero function on $(S \times X^n)^\wedge_T$. It now suffices to check that $a^{(n)}$ and $b^{(n)}$ agree on $(S \times X^n)^\wedge_T$. 
		
		The pullback of $a^{(n)}$ to $(S \times X^n)^\wedge_T$ is also the pullback of $\mathring{a}^{(n)}$ to $(S \times X^n)^\wedge_T$. 	
		This shows that $a^{(n)}$ is part of a trivialization of the restriction of $\mc{L}$ to $\mc{Y}$, where 
		\[
			\mc{Y} := \colim_{\mc{C}^\op}(S \times X^n)^\wedge_{T}  
		\]
		was defined in~\ref{punctured}, and $\mc{L}|_{\mc{Y}}$ is viewed as a rigidified line bundle in the sense of~\ref{rigid-variation}. On the other hand, $b^{(n)}$ is already part of a trivialization of $\mc{L}|_{\mc{Y}'}$ from before, and this pulls back to a trivialization of $\mc{L}_{\mc{Y}}$. Now Lemma~\ref{func2}  implies that these two trivializations coincide, thereby showing that the restrictions of $a^{(n)}$ and $b^{(n)}$ to $(S \times X^n)^\wedge_{T}$ are equal. 	
	\end{proof} 
	
	Proposition~\ref{gluing} gives a section 
	\[
		c^{(n)} \in \Gamma(S \times X^n, \mc{L})
	\]
	which restricts to $a^{(n)}$ and $b^{(n)}$. This section is nonvanishing since $a^{(n)}$ and $b^{(n)}$ are nonvanishing. We wish to show that these $c^{(n)}$ determine a trivialization of $\mc{L}$ on $S \times \Ran(X)$. 
	
	\begin{claim} 
		The section $c^{(n)}$ does not depend on the choice of $f_n$. 
	\end{claim} 
	\begin{proof} 
		Let $c'^{(n)}$ correspond to the choice of another $f_n'$, and let $c''^{(n)}$ correspond to $f_nf_n'$. It suffices to show that $c^{(n)} = c''^{(n)}$, because the same argument could be repeated to show $c'^{(n)} = c''^{(n)}$. 
		
		Since $c^{(n)}$ and $c''^{(n)}$ both restrict to $\mathring{a}^{(n)}$, their difference restricts to zero on $S_T \times X^n$. This implies that the difference is supported on the closure of the union of the embedded points of $S \times X^n$, which is the same as the preimage of the analogous locus in $S$. In particular, the difference is annihilated by a sufficiently high power of $f_n$, since $f_n$ vanishes on the embedded points of $S$. 
		
		By construction, the restriction of $c^{(n)}$ to $(S \times X^n)^{\wedge, f_n}$ is $b^{(n)}$ which is part of a trivialization of $\mc{L}$ on $\mc{Y}'$ as defined before.\footnote{The additional superscript indicates completion with respect to $(f_n)$.} Similarly, the restriction of $c''^{(n)}$ to $(S \times X^n)^{\wedge, f_nf_n'}$ is some section $b''^{(n)}$ which is part of a trivialization of $\mc{L}$ on 
		\[
			\mc{Y}'' := \colim_{\mc{C}^\op} (S \times X^n)^{\wedge, f_nf_n'}, 
		\]
		as a rigidified line bundle. We have a natural map $\mc{Y}' \to \mc{Y}''$, and the uniqueness of trivializations which follows from~\ref{func2} implies that the restriction of $b''^{(n)}$ to $(S \times X^n)^{\wedge, f_n}$ coincides with $b^{(n)}$. Therefore the difference $c^{(n)} - c''^{(n)}$ restricts to zero on $(S \times X^n)^{\wedge, f_n}$. Since this difference lies in $\Ann(f_n)$, Corollary~\ref{ann} implies that it equals zero. 
	\end{proof} 
	
	Now, consider an arbitrary surjective map $\xi_{21} : [n_1] \sra [n_2]$ which gives rise to the diagonal map $\Delta_{\xi_{21}} : S \times X^{n_2} \to S \times X^{n_1}$. By our construction, there are functions $f_{n_1}, f_{n_2} \in T$ which give sections $c^{(n_1)}$ and $c^{(n_2)}$. By the previous claim, we can replace both $f_{n_1}$ and $f_{n_2}$ by $f_{n_1}f_{n_2}$ without changing $c^{(n_1)}$ and $c^{(n_2)}$. Therefore we may write $f:= f_{n_1} = f_{n_2}$. Now $c^{(n_2)}$ and $(\Delta_{\xi_{21}})^*c^{(n_1)}$ are two sections of $\mc{L}|_{S \times X^{n_2}}$ which coincide when restricted to $S_T \times X^{n_2}$ and $(S \times X^{n_2})^{\wedge, f}$. The same argument as was used to prove the claim shows that $c^{(n_2)} = (\Delta_{\xi_{21}})^*c^{(n_1)}$, as desired. 	
\end{proof} 

\subsection{Beauville--Laszlo gluing for regular functions} \label{bl-sec} 
\subsubsection{Setup} \label{setup} Let $A$ be a Noetherian ring and let $f \in A$ be an arbitrary element. Let $\wh{A}$ be the completion of $A$ along the ideal $(f)$, and let $\wh{A}_f := A_f \otimes_A \wt{A}$ be the localization of $\wh{A}$ by the image of $f$. We have an obvious commuting diagram: 
\begin{cd} 
	A \ar{r} \ar{d} & A_f \ar{d} \\
	\wh{A} \ar{r} & \wh{A}_f
\end{cd} 

\subsubsection{} \label{bl-iso} 
We shall use the following elementary result: 
\begin{lem} 
	~\\ \vspace{-0.25cm} 
	\begin{enumerate}[label=(\roman*)]
		\item Let $M$ be an $A$-module which is set-theoretically supported on the vanishing locus of $f$. Then the natural map $M \to \wh{A} \otimes_A M$ is an isomorphism. 
		\item The natural map $A_f / A \to \wh{A} \otimes_A (A_f / A)$ is an isomorphism. 
	\end{enumerate} 
\end{lem} 
\begin{proof} 
	Point (i) appears as Lemme 1 in \cite{bl}, and (ii) follows from (i). 
\end{proof} 

\subsubsection{} \label{ann} 
Let $\Ann_A(f) \subset A$ denote the ideal consisting of elements which are annihilated by some power of $f$, and define $\Ann_{\wh{A}}(f)$ similarly. 
\begin{cor} 
	The natural map $A \to \wh{A}$ induces an isomorphism $\Ann_A(f) \simeq \Ann_{\wh{A}}(f)$. 
\end{cor} 
\begin{proof} 
	Because $\Ann_A(f)$ is finitely generated (this uses that $A$ is Noetherian), we have the following exact sequence for all sufficiently large $n$: 
	\begin{cd} 
		0 \ar[r] & \Ann_A(f) \ar[r] & A \ar[r, "f^n"] & A 
	\end{cd} 
	Since $\wh{A}$ is flat over $A$ by \cite[Thm.\ 29.2.6(a)]{v}, we obtain another exact sequence 
	\begin{cd} 
		0 \ar[r] & \wh{A} \otimes_A \Ann_A(f) \ar[r] & \wh{A} \ar[r, "f^n"] & \wh{A} 
	\end{cd} 
	Since this holds for arbitrarily large $n$, we conclude that $\wh{A} \otimes_A \Ann_A(f) \simeq \Ann_{\wh{A}}(f)$. But $\Ann_A(f)$ is supported on the vanishing locus of $f$, so Lemma~\ref{bl-iso} implies that $\wh{A} \otimes_A \Ann_A(f) \simeq \Ann_A(f)$. 
\end{proof} 

\subsubsection{Gluing} \label{gluing} 
This gluing result for functions is used in the proof of Proposition~\ref{induct}. 
\begin{prop} 
	Let $a \in A_f$ and $b \in \wh{A}$ be such that their images in $\wh{A}_f$ agree. Then there exists a unique $c \in A$ which simultaneously maps to $a$ and $b$. 
\end{prop}  
\begin{proof} 
	Since tensors commute with quotients, Lemma~\ref{bl-iso}(ii) says that $A_f / A \simeq \wh{A}_f / \wh{A}$. By hypothesis, the image of $a$ in $\wh{A}_f$ lies in $\wh{A}$. Therefore the image of $a$ in $A_f / A$ vanishes, so $a$ can be lifted to $\wt{a} \in A$. Subtracting off $\wt{a}$, we may assume that $a = 0$. 
	
	With the assumption that $a = 0$, we have that $b$ lies in the kernel of $\wh{A} \to \wh{A}_f$. The isomorphism of Corollary~\ref{ann} implies that $b \in \wh{A}$ is the image of a \emph{unique} $c \in A$ which is also annihilated by $f$. But then $c$ also maps to zero in $A_f$, so it satisfies the requirement of mapping to $a = 0$ and to $b$. 
\end{proof}

\subsection{The general case} \label{ssec-general}

We now complete the proof of Theorem~\ref{main2} by explaining how to remove the hypothesis that $X$ is affine from the statement of Proposition~\ref{final}. 

\subsubsection{} 
Assume that $Y$ satisfies (C) and has a basepoint $y_0 \in Y(k)$. Let $S$ be a locally Noetherian $k$-scheme, and take a rigidified line bundle $\mc{L} \in \Pic^e(S \times \Ran(Y))$.

\begin{lem}\label{affine-y}
	 There exists a unique trivialization $r \in \mc{L}|_{S \times Y}$ such that, for any affine open $U \subset Y$ containing $y_0$, the restriction of $r$ to $S \times U$ equals the $(S \times U)$-part of the unique trivialization of $\mc{L}|_{S \times \Ran(U)} \in \Pic^e(S \times \Ran(U))$ guaranteed by Proposition~\ref{final}. 
\end{lem}
\begin{proof}
	Since $Y$ is covered by the affine open subschemes $U \subset Y$ containing $y_0$, the uniqueness of $r$ is automatic. To show existence, we need only show the following statement: 
	\begin{itemize}
		\item For every such $U$, let $r_U \in \mc{L}|_{S \times U}$ be the $S \times U$-part of the unique trivialization of $\mc{L}|_{S \times \Ran(U)}$. Then, for any two such $U_1, U_2$ with $U_1 \subset U_2$, we have $r_{U_2}|_{S \times U_1} = r_{U_1}$. 
	\end{itemize}
	Then the Zariski sheaf condition on $\mc{L}|_{S \times Y}$ shows that the $r_U$ glue into a nonvanishing section $r$. 
	
	The bullet point statement follows from the uniqueness of trivialization of a rigidified line bundle on $S \times \Ran(U_1)$ (see Corollary~\ref{sheaf}). Indeed, the trivialization of $\mc{L}|_{S \times \Ran(U_2)}$ pulls back to a trivialization of $\mc{L}|_{S \times \Ran(U_1)}$ under the map $S \times \Ran(U_1) \to S \times \Ran(U_2)$, and looking at sections on $S \times U_1$ yields the claimed equality. 
\end{proof}

\subsubsection{} \label{last-triv} 
Assume that $X$ satisfies (C) and has a basepoint $x_0 \in X(k)$. Let $S$ be a locally Noetherian $k$-scheme, and take a rigidified line bundle $\mc{L} \in \Pic^e(S \times \Ran(X))$. For each $I$, we have a functor $\fsetsurjne \to \fsetsurjne$ given by $I \times (-)$, and the colimit version of Lemma~\ref{lims} yields a map $\Ran(X^I) \xrightarrow{\phi^{(I)}} \Ran(X)$ which is functorial in $I$. (This construction was already used in~\ref{ssec-canon-proof}.) We observe here that there is a commutative diagram 
\begin{cd}
	X^I \ar[r] \ar[rd] & \Ran(X^I) \ar[d, "\phi^{(I)}"] \\
	& \Ran(X)
\end{cd}
which is also functorial in $I$. 

Let $\mc{L}^{(I)} \in \Pic^e(S \times \Ran(X^I))$ be the pullback of $\mc{L}$ under the resulting map $S \times \Ran(X^I) \to S \times \Ran(X)$, and let $r^{(I)} \in \mc{L}^{(I)}|_{S \times X^I} \simeq \mc{L}|_{S \times X^I}$ be the trivialization guaranteed by Lemma~\ref{affine-y}, where we choose the basepoint $x_0^I \in X^I(k)$. 

\begin{lem}
	 The trivializations $r^{(I)}$ are compatible under surjections $\xi : I \sra J$ in the sense that $(\id_S \times \Delta_{\xi})^*r^{(I)} = r^{(J)}$. Therefore, the $r^{(I)}$ collectively yield a trivialization of $\mc{L}$. 
\end{lem}
\begin{proof}
	Given a surjection $\xi : I \sra J$, the claimed compatibility involves the diagram 
	\begin{cd}
		X^J \ar[dd, swap, "\Delta_{\xi}"] \ar[rd] \\
		& \Ran(X) \\
		X^I \ar[ru] 
	\end{cd}
	which we factor as follows: 
	\begin{cd}
		X^J \ar[dd, swap, "\Delta_{\xi}"] \ar[r] & \Ran(X^J)\ar[dd] \ar[rd, "\phi^{(J)}"] \\
		& & \Ran(X) \\
		X^I \ar[r] & \Ran(X^I) \ar[ru, swap, "\phi^{(I)}"] 
	\end{cd}
	By definition, $r^{(I)}$ is characterized by the property that $r^{(I)}|_{S \times U}$ is part of the unique trivialization of $\mc{L}^{(I)}|_{S \times \Ran(U)}$, for any open $U \subset X^I$ containing the basepoint $x_0^I$. For a fixed $U$, let $V \subset X^J$ be an affine open subscheme containing $x_0^J$ such that $\Delta_\xi(V) \subset U$. By the above commutative diagram, the pullback $(\id_S \times \Delta_{\xi})^*r^{(I)}$ restricts to the $(S \times V)$-part of the unique trivialization of $\mc{L}^{(J)}|_{S \times \Ran(V)}$. Therefore, we have 
	\[
		(\id_S \times \Delta_{\xi})^*r^{(I)}|_{S \times V} = r^{(J)}|_{S \times V}
	\]
	by the defining property of $r^{(J)}$. As $U$ and $V$ vary, the resulting $V$'s cover $X^J$, and the lemma follows. 
\end{proof}

\subsubsection{Proof of Theorem~\ref{main2}}

Assume that $X$ satisfies (C) and let $S$ be a locally Noetherian $k$-scheme. By~\ref{based-based}, we may assume that $X(k)$ is nonempty. Now, Lemma~\ref{last-triv} implies that every $\mc{L} \in \Pic^e(S \times \Ran(X))$ is trivial, i.e. $\Pic^e(S \times \Ran(X)) = 0$, as desired.


\begin{thebibliography}{99} 
	
	\bibitem[BD]{bd} A.\ Beilinson and V.\ Drinfeld, \emph{Chiral algebras}, American Mathematical Society Colloquium Publications~\textbf{51} (2004). 
	
	\bibitem[BK]{bk} A.\ K.\ Bousfield and D.\ M.\ Kan, \emph{Homotopy limits, completions, and localizations}, Lecture Notes in Mathematics~\textbf{304} (1972). 
	
	\bibitem[BL]{bl} A.\ Beauville and Y.\ Laszlo, \emph{Un lemme de descente}, C.\ R.\ Acad.\ Sci.\ Paris S\'er.\ I Math.\ \textbf{320} (1995), no.\ 3, 335--340. 
	
	\bibitem[B]{berger} C.\ Berger, \emph{Goodwillie calculus for gamma-spaces} (2012), slides available at \url{https://math.unice.fr/~cberger/GammaGood.pdf}. 
	
	\bibitem[C]{cast} P.\ A.\ Castillejo, \emph{The theorem of the cube} (2014), available at \url{http://page.mi.fu-berlin.de/castillejo/docs/140527_Theorem_of_the_cube.pdf}. 
	
	\bibitem[CG]{cg} N.\ Chriss and V.\ Ginzburg, \emph{Complex geometry and representation theory}, Birkh\"auser, Modern Birkh\"auser Classics (2010), reprint of 1997 edition. 
	
	\bibitem[D]{sga43} P.\ Deligne, ``La formule de dualite globale,'' Expos\'e xviii in SGA 4, tome 3, available at \url{http://www.cmls.polytechnique.fr/perso/laszlo/sga4/SGA4-3/sga43.pdf}. 
	
	\bibitem[EGA]{ega} A.\ Grothendieck, with J.\ Dieudonn\'e, \emph{\'Elements de g\'eom\'etrie alg\'ebrique: I.\ Le langage des sch\'emas}, Publ.\ Math.\ IH\'ES \textbf{4} (1960), pp.\ 5--228. 
	
	\bibitem[FG]{fg} J.\ Francis and D.\ Gaitsgory, \emph{Chiral Koszul duality}, Sel.\ Math.\ New Ser.\ \textbf{18} (2012), no.\ 1, pp.\ 27--87. 
	
	\bibitem[G]{g} D.\ Gaitsgory, \emph{Contractibility of the space of rational maps}, Invent.\ Math.\ \textbf{191} (2013), no.\ 1, pp.\ 91--196. 
	
	\bibitem[GR]{crys} D.\ Gaitsgory and N.\ Rozenblyum, \emph{Crystals and $\D$-modules}, Pure Appl.\ Math.\ Q.\ \textbf{10} (2014), no.\ 1, pp.\ 57--154. 
	
	\bibitem[GL]{gl} D.\ Gaitsgory and J.\ Lurie, \emph{Weil's conjecture for function fields}, available at~\url{http://www.math.harvard.edu/~lurie/papers/tamagawa.pdf}. 
	
	\bibitem[I]{fga} L.\ Illusie, ``Grothendieck's existence theorem in formal geometry,'' in \emph{Fundamental Algebraic Geometry: Grothendieck's FGA explained}, Mathematical Surveys and Monographs~\textbf{123} (2005). 
	
	\bibitem[K]{kivimae} Pax Kivimae (\url{https://mathoverflow.net/users/41103/pax-kivimae}), \emph{What is the theorem of the cube?} (2016), available at \url{https://mathoverflow.net/q/241891}. 
	
	\bibitem[L1]{htt} J.\ Lurie, \emph{Higher topos theory}, available at~\url{http://www.math.harvard.edu/~lurie/papers/HTT.pdf}. 
	
	\bibitem[L2]{ha} J.\ Lurie, \emph{Higher algebra}, available at~\url{http://www.math.harvard.edu/~lurie/papers/HA.pdf}. 
	
	\bibitem[Stacks]{stacks} The {Stacks project authors}, \emph{The Stacks project}, available at~\url{https://stacks.math.columbia.edu}. 
	
	\bibitem[TZ]{tz} J.\ Tao and Y.\ Zhao, \emph{Extensions by $\mathbf{K}_2$ and factorization line bundles}, available at \href{https://arxiv.org/abs/1901.08760}{arXiv:1901.08760}. 
	
	\bibitem[V]{v} R.\ Vakil, \emph{The rising sea: foundations of algebraic geometry}, draft from 12/29/2015, available at \href{math216.wordpress.com}{math216.wordpress.com}. 
	
	\bibitem[W]{w} W.\ Chach\'olski and J.\ Scherer, \emph{Homotopy theory of diagrams}, Memoirs of the American Mathematical Society~\textbf{736} (2002). 
	
	\bibitem[Z]{z} X.\ Zhu, \emph{An introduction to affine Grassmannians and the geometric Satake equivalence}, notes from 2015 PCMI summer school, available at \href{https://arxiv.org/abs/1603.05593}{arXiv:1603.05593}. 
	
\end{thebibliography}
\end{document}